\documentclass[11pt]{amsart}  
\usepackage{amscd,graphicx,epsfig}
\usepackage{amssymb}

\usepackage[hypcap=false]{caption}
\usepackage[left=3cm,right=3cm,top=3cm,bottom=3.5cm]{geometry}

\usepackage[colorlinks=true, linktocpage,  linkcolor=blue, citecolor=cyan, urlcolor=blue]{hyperref}

\usepackage{tikz}
\usetikzlibrary{arrows,decorations.markings}

\usetikzlibrary{matrix}
\usetikzlibrary{graphs}

\title[Small  $G$-varieties]{Small $G$-varieties}

\author{Hanspeter Kraft, Andriy Regeta, and Susanna Zimmermann}
\date{\today}
\thanks{During this work, the second and third authors were supported by the Swiss National Science Foundation. The third author received funding from Projet PEPS ``jc/jc" 2018--2019, the ANR Project FIBALGA ANR-18-CE40-0003-01 and the Project \'Etoiles Montantes of the R\'egion Pays de la Loire}

\address{\noindent Departement Mathematik und Informatik, Universit\"at Basel\newline\indent
Spiegelgasse 1, CH-4051 Basel}
\email{hanspeter.kraft@unibas.ch}

\address{\noindent Institut f\"{u}r Mathematik, Friedrich-Schiller-Universit\"{a}t Jena, \newline
\indent  Jena 07737, Germany}
\email{andriyregeta@gmail.com}

\address{\noindent Laboratoire angevin de recherche en math\'ematiques (LAREMA), CNRS,
\newline\indent2 Bd Lavoisier, 49045 Angers cedex 01}
\email{susanna.zimmermann@univ-angers.fr}

\newtheorem{thm}{Theorem}[section]
\newtheorem*{thm*}{Theorem}

\newtheorem*{conj*}{Conjecture}
\newtheorem*{mthm*}{Main Theorem}
\newtheorem*{mlem*}{Main Lemma}

\newtheorem{prop}[thm]{Proposition}
\newtheorem*{prop*}{Proposition}

\newtheorem{lem}[thm]{Lemma}
\newtheorem*{lem*}{Lemma}
\newtheorem{cor}[thm]{Corollary}
\newtheorem*{cor*}{Corollary}

\theoremstyle{definition}
\newtheorem{defn}[thm]{Definition}
\newtheorem{exa}[thm]{Example}
\newtheorem{set-up}[thm]{Set-up}

\newtheorem{rem}[thm]{Remark}

\DeclareMathOperator{\End}{End}
\DeclareMathOperator{\id}{id}
\DeclareMathOperator{\SL}{SL}
\DeclareMathOperator{\PSL}{PSL}
\DeclareMathOperator{\Sp}{Sp}
\DeclareMathOperator{\Aut}{Aut}
\DeclareMathOperator{\Hom}{Hom}
\DeclareMathOperator{\SO}{SO}

\DeclareMathOperator{\GL}{GL}
\DeclareMathOperator{\Lie}{Lie}
\DeclareMathOperator{\Norm}{Norm}

\DeclareMathOperator{\gl}{\mathfrak{gl}}
\DeclareMathOperator{\codim}{codim}

\DeclareMathOperator{\Cent}{Cent}
\DeclareMathOperator{\pr}{pr}
\DeclareMathOperator{\gr}{gr}
\DeclareMathOperator{\rk}{rank}

\newcommand{\name}[1]{\textsc{#1\/}}
\newcommand{\NN}{{\mathbb N}}
\newcommand{\ZZ}{{\mathbb Z}}
\newcommand{\PP}{{\mathbb P}}
\newcommand{\RR}{{\mathbb R}}
\newcommand{\QQ}{{\mathbb Q}}

\newcommand{\NNN}{\mathcal N}
\newcommand{\VVV}{\mathcal V}
\newcommand{\OOO}{\mathcal O}
\newcommand{\CCC}{\mathcal C}
\newcommand{\PPP}{\mathcal P}
\newcommand{\simto}{\xrightarrow{\sim}}
\newcommand{\simot}{\xleftarrow{\sim}}
\newcommand{\be}{\begin{enumerate}}
\newcommand{\ee}{\end{enumerate}}
\newcommand{\eps}{\varepsilon}
\renewcommand{\phi}{\varphi}

\newcommand{\quot}{/\!\!/}
\newcommand{\K}{\mathbb{K}}
\newcommand{\Kst}{{\K^{*}}}
\newcommand{\Kn}{{\K^{n}}}
\newcommand{\Ktt}{\K[t_{1},\ldots,t_{r}]}
\newcommand{\Kt}{\K[\mathbf{t}]}

\newcommand{\onto}{\twoheadrightarrow}
\newcommand{\into}{\hookrightarrow}
\newcommand{\ps}{\par\smallskip}
\newcommand{\pmed}{\par\medskip}

\renewcommand{\gg}{\mathfrak g}

\newcommand{\bb}{\mathfrak b}
\newcommand{\hh}{\mathfrak h}
\renewcommand{\ll}{\mathfrak l}
\newcommand{\uu}{\mathfrak u}
\newcommand{\pp}{\mathfrak p}
\newcommand{\nn}{\mathfrak n}
\newcommand{\zz}{\mathfrak z}
\newcommand{\dd}{\mathfrak d}
\newcommand{\sP}{\mathfrak{sp}}
\newcommand{\sO}{\mathfrak{so}}
\newcommand{\so}{\mathfrak{so}}
\newcommand{\sL}{\mathfrak{sl}}
\renewcommand{\ss}{\mathfrak s}
\newcommand{\mm}{\mathfrak{m}}

\newcommand{\Tone}{{\mathsf{T}_{1}}}

\newcommand{\A}{\mathsf{A}}
\newcommand{\An}{{\mathsf{A}_n}}
\newcommand{\Aone}{{\mathsf{A}_{1}}}
\newcommand{\Atwo}{{\mathsf{A}_{2}}}
\newcommand{\Athree}{{\mathsf{A}_{3}}}
\newcommand{\Bn}{{\mathsf{B}_{n}}}
\newcommand{\Btwo}{{\mathsf{B}_{2}}}
\newcommand{\Ctwo}{{\mathsf{C}_{2}}}

\newcommand{\B}{\mathsf{B}}
\newcommand{\Cn}{{\mathsf{C}_{n}}}
\newcommand{\C}{{\mathsf{C}}}
\newcommand{\D}{\mathsf{D}}
\newcommand{\Dn}{{\mathsf{D}_{n}}}
\newcommand{\Dfour}{{\mathsf{D}_{4}}}
\newcommand{\Dfive}{{\mathsf{D}_{5}}}
\newcommand{\Dsix}{{\mathsf{D}_{6}}}
\newcommand{\Dseven}{{\mathsf{D}_{7}}}
\newcommand{\Gtwo}{{\mathsf{G}_{2}}}
\newcommand{\Ffour}{{\mathsf{F}_{4}}}
\newcommand{\Esix}{{\mathsf{E}_{6}}}
\newcommand{\Eseven}{{\mathsf{E}_{7}}}
\newcommand{\Eeight}{{\mathsf{E}_{8}}}

\newcommand{\SLtwo}{\SL_{2}}
\newcommand{\SLn}{\SL_{n}}

\newcommand{\Ob}{\overline{O}}
\newcommand{\Olb}{\overline{O_{\lambda}}}
\newcommand{\Olam}{O_{\lambda}}
\newcommand{\Olamb}{\overline{\Olam}}
\newcommand{\Obn}{\overline{O}^{\,n}}
\newcommand{\Olbn}{\overline{O_{\lambda}}^{\,n}}

\def\DynkinNodeSize{1mm}
\def\DynkinArrowLength{2mm}
\tikzset{
  dnode/.style={
    circle,
    inner sep=0pt,
    minimum size=\DynkinNodeSize,
    fill=black,
    draw},
  middlearrow/.style={
    decoration={markings,
      mark=at position 0.6 with
      {\draw (0:0mm) -- +(+135:\DynkinArrowLength); \draw (0:0mm) -- +(-135:\DynkinArrowLength);},
    },
    postaction={decorate}
  },
  leftrightarrow/.style={
    decoration={markings,
      mark=at position 0.999 with
      {
      \draw (0:0mm) -- +(+135:\DynkinArrowLength); \draw (0:0mm) -- +(-135:\DynkinArrowLength);
      },
      mark=at position 0.001 with
      {
      \draw (0:0mm) -- +(+45:\DynkinArrowLength); \draw (0:0mm) -- +(-45:\DynkinArrowLength);
      },
    },
    postaction={decorate}
  },
  sedge/.style={
  },
  dedge/.style={
    middlearrow,
    double distance=0.5mm,
  },
  tedge/.style={
    middlearrow,
    double distance=1.0mm+\pgflinewidth,
    postaction={draw}, 
  },
  infedge/.style={
    leftrightarrow,
    double distance=0.5mm,
  },
}
\newcommand{\DynkinAn}{
\raisebox{-3.6mm}{\scriptsize
\begin{tikzpicture}[scale=.8]
    \node[dnode,fill=black,label=below:$\alpha_1$] (1) at (0,0) {};
    \node[dnode,fill=black,label=below:$\alpha_2$] (2) at (1,0) {};
    \node[dnode,fill=black,label=below:$\alpha_{n-1}$] (3) at (3,0) {};
    \node[dnode,fill=black,label=below:$\alpha_{n}$] (4) at (4,0) {};
    \path (1) edge[sedge] (2)
          (2) edge[sedge,dashed] (3)
          (3) edge[sedge] (4)
          ;
\end{tikzpicture}
}}
\newcommand{\DynkinBn}{
\raisebox{-3.6mm}{\scriptsize
\begin{tikzpicture}[scale=.8]
    \node[dnode,label=below:$\alpha_1$] (1) at (0,0) {};
    \node[dnode,label=below:$\alpha_2$] (2) at (1,0) {};
    \node[dnode,label=below:$\alpha_{n-1}$] (3) at (3,0) {};
    \node[dnode,label=below:$\alpha_{n}$] (4) at (4,0) {};
    \path (1) edge[sedge] (2)
          (2) edge[sedge,dashed] (3)
          (3) edge[dedge] (4)
          ;
\end{tikzpicture}
}}
\newcommand{\DynkinCn}{
\raisebox{-3.6mm}{\scriptsize
\begin{tikzpicture}[scale=.8]
    \node[dnode,label=below:$\alpha_1$] (1) at (0,0) {};
    \node[dnode,label=below:$\alpha_2$] (2) at (1,0) {};
    \node[dnode,label=below:$\alpha_{n-1}$] (3) at (3,0) {};
    \node[dnode,label=below:$\alpha_{n}$] (4) at (4,0) {};
    \path (1) edge[sedge] (2)
          (2) edge[sedge,dashed] (3)
          (4) edge[dedge] (3)
          ;
\end{tikzpicture}
}}
\newcommand{\DynkinDn}{
\raisebox{-8mm}{\scriptsize
\begin{tikzpicture}[scale=.8]
    \node[dnode,label=below:$\alpha_1$] (1) at (0,0) {};
    \node[dnode,label=below:$\alpha_2$] (2) at (1,0) {};
    \node[dnode,label=below:$\alpha_{n-2}$] (3) at (3,0) {};
    \node[dnode,label=above:$\alpha_{n-1}$] (4) at (4,0.5) {};
    \node[dnode,label=below:$\alpha_n$] (5) at (4,-0.5) {};
    \path (1) edge[sedge] (2)
          (2) edge[sedge,dashed] (3)
          (3) edge[sedge] (4)
              edge[sedge] (5)
          ;
\end{tikzpicture}
}}
\newcommand{\DynkinEsix}{
\raisebox{-3.6mm}{\scriptsize
\begin{tikzpicture}[scale=.7] 
    \node[dnode,label=below:$\alpha_1$] (1) at (0,0) {};
    \node[dnode,label=above:$\alpha_2$] (2) at (2,1) {};
    \node[dnode,label=below:$\alpha_3$] (3) at (1,0) {};
    \node[dnode,label=below:$\alpha_4$] (4) at (2,0) {};
    \node[dnode,label=below:$\alpha_5$] (5) at (3,0) {};
    \node[dnode,label=below:$\alpha_6$] (6) at (4,0) {};
    \path (1) edge[sedge] (3)
          (3) edge[sedge] (4)
          (4) edge[sedge] (5)
 		      edge[sedge] (2)
          (5) edge[sedge] (6);
\end{tikzpicture}
}}
\newcommand{\DynkinEseven}{
\raisebox{-3.6mm}{\scriptsize
\begin{tikzpicture}[scale=.7] 
    \node[dnode,label=below:$\alpha_1$] (1) at (0,0) {};
    \node[dnode,label=above:$\alpha_2$] (2) at (2,1) {};
    \node[dnode,label=below:$\alpha_3$] (3) at (1,0) {};
    \node[dnode,label=below:$\alpha_4$] (4) at (2,0) {};
    \node[dnode,label=below:$\alpha_5$] (5) at (3,0) {};
    \node[dnode,label=below:$\alpha_6$] (6) at (4,0) {};
    \node[dnode,label=below:$\alpha_7$] (7) at (5,0) {};
    \path (1) edge[sedge] (3)
          (3) edge[sedge] (4)
          (4) edge[sedge] (5)
 		      edge[sedge] (2)
          (5) edge[sedge] (6)
          (6) edge[sedge] (7);
\end{tikzpicture}
}}
\newcommand{\DynkinEeight}{
\raisebox{-3.6mm}{\scriptsize
\begin{tikzpicture}[scale=.7] 
    \node[dnode,label=below:$\alpha_1$] (1) at (0,0) {};
    \node[dnode,label=above:$\alpha_2$] (2) at (2,1) {};
    \node[dnode,label=below:$\alpha_3$] (3) at (1,0) {};
    \node[dnode,label=below:$\alpha_4$] (4) at (2,0) {};
    \node[dnode,label=below:$\alpha_5$] (5) at (3,0) {};
    \node[dnode,label=below:$\alpha_6$] (6) at (4,0) {};
    \node[dnode,label=below:$\alpha_7$] (7) at (5,0) {};
    \node[dnode,label=below:$\alpha_8$] (8) at (6,0) {};
    \path (1) edge[sedge] (3)
          (3) edge[sedge] (4)
          (4) edge[sedge] (5)
 		      edge[sedge] (2)
          (5) edge[sedge] (6)
          (6) edge[sedge] (7)
          (7) edge[sedge] (8);
\end{tikzpicture}
}}
\newcommand{\DynkinF}{
\raisebox{-3.6mm}{\scriptsize
\begin{tikzpicture}[scale=.7]
    \node[dnode,label=below:$\alpha_1$] (1) at (0,0) {};
    \node[dnode,label=below:$\alpha_2$] (2) at (1,0) {};
    \node[dnode,label=below:$\alpha_3$] (3) at (2,0) {};
    \node[dnode,label=below:$\alpha_4$] (4) at (3,0) {};
    \path (1) edge[sedge] (2)
          (2) edge[dedge] (3)
          (3) edge[sedge] (4)
          ;
\end{tikzpicture}
}}
\newcommand{\DynkinG}{
\raisebox{-3.6mm}{\scriptsize
\begin{tikzpicture}[scale=.8]
    \node[dnode,label=below:$\alpha_1$] (1) at (0,0) {};
    \node[dnode,label=below:$\alpha_2$] (2) at (1,0) {};
    \path (2) edge[tedge] (1)
          ;
\end{tikzpicture}
}}

\begin{document}

\begin{abstract}
An affine varieties with an action of a semisimple group $G$ is called ``small'' if every non-trivial $G$-orbit in $X$ is isomorphic to the orbit of a highest weight vector. Such a variety $X$ carries a canonical action of the multiplicative group $\Kst$ commuting with the $G$-action. We show that $X$ is determined by the $\Kst$-variety $X^U$ of fixed points under a maximal unipotent subgroups $U \subset G$. Moreover, if $X$ is smooth, then $X$ is a $G$-vector bundle over the quotient $X\quot G$. 

If $G$ is of type $\An$ ($n\geq2$), $\Cn$, $\Esix$, $\Eseven$ or $\Eeight$, we show that all affine $G$-varieties up to a certain dimension are small. As a consequence we have the following result. If $n \geq 5$, every smooth affine $\SL_n$-variety of dimension $<2n$ is an $\SL_n$-vector bundle over the smooth quotient $X\quot\SL_n$, with fiber isomorphic to the natural representation or its dual.
\end{abstract}

\maketitle
\setcounter{tocdepth}{2}

{\small
\tableofcontents
}


\section{Introduction}
%
Our base field $\K$ is algebraically closed of characteristic zero. If a semisimple algebraic group $G$ acts on an affine variety $X$, then the closure of an orbit $G x$ is a union of $G$-orbits and contains a unique closed orbit. 
A very interesting special case is when the closure is  the union of the orbit $G x$ and a fixed point $x_0\in X$: $\overline{G x} = G x \cup\{x_0\}$. Such an orbit is called a {\it minimal orbit}. It turns out that this condition does not depend on the embedding of the orbit $G x$ into an affine $G$-variety. In fact, the minimal orbits are isomorphic to highest weight orbits $O_\lambda$ in irreducible representations  $V_\lambda$ of $G$.

If a $G$-variety $X$ is affine and all orbits in $X$ are either minimal or fixed points, then the variety $X$ is called {\it small}. 

The following result shows that smooth small $G$-varieties have a very special structure.

 \begin{thm}\label{mainthm.cor}
Let $G$ be a simple group and $X$ a smooth irreducible small $G$-variety. 
Then $G\simeq \SL_n$ or $G\simeq \Sp_{2n}$, and the algebraic quotient $X\to X\quot G$ is a $G$-vector bundle with fiber
\be
\item[$\bullet$] the standard representations $\K^{n}$ or its dual $(\K^{n})^{\vee}$ if $G=\SL_n$,
\item[$\bullet$] the standard representation $\K^{2n}$ if $G=\Sp_{2n}$.
\ee
In particular, every fiber is the closure of a minimal orbit. 
\end{thm}
For $G = \SL_n$ or $G = \Sp_{2n}$ it turns out that an affine $G$-variety is small if its dimension is small enough. More precisely, we have the following result.
\begin{thm}\label{mainthm.lem}\strut
\be
\item
For $n\geq 5$ an irreducible affine $\SL_{n}$-variety $X$ of dimension $< 2n-2$ is small. In particular, if $X$ is also smooth, then $X$ is an $\SL_{n}$-vector bundle over $X\quot\SL_{n}$ with fiber $\K^{n}$ or $(\K^{n})^{\vee}$. 
\item
For $n\geq3$, an irreducible affine $\Sp_{2n}$-variety $X$ of dimension $<4n-4$ is small. In particular, if $X$ is also smooth, then it is an $\Sp_{2n}$-vector bundle over $X\quot\Sp_{2n}$ with fiber $\K^{2n}$.
\ee
\end{thm}

In general, we have the following theorem about the structure of a small $G$-variety where $G$ is a semisimple algebraic group. As usual, we fix a Borel subgroup $B \subset G$ and a maximal unipotent subgroup $U \subset B$. For a simple $G$-module $V_\lambda$ of highest weight $\lambda$ we denote by $O_\lambda\subset V_\lambda$ the orbit of highest weight vectors, and by $P_\lambda$ the corresponding parabolic subgroup, i.e. the normalizer of $V_\lambda^U$.

For any minimal orbit $O$ there is a well-defined cyclic covering $O_{\lambda} \to O$ where $\lambda$ is an {\it indivisible} dominant weight, i.e. $\lambda$ is not an integral multiple of another dominant weight. This $\lambda$ is called the {\it type} of the minimal orbit $O$. 

An action of a reductive group $G$ on an affine variety $X$  is called {\it fix-pointed\/} if the closed orbits are fixed points.

\begin{thm}\label{mainthm1}
Let $X$ be an irreducible small $G$-variety. Then the following holds.
\be
\item  
The $G$-action is fix-pointed and in particular $X^G\simto X\quot G$.
\item\label{mainthm1:1} 
All minimal orbits in $X$ have the same type $\lambda$.
\item 
The quotient $X\to X\quot U^-$ restricts to an isomorphism $X^U\simto X\quot U^-$.  In particular, $X$ is normal if and only if $X^U$ is normal.
\item
There is a unique $\Kst$-action on $X$ which induces the canonical $\Kst$-action on each minimal orbit of $X$ and commutes with the $G$-action.
Its action on $X^{U}$ is fix-pointed, and $X^{U}\quot\Kst \simto X\quot G \simot X^{G}$.
\item
The morphism $G\times X^U \to X$, $(g,x)\mapsto  g x$, induces a $G$-equivariant isomorphism
\[
\Phi\colon\Olb\times^{\Kst} X^{U}\overset{\simeq}{\longrightarrow} X
\]
where $\Kst$ acts on $\Olb$ by $(t,x)\mapsto \lambda(t^{-1})\cdot x$.
\item 
We have $\mathrm{Stab}_{G}(X^{U})=P_{\lambda}$, and
the $G$-equivariant morphism 
\[
\Psi\colon G\times^{P_{\lambda}}X^{U}\rightarrow X,\quad [g,x]\mapsto g x,
\] 
is proper, surjective and birational, and induces an isomorphism between the algebras of regular functions.  
\ee
\end{thm}
\noindent
The proofs are given in Proposition~\ref{main.lem} for the statements (1)--(3) and in Proposition~\ref{small-var.thm} for the statements (4)--(6). We define the {\it canonical $\Kst$-action} on a minimal orbit in Section \ref{canonical.subsec}.

\ps
As a consequence, we obtain the following one-to-one correspondence between irreducible small $G$-varieties of a given type and certain irreducible fix-pointed affine $\K^*$-varieties. The $\K^*$-action on a variety $Y$ is called {\em positively fix-pointed} if for  every $y \in Y$ the limit $\lim_{t\to 0}t y$ exists and is therefore a fixed point.

\begin{cor}\label{equivalence.cor}
For any indivisible highest weight  $\lambda\in \Lambda_{G}$, the functor $F \colon X \mapsto X^{U}$ defines an equivalence of categories
\[
\left\{\begin{array}{l}
\text{\it irreducible small $G$-varieties $X$} \\ \text{\it of type $\lambda$}
\end{array}\right\}
\overset{F}{\longrightarrow}
\left\{
\begin{array}{l}
\text{\it irreducible positively fix-pointed} \\ \text{\it affine $\Kst$-varieties $Y$}
\end{array}\right\}.
\]
The inverse of $F$ is given by $Y \mapsto \Olb \times^{\Kst} Y$ where the $\Kst$-action on $\Olb\times Y$ is defined as $t(v,y)\mapsto(\lambda(t^{-1}) v, t y)$. 
\end{cor}
Our Theorem~\ref{mainthm.cor} above  is a special case of the following description of smooth small $G$-varieties.

\renewcommand{\theenumi}{\roman{enumi}}
\begin{thm}[see Theorem~\ref{mainthm2b}]\label{mainthm2}
Let $X$ be an irreducible small $G$-variety of type $\lambda$, and consider the following statements.
\be
\item\label{vectorbundle}
The quotient $\pi\colon X \to X\quot G$ is a $G$-vector bundle with fiber $V_{\lambda}$.
\item\label{linebundle}
$\Kst$ acts faithfully on $X^{U}$, the quotient $X^{U} \to X^{U}\quot\Kst$ is a line bundle, and $V_{\lambda}=\overline{O_{\lambda}}$.
\item\label{principal}
The quotient $X^{U}\setminus X^{G}\to X^{U}\quot\Kst$ is a principal $\Kst$-bundle, and $V_{\lambda}=\overline{O_{\lambda}}$.
\item\label{disjoint}
The closures of the minimal orbits of $X$ are smooth and pairwise disjoint.
\item\label{smooth}
The quotient morphism $\pi\colon X \to X\quot G$ is smooth.
\ee
Then the assertions {\rm(\ref{vectorbundle2})} and {\rm(\ref{linebundle2})} are equivalent and imply {\rm(\ref{principal2})--(\ref{smooth2})}. If $X$ (or $X^U$) is normal, all assertions are equivalent.

Furthermore, $X$ is smooth if and only if $X\quot G$ is smooth and $\pi\colon X \to X\quot G$ is a $G$-vector bundle.
\end{thm}
\renewcommand{\theenumi}{\arabic{enumi}}

In order to see that small-dimensional $G$-varieties are small (see Theorem~\ref{mainthm.lem}) we have to compute the minimal dimension $d_G$ of a non-minimal quasi-affine $G$-orbit. In fact, if the dimension of the affine $G$-variety $X$ is less than $d_G$, then every orbit in $X$ is either minimal or a fixed point, hence $X$ is small.

We define the following invariants for a semisimple group $G$.
\begin{align*}
m_{G} &:=\min\{\dim O \mid O \text{ a minimal $G$-orbit}\},\\
d_{G} &:=\min\{\dim O \mid O \text{ a non-minimal quasi-affine nontrivial $G$-orbit}\},\\
r_{G} &:= \min\{\codim H\mid H \subsetneqq G \text{ reductive subgroup}\}.
\end{align*}

The following theorem lists $m_G$, $d_G$ and $r_G$ for the simply connected simple groups, and also gives the closure $\overline{O}$ of the minimal orbits realizing $m_G$ and the reductive subgroup $H$ of $G$ realizing $r_G$. In the last column the null cone $\NNN_V$ appears only if $\NNN_V \subsetneqq V$.

\begin{thm}\label{mainthm3}
Let $G$ be a simply connected simple group. Then the invariants $m_G,r_G,d_G$ are given by the following table. In particular, $d_G=r_G$ except for $\Eseven$ and $\Eeight$.
\ps
\begin{center}
\begin{tabular}{ c | c | c | c | c | c | c } 
$G$ &  $\dim G$ & $m_{G}$ &$d_{G}$ &$r_{G}$ & $H$  & $\overline{O}$  \\ \hline\hline
$\Aone$ & $3$ & $2$ & $2$ & $2$ & $ \Tone $ &  $\K^{2}$  \\ \hline
$\Atwo$ & $8$ & $3$ &$4$ &  $4$ &$\Aone\times\Tone$ & $\K^{3}, (\K^{3})^{\vee}$  \\ \hline
$\A_3$ & $15$ & $4$ &$5$ & $5$ & $\Btwo$ &  $\K^{4}, (\K^{4})^{\vee}$  \\ \hline
$\An,n>3$ & $n(n+2)$ & $n+1$ &$2n$ & $2n$ &$\A_{n-1}\times\Tone$ & $\K^{n+1}, (\K^{n+1})^{\vee}$  \\ \hline
$\Btwo$ & $10$ & $4$ &$4$ & $4$ &$\Aone\times\Aone$ & $ \NNN_{V_{\omega_1}}, V_{\omega_2}=\K^4$ \\ \hline
$\Bn, n > 2 $ & $n(2n+1)$ & $2n$ &$2n$ & $2n$ & $\Dn$ & $\NNN_{V_{\omega_1}}$  \\ \hline
$\Cn, n\geq3$ & $n(2n+1)$ & $2n$ &$4n-4$ & $4n-4$ & $\C_{n-1}\times \Aone$ & $\K^{2n}$  \\\hline
$\Dfour$ & $14$ & $7$ &$7$ & $7$& $\B_{6}$ & $\NNN_{V_{\omega_1}},\NNN_{V_{\omega_3}},\NNN_{V_{\omega_4}}$  \\\hline
$\Dn,n\geq5$ & $n(2n-1)$ & $2n-1$ &$2n-1$ & $2n-1$& $\B_{n-1}$ & $\NNN_{V_{\omega_1}}$  \\\hline
$\Esix$ & $78$ & $17$ & $26$ & $26$ & $\Ffour$ & $\subsetneqq \NNN_{V_{\omega_{i}}}, i=1,6$\\ \hline
$\Eseven$ & $133$ & $28$ &$ 45 $ & $54$ & $\Esix\times \Tone$ & $\subsetneqq \NNN_{V_{\omega_{7}}}$ \\ \hline
$\Eeight$ & $248$ & $58$ &$ 86 $ & $112$ & $\Eseven\times\Aone$ & $\subsetneqq \NNN_{\Lie \Eeight}$  \\ \hline
$\Ffour$ & $52$ & $16$ &$16$ & $16$ & $\B_{4}$ &  $\subsetneqq \NNN_{V_{\omega_{i}}}, i=1,4$  \\ \hline
$\Gtwo$ & $14$ & $6$ &$6$ & $6$ & $\Atwo$ & $\NNN_{V_{\omega_1}}, \subsetneqq\NNN_{\Lie\Gtwo}$ \\ \hline
\end{tabular}
\par\medskip
\captionof{table}{\label{tab2}The invariants $m_{G}$, $r_{G}$, $d_{G}$ for the simple groups, the orbit closures realizing $m_G$ and the reductive subgroups $H\subsetneqq G$ realizing $r_G$.}
\end{center}
\end{thm}

The third and last columns of Table~\ref{tab2} will be provided by Lemma~\ref{min-parabolics.lem}, the fourth column by Proposition~\ref{dG.prop} and the fifth and sixth columns by Lemma~\ref{rG.lem}.

Note also that Theorem~\ref{mainthm.lem} is a consequence of Theorem~\ref{mainthm.cor} and Theorem~\ref{mainthm3}, because  $X$ is a small $G$-variety in case $\dim X<d_G$.
\par\bigskip
{\small
\noindent
{\bf Acknowledgments:} We thank \name{Oksana Yakimova} for her help with the computation of the invariant $d_G$. We would also like to thank \name{Michel Brion} for interesting and helpful discussions.}

\par\bigskip
\section{Minimal \texorpdfstring{$G$}{G}-orbits}\label{sec:minimal_orbits}

In this paragraph we introduce and study  {\it minimal orbits} of a semisimple group $G$.  
We will use the standard notation below and refer to the literature for details (see for instance \cite{Bo1991Linear-algebraic-g,FuHa1991Representation-the,Hu1978Introduction-to-Li,Hu1975Linear-algebraic-g, Ja2003Representations-of,  Kr1984Geometrische-Metho,Pr2007Lie-groups}). 

Let $G$ be a semisimple group. We fix a Borel subgroup $B \subset G$ and a maximal torus $T \subset B$, and denote by $U := B_{u}$ the unipotent radical of $B$. 

\ps
\subsection{Highest weight orbits}\label{min-orbits.subsec}
Let $\Lambda_{G} \subset X(T):=\Hom(T,\Kst)$ be the monoid of {\it dominant weights} of $G$. A simple $G$-module $V$ is determined by its highest weight $\lambda\in \Lambda_{G}$ which is the weight of the one-dimensional subspace $V^{U}$, and we write $V=V_{\lambda}$. 
The dual module of a $G$-module $W$ will be denoted by $W^{\vee}$, and for the highest weight of the dual module $V_{\lambda}^{\vee}$ we write  $\lambda^{\vee}$. 

\begin{rem}\label{minimal-lambda.rem} 
Define $\Lambda := \bigoplus_{i=1}^{r}\NN\omega_{i}\subseteq \Lambda\otimes_\ZZ \QQ$ where $\omega_{1},\ldots,\omega_{r}$ are the {\it fundamental weights}. We have $\Lambda_{G}\subseteq \Lambda$ with equality if and only if $G$ is simply connected.
In general, let $\gamma\colon\tilde G \to G$ be the universal covering, i.e. $\tilde G$ is simply connected and $\gamma$ is a homomorphism with a finite kernel $F$, then $F$ is contained in the center $\Cent \tilde G$ of $\tilde G$ which acts canonically on $\Lambda$, and we have $\Lambda_{G}  = X(T) \cap \Lambda = \Lambda^{F}$.
\end{rem}

For an affine $G$-variety $X$, we denote by $\pi\colon  X \to X\quot G$ the {\it algebraic quotient}, i.e. the morphism defined by the inclusion $\OOO(X)^{G} \into \OOO(X)$. If $X=V$ is a $G$-module, then the closed subset
$$
\NNN_{V}:=\pi^{-1}(\pi(0)) = \{ v\in V \mid \overline{G v}\ni 0\} \subseteq V
$$
is called the {\it null cone\/} or {\it null fiber\/} of $V$. It is a {\it closed cone in $V$}, i.e. it is closed and contains with any $v$ the line $\K v \subset V$.
\par\medskip
Let  $V=V_{\lambda}$ be a simple $G$-module of highest weight $\lambda \in \Lambda_{G}$. Then $\dim V^U = 1$, and we define the {\it highest weight orbit} to be $O_{\lambda}:=G v \subset V$ where $v\in V^U\setminus\{0\}$ is an arbitrary highest weight vector of $V$.
It is a {\it cone}, i.e.  stable under scalar multiplication. These orbits and their closure have first been studied in \cite{ViPo1972A-certain-class-of}. 

For a subset $S$ of a $G$-variety $X$, the {\it normalizer} of $S$ is defined in the usual way:   $\Norm_G(S):=\{g\in G\mid g S=S\}$. The group of $G$-equivariant automorphisms of $X$ will be denoted by $\Aut_G(X)$.

\begin{lem}\label{hw-orbits.lem}
Let $V=V_\lambda$ be a simple $G$-module of highest weight $\lambda$, and let $v\in V^U$ be a highest weight vector. Then the following holds.
\begin{enumerate}
\item\label{normal}
We have $\Olamb=G V^{U} = O_{\lambda}\cup \{0\}$, and $\Olamb$ is a normal variety.
\item\label{hw-grading}
We have $\OOO(\Olam) = \OOO(\Olamb) \simeq \bigoplus_{k\geq 0} {V_{k\lambda}}^{\!\!\vee} = \bigoplus_{k\geq 0} V_{k\lambda^{\!\vee}}$. In particular, $\Olam$ is not affine.
\item\label{U-fixed}
We have $O_{\lambda}^{U}= \Kst v$, and so $G_v = \Cent_{G}(O_{\lambda}^{U})$. Moreover, $V^U = \K v = V^{G_v} = V^{G_v^\circ}$.
\item\label{parabolic}
The group $P_{\lambda}: = \Norm_G(O_{\lambda}^{U})=\Norm_G(\K v)\subset G$ is a proper parabolic subgroup. We have $P_v = \Norm_{G}G_{v}
=\Norm_G(G_v^\circ)$,  and $\dim\Olam=\codim P_\lambda+1$.
\item\label{scalarmultiplication}
The scalar multiplication on $V$ induces an isomorphism
$\Kst \simto \Aut_{G}(\Olamb)=\Aut_{G}(O_{\lambda})$.
\item\label{minimal}
If $w \in \NNN_{V}$ and $w\neq 0$, then $\overline{G w} \supset \Olam$.
\item\label{sing} 
The closure $\Olamb$ is nonsingular if and only if $\Olamb=V_{\lambda}$.
\end{enumerate}
\end{lem}
\begin{proof}
(\ref{normal})--(\ref{hw-grading})  
These two statements can be found in \cite[Theorem~1 and 2]{ViPo1972A-certain-class-of}. 
\ps
(\ref{U-fixed}) 
We have $\Olam^U\subset V^U=\K^* v\cup\{0\}$, hence $\Olam^U\subset\K^*v$. They are equal because $\Olam$ is a cone. Since $G_v = G_w$ for all $w \in \Kst v$ we see that $G_v = \Cent_{G}(O_{\lambda}^{U})$ and $V^{G_v} \supseteq \K v$. Now the second claim follows, because $U \subseteq G_v^\circ \subseteq G_v$, and so $V^{G_v} \subseteq V^{G_v^\circ} \subseteq V^U = \K v$.
\ps
(\ref{parabolic}) 
$G$ acts on the projective space $\PP(V)$, and the projection $p\colon V\setminus\{0\} \to \PP(V)$ is $G$-equivariant and sends closed cones to closed subsets. In particular, $p(\Olam)=G\,p(v)$ is closed, and so $P_{\lambda} := G_{p(v)}=\Norm_G(\K v)= \Norm_G(O_\lambda^U)\subset G$ is a parabolic subgroup normalizing $G_v$. If $g \in G$ normalizes $G_v^\circ$, then $G_{g v}^\circ=G_v^\circ$ and so $g v \in \Kst v=O_\lambda^U$ by (\ref{U-fixed}). Hence, $\Norm_G(G_v) \subseteq \Norm_G(G_v^\circ) \subseteq \Norm(O_\lambda^U) = P_\lambda \subseteq \Norm_G(G_v)$.
\ps
(\ref{scalarmultiplication}) 
By (\ref{normal}), we have $\Aut_{G}(\Olamb)=\Aut_{G}(O_{\lambda})$. Since $\Olamb$ is a cone, we have an inclusion $\K^*\hookrightarrow \Aut_{G}(\Olamb)$. Any  $\alpha\in \Aut_{G}(\Olamb)$ is $U$-equivariant and hence preserves ${\Olamb}^U = V^{U}$ as well as $\{0\}\in V^{U}$, and the claim follows.
\ps
(\ref{minimal}) 
Let $Y:=\overline{G v} \subset \NNN_{V}$ which implies that $0\in Y$. Since $Y$ is irreducible, the fixed point set $Y^{U}$ does not contain isolated points 
(see e.g. \cite[III.5, Theorem~5.8.8]{Kr2016Algebraic-Transfor}), 
and so  $Y^U \neq \{0\}$. Hence $Y$ contains a highest weight vector, and so $Y \supset \Olam$.
\ps
(\ref{sing}) 
The tangent space $T_0\Olb$ is a nontrivial submodule of $V_\lambda$, hence equal to $V_\lambda$. If $\Olb$ is smooth, then $\dim \Olb = \dim T_0\Olb = \dim V_\lambda$ and so $\Olb = V_\lambda$. The other implication is clear.
\end{proof}

For any $k\geq 1$ the $k$th symmetric power $S^{k}(V_{\lambda})$ contains $V_{k\lambda}$ with multiplicity 1. It is the $G$-submodule generated by $v_{0}^{k}\in S^{k}(V_{\lambda})$ where $v_{0}\in V_{\lambda}$ is a highest weight vector. Let $p\colon S^{k}(V_{\lambda}) \to V_{k\lambda}$ be the linear projection. 
Then the map $v \mapsto p(v^{k})$ is a homogeneous $G$-equivariant morphism $\phi_{k}\colon V_{\lambda} \to V_{k\lambda}$ of degree $k$, classically called a {\it covariant}.

\begin{lem}\label{quotient-of-Omin.lem}
Let $V=V_\lambda$ be a simple $G$-module of highest weight $\lambda$ and $v\in V^U$ a highest weight vector. For $k\geq1$ 
define $\mu_{k}:=\{\zeta\in\Kst \mid \zeta^{k}=1\} \subset \Kst$.

The covariant $\phi_{k}\colon V_{\lambda}\to V_{k\lambda}$ is a finite morphism of degree $k$ and induces a bijective morphism $\bar\phi_{k}\colon V_{\lambda}/\mu_{k}\to \phi_{k}(V_{\lambda})$, where $\mu_{k}$ acts by scalar multiplication on $V_{\lambda}$. 

In particular, the induced map $\phi_{k}\colon \Olam\to O_{k\lambda}$ is a finite $G$-equivariant cyclic covering of degree $k$, and $\phi_{k}\colon\Olamb\to \overline{O_{k\lambda}}$ is the quotient by the action of $\mu_{k}$. 
\end{lem}
\begin{proof}
Since $\phi_{k}^{-1}(0) = \{0\}$, the homogeneous morphism $\phi_{k}$ is finite, the image $\phi_{k}(V_{\lambda})$ is closed, and the fibers of $\phi_{k}$ are the $\mu_{k}$-orbits. This yields the first statement. The last statement follows from the fact that $\overline{O_{k\lambda}}$ is normal, by Lemma~\ref{hw-orbits.lem}(\ref{normal}).
\end{proof}

\begin{rem}\label{covering-Omin.rem}
The following remarks are direct consequences of the lemma above.
\be
\item For $k>1$  we have $\phi_{k}(V_{\lambda})\subsetneqq V_{k\lambda}$, because the quotient $V_{\lambda}/\mu_{k}$ is always singular in the origin. In particular, $\dim V_{k\lambda}>\dim V_{k}$.
\item
The image under $\phi_{k}$ of any nontrivial orbit $O \subset V_{\lambda}$ is an orbit $\phi_{k}(O) \subset V_{k\lambda}$, and the induced map $\phi_{k}\colon O \to \phi_{k}(O)$ is a cyclic covering of degree $k$. 
\item\label{covering-Omin.rem:4}
For $k>1$ we have $\dim V_{k\lambda}>\dim V_{\lambda}\geq \dim\overline{\Olam}=\dim \overline{O_{k\lambda}}$, hence
$\overline{O_{k\lambda}}$ is  singular in the origin, by Lemma~\ref{hw-orbits.lem}(\ref{sing}).
\ee
\end{rem}

The following lemma states that orbits of the form $\Olam$ are  minimal among $G$-orbits.

\begin{lem}\label{Omin.lem}
Let $W$ be a $G$-module and $w\in W$ a nonzero element. If $p\colon W \onto V$ is the projection onto a simple factor $V\simeq V_{\lambda}$ of $W$ such that $p(w)\neq 0$, then $\dim G w \geq \dim O_{\lambda}$.
\end{lem}
\begin{proof}
If $v:=p(w)\neq 0$, then $\dim G w \geq \dim G v > 0$. Hence we can assume that $W = V$ is a simple $G$-module and $p$ the identity map.

Given a closed subset $Y \subset V$ of a vector space, one defines the {\it associated cone} $\CCC Y \subset V$ to be the zero set of the functions $\gr f$, $f \in I(Y)\subset \OOO(V)$, where $\gr f$ denotes the homogeneous term of $f$ of maximal degree. If $Y$ is irreducible, $G$-stable and belongs to a fiber $\pi^{-1}(z)$ of the quotient morphism $\pi\colon V \to V \quot G$, then $\CCC Y \subseteq \NNN_{V}$, and $\CCC Y$ is $G$-stable and equidimensional of dimension $\dim Y$, see \cite[\S 3]{BoKr1979Uber-Bahnen-und-de}. Lemma~\ref{hw-orbits.lem}(\ref{minimal}) now implies that the highest weight orbit $O \subset V$ belongs to $\CCC Y$, and the claim follows.
\end{proof}

\begin{exa}
The simple $\SLtwo$-modules are given by the binary forms $V_{m}:=\K[x,y]_{m}$, $m \in \NN$. The form $y^{m}\in V_{m}$ is a highest weight vector whose stabilizer is 
$$
U_{m}:=  \left\{\left[\begin{smallmatrix} \zeta & s \\ & \zeta^{-1}\end{smallmatrix}\right]\mid \zeta^{m}=1, s\in \K\right\},
$$ 
hence $O_m\simeq \SLtwo/U_{m}$. If $m=2k$ is even, then $x^{k}y^{k}\in V_{m}$ is fixed by the diagonal torus $T\subset \SLtwo$, and the orbit $O = \SLtwo x^{k}y^{k}$ is closed and isomorphic to $\SLtwo/T$ for odd $k$ and to $\SLtwo/N$ for even $k$ where $N\subset \SLtwo$ is the normalizer of $T$. It is easy to see that in both cases the associated cone $\CCC O$ is equal to $\overline{O_m}$.
\end{exa}

\ps
\subsection{Stabilizer of a highest weight vector and coverings}\label{ass-parabolic.subsec}
Let $O_{\lambda}= G v \subset V_{\lambda}$ be a highest weight orbit where $v \in V_{\lambda}^{U}$. 
We have seen in Lemma~\ref{hw-orbits.lem}(\ref{parabolic}) that
$$
P_{\lambda}: = \Norm_G(O_{\lambda}^{U})=\Norm_G(\K v) = \Norm_{G}G_{v}\subset G
$$
is a parabolic subgroup.
 It follows that the weight $\lambda$ extends to a character of $P_\lambda$ defining the action of $P_\lambda$ on $\K v$: 
$$
p v' = \lambda(p)\cdot v' \text{ for } v' \in \K v\text{ and }  p\in P_{\lambda}.
$$ 
Note that $G_{v}= \ker \lambda$, and so $P_{\lambda}/G_{v} \simto \Kst$.

A dominant weight $\lambda \in \Lambda_{G}$ is called {\it indivisible\/} if $\lambda$ is not an integral multiple of some $\lambda' \in \Lambda_{G}$,  $\lambda'\neq\lambda$. For an affine algebraic group $H$, we denote by $H^{\circ}$ its connected component. 

\begin{lem}\label{ray.lem}\strut
\be
\item
Let $\lambda\in \Lambda_{G}$ be a dominant weight of $G$. If $\lambda_{0} \in \QQ\lambda \cap \Lambda_{G}$ is an indivisible element, then $\QQ\lambda \cap \Lambda_{G} = \NN\lambda_{0}$. 
\item
Let $v\in V_{\lambda}$ and $v_0\in V_{\lambda_0}$ be highest weight vectors, and let $k\geq1$ be the integer such that $\lambda=k\lambda_0$. Then: 
\be
\item\label{ray.lem:1} 
$P_{\lambda}=P_{\lambda_{0}}$,
\item\label{ray.lem:2} 
 $G_{v}^{\circ}=G_{v_0}$, and $G_{v}/G_{v}^{\circ}$ is finite and cyclic of order $k$,
\item\label{ray.lem:3}  
$G_{v}$ is connected if and only if $\lambda$ is indivisible.
\item\label{ray.lem:4}
If $\overline{\Olam}$ is smooth, then $\lambda$ is indivisible.
\ee
\item\label{finite-covering}
If $O$ is an orbit and $\phi\colon O \to \Olam$ a finite $G$-equivariant covering, then $O \simeq O_\mu$ where $\lambda = \ell\mu$ for an integer $\ell\geq 1$, and $\phi$ is cyclic of degree $\ell$.
\ee
\end{lem}

\begin{proof}
(1) We have $\Lambda_{G}=\Lambda\cap X(T)$ where $\Lambda := \bigoplus_{i=1}^{r}\NN\omega_{i}$, see Remark~\ref{minimal-lambda.rem}. Then $\QQ\lambda\cap \Lambda=\NN\lambda_1$ for some $\lambda_1\in\Lambda$, and so
\[
\QQ\lambda\cap \Lambda_G=\QQ\lambda\cap(\Lambda\cap X(T))=\NN\lambda_1\cap X(T)=\NN\lambda_0
\]
for some $\lambda_0\in\Lambda_G$. In particular, $\lambda_0$ is the only indivisible element in $\QQ\lambda\cap \Lambda_G$. 
\ps
(\ref{ray.lem:1}) Consider the covariant $\varphi_k\colon V_{\lambda_0}\rightarrow V_{\lambda}$. We have 
\[\varphi_k^{-1}(\K v)=\varphi_k^{-1}(V_{\lambda}^U)=V_{\lambda_0}^U=\K v_0,\] 
so by Lemma~\ref{hw-orbits.lem}(\ref{parabolic}) we obtain that 
\[
P_{\lambda}=\Norm_{G}(\K v)=\Norm_G(\K v_0)=P_{\lambda_0}.
\]

(\ref{ray.lem:2}) We can assume that $v = \phi_k(v_0)$. Then $G_v=\ker(\lambda)\supset \ker(\lambda_0)=G_{v_0}$ and $G_v/G_{v_0}$ is cyclic of order $k$. It follows that $G_{v_0}$ contains the connected component $G_v^{\circ}$ of $G_v$, and hence $G_{v}^{\circ}=G_{v_0}^{\circ}$. 
Moreover, $G_{v_0}=\ker(\lambda_0)$ is connected, because $\lambda_0$ is indivisible. 
\ps
(\ref{ray.lem:3}) follows from (\ref{ray.lem:2}), and (\ref{ray.lem:4}) from Lemma~\ref{hw-orbits.lem}(\ref{sing}) and Remark~\ref{covering-Omin.rem}(\ref{covering-Omin.rem:4}). 
\ps
(\ref{finite-covering})
For $w \in O$ and $v=\phi(w) \in \Olam$ we get a finite covering $G/G_w=O \to O_\lambda = G/G_v$, hence $G_v^\circ \subseteq G_w \subseteq G_v$. By (\ref{ray.lem:2}) we have $G/G_v^\circ = O_{\lambda_0}$ where $\lambda = k\lambda_0$ for an integer $k\geq1$, and the composition $G/G_{\lambda_0} = O_{\lambda_0} \to G/G_w = O \to G/G_v = \Olam$ is a cyclic covering of degree $k$. Therefore, $O_{\lambda_0} \to O$ and $O \to \Olam$ are both cyclic, of degree $m$ and $\ell$ respectively, and $k = \ell m$. Hence $O \simeq O_{m\lambda_0}$ and $\ell(m\lambda_0) = \lambda$.
\end{proof}

\ps
\subsection{Minimal orbits}
In this section we define the central notion of {\it minimal orbits} and prove some remarkable properties.

\begin{defn}\label{minimal-orbit.def}
An orbit $O$ in an affine $G$-variety $X$ isomorphic to a highest weight orbit $\Olam$ will be called {\it a minimal orbit.} This name is motivated by Lemma~\ref{Omin.lem}.
The {\it type} of a minimal orbit $O\simeq O_{\lambda}$ is defined to be the indivisible element $\lambda_{0}\in \QQ\lambda \cap \Lambda_{G}\simeq \NN\lambda_0$ from Lemma~\ref{ray.lem}. 

We denote by $\Obn$ the normalization of $\overline{O} \subset X$  and call it the {\it normal closure of $O$}. Clearly, $\Obn$ is an affine $G$-variety, and the normalization $\eta\colon \Obn \to \Ob$ is finite and $G$-equivariant.
\end{defn}
\begin{lem}\label{normal-closure.lem}
The normalization $\eta\colon \Obn \to \Ob$ is bijective. In particular, $O \subset \Obn$ in a natural way and $\Obn\setminus O$ is a fixed point, as well as $\Ob \setminus O$. Moreover, $\OOO(\Obn) = \OOO(O)$.
\end{lem}
\begin{proof}
Since $O \simeq O_\lambda$ for some $\lambda$ and $\Olb$ is normal (Lemma~\ref{hw-orbits.lem}(\ref{normal})) we have $\Obn \simeq \Olamb$. Therefore, we have  $O \subset \Obn$ in a natural way, and $\Obn \setminus O$ is a fixed point, again by Lemma~\ref{hw-orbits.lem}(\ref{normal}). It follows that 
the normalization $\eta\colon \Obn \to \Ob$ is bijective. This proves the first part of the lemma. The last statement follows from Lemma~\ref{hw-orbits.lem}(\ref{hw-grading}).
\end{proof}

\begin{rem}\label{min-orbits.rem}\strut
\be
\item\label{same-type}
Two minimal orbits $O_{1} \simeq O_{\lambda_{1}}$ and $O_{2}\simeq O_{\lambda_{2}}$ are of the same type if and only if $\QQ\lambda_{1} = \QQ\lambda_{2}$ (Lemma~\ref{ray.lem}). This is the case if and only if for $v_{i}\in O_{i}$ the groups $G_{v_{1}}^{\circ}$ and $G_{v_{2}}^{\circ}$ are conjugate (Lemma~\ref{ray.lem}(\ref{ray.lem:2})), and this holds if and only if the parabolic subgroups $P_{1}:=\Norm_{G}G_{v_{1}}$ and $P_{2}:=\Norm_{G}G_{v_{2}}$ are conjugate.
\item\label{covering2}
Let $O$ be a minimal orbit of type $\lambda_0$, $O\simeq O_{k\lambda_0}$ for an integer $k\geq1$. Then there is a finite cyclic $G$-equivariant covering $O_{\lambda_0} \to O$ of degree $k$ (Lemma~\ref{quotient-of-Omin.lem}). Moreover, $O_{\lambda_0}\simeq G/H$ where $H$ is connected (Lemma~\ref{ray.lem}(\ref{ray.lem:3})). In particular,
if $G$ is simply connected, then $O_{\lambda_0}$ is simply connected and $O_{\lambda_0}\to O$ is the universal covering. 
\ee
\end{rem}

In general, the closure of a minimal orbit needs not to be normal, as shown by the following example.

\begin{exa}\label{SLn.exa}
Let $V_{\eps_{1}} = \Kn$ be the standard representation of $\SL_{n}$. 
For any $k\geq 1$ the minimal orbit $O_{k\eps_{1}}\subset V_{k\eps_{1}}=S^{k}\Kn$ is  the orbit of $e_{1}^{k}$ where $e_1=(1,0,\ldots,0)$, and $O_{\eps_{1}}=\Kn\setminus\{0\} \to O_{k\eps_{1}}$ is the universal covering which is cyclic of degree $k$ and extends to a finite morphism $\Kn \to \overline{O_{k\eps_{1}}}$, $v\mapsto v^{k}$.

Now consider the $\SLn$-module  $W:=\bigoplus_{i=1}^{m} V_{k_{i}\eps_{1}}$ where  $k_{1},\ldots,k_{m}$ are coprime and all $k_{i}>1$. For $w = (e_{1}^{k_{1}},\ldots,e_{1}^{k_{m}})\in W$ we have an $\SLn$-equivariant isomorphism $O_{\eps_{1}} \simto O:=\SLn w$ which extends to a bijective morphism $\phi\colon V_{\eps_{1}} \to \overline{O}$. But $\phi$ is not an isomorphism, because $T_{0}\overline{O}$ is a submodule of $W$, hence cannot be isomorphic to $V_{\eps_{1}}$. In particular, $\overline{O}$ is not normal. The fixed point set $\overline{O}^U$ is the cuspidal curve given by the image of the bijective morphism $\K \to \K^{m}$, $c\mapsto (c^{k_{1}},\ldots,c^{k_{m}})$ which shows again that $\overline{O}$ is not normal.
\end{exa}

The following result collects some important properties of minimal orbits.

\begin{prop}\label{min-orbits.prop}
Let $X$ be an affine $G$-variety and $O \subset X$ a  $G$-orbit. 
\be
\item\label{min-orbits.prop1}
Then $O$ is minimal if and only if 
$\overline{O}\setminus O$ is a single point (which is a fixed point of $G$).
\item\label{min-orbits.prop2} 
If  $O$ is a minimal orbit and $\phi\colon O\to X$ a non-constant $G$-equivariant morphism, then $\phi(O)$ is a minimal orbit of the same type as $O$, and $\phi$ extends to a finite morphism $\bar\phi\colon \Obn\to\overline{\phi(O)}$.
\item \label{min-orbits.prop3}
Let $O$ be a minimal orbit, $Z$ a connected quasi-affine $G$-variety and $\delta\colon Z \to O$ a finite $G$-equivariant covering. Then $Z$ is a minimal orbit of the same type as $O$ and $\delta$ is a cyclic covering.
\item\label{min-orbits.prop4}
If $O_{1}$, $O_{2}$ are minimal orbits and $\eta\colon O_{1} \to O_{2}$ a finite $G$-equivariant covering, then $O_{1}$ and $O_{2}$ have the same type, and the covering is cyclic.
\item\label{min-orbits.prop5}
If $O \subset X$ is minimal, then $\overline{O} \subseteq X$ is smooth if and only if $\overline{O}$ is $G$-isomorphic to a simple $G$-module $V_{\lambda}$. In that case, $\lambda$ is indivisible. 
\ee
\end{prop}

For the proof we will use the following lemma.

\begin{lem}\label{min-orbits.lem}
Let $X, Z$ be affine $G$-varieties and $O \subset Z$ a $G$-orbit. Assume that $\overline{O}\setminus O$ is a fixed point in $Z^G$, and denote by $\eta\colon Y \to \Ob$ the normalization.
\be
\item\label{min-orbits.lem1} 
The morphism $\eta$ induces an isomorphism $\eta^{-1}(O) \simto O$, $Y \setminus \eta^{-1}(O)$ is a fixed point, and $\OOO(O) \simto \OOO(\eta^{-1}(O)) = \OOO(Y)$.
\item\label{min-orbits.lem2}
Every $G$-equivariant non-constant morphism $\phi\colon O \to X$ induces a finite $G$-equivariant morphism $\tilde\phi\colon Y \to X$
$$
\begin{CD}
\eta^{-1}(O) @>{\simeq}>> O \\
@VV{\subseteq}V  @VV{\phi}V \\
Y  @>{\tilde\phi}>> X
\end{CD}
$$
and $\overline{\phi(O)}\setminus\phi(O)$ is a fixed point in $X^G$.
\item\label{min-orbits.lem3}
The orbit $O$ is a minimal orbit, as well as its image $\phi(O) \subset X$, and both have the same type.
\ee
\end{lem}
\begin{proof}
(\ref{min-orbits.lem1}) Let $\overline{O}= O\cup\{x\}$ for some fixed point $x\in Z$. 
If $\eta\colon Y \to \overline{O}$ is the normalization, then $\eta^{-1}(O) \to O$ is an isomorphism, because $O$ is normal. Since $\eta^{-1}(x)$ is finite and $G$-stable it must be a single fixed point $y \in Y$. Moreover, $Y\setminus\eta^{-1}(O)=\{y\}$ has codimension $\geq2$ in $Y$, because a semisimple group does not have 1-dimensional quasi-affine orbits. (In fact,  if there is such a one-dimensional $G$-orbit, then there exists a subgroup $H$ of $G$ isomorphic either to 
 $\SL_2$ or to $\PSL_2$ that acts faithfully on this orbit. But this is not possible, because all  closed subgroups of $H$ of codimension $1$ are conjugate to a Borel subgroup $B$ and $H/B$ is compact.) It follows that $\OOO(Y)=\OOO(O)$.
\ps
(\ref{min-orbits.lem2})-(\ref{min-orbits.lem3}) Since $\OOO(O) \simto \OOO(Y)$ by (\ref{min-orbits.lem1}) and $X$ is affine, the $G$-equivariant morphism $\phi\colon O \to X$ induces a $G$-equivariant morphism $\tilde\phi\colon Y \to X$. There is a closed $G$-equivariant embedding of $X$ into a $G$-module $W$, $X \into W$, and a linear projection $\pr_{V_\lambda}\colon W \to V_{\lambda}$ onto a simple $G$-module $V_{\lambda}$ such that $\phi(O)$ is not in the kernel of $\pr_{V_\lambda}$. 

Set $\psi:=\pr_{V_\lambda}\circ \bar\phi \colon Y \to V_{\lambda}$. Since a unipotent group $U$ does not have isolated fixed points on an irreducible affine $U$-variety (see e.g. \cite[Theorem 5.8.8]{Kr2016Algebraic-Transfor}), we get $O^{U}\neq \emptyset$, and so $\psi(O)^{U}\neq \emptyset$.
This implies that $\psi(O) = O_{\lambda}$ and $\psi(Y) = \Olamb$. 
We have $\psi^{-1}(0) = \{y\}$, and so $\psi$ is finite and surjective. In particular, $O$ is a minimal orbit of the same type as $\Olam$, by Lemma~\ref{ray.lem}(\ref{finite-covering}).
From the factorization
$$
\begin{CD}
\psi \colon Y @>{\tilde\phi}>> \overline{\phi(O)} @>{\pr_V}>> \Olamb
\end{CD}
$$
we see that both maps are  finite, and so $\phi(O)$ is a minimal orbit as well, of the same type as $O_{\lambda}$, again by Lemma~\ref{ray.lem}(\ref{finite-covering}).
\end{proof}

\begin{proof}[Proof of Proposition~\ref{min-orbits.prop}]
(\ref{min-orbits.prop1})  
One implication follows from Lemma~\ref{normal-closure.lem}, the other is Lemma~\ref{min-orbits.lem}(\ref{min-orbits.lem3}).
\ps
(\ref{min-orbits.prop2}) 
This follows from (\ref{min-orbits.prop1}) and Lemma~\ref{min-orbits.lem}(\ref{min-orbits.lem2}).
\ps
(\ref{min-orbits.prop3}) 
We can assume that $O = O_{\lambda} \subset V_{\lambda}$. Let $v_{0} \in V_{\lambda}^{U}$ be a highest weight vector. Since $Z$ is connected, it is a $G$-orbit. Therefore, $Z\simeq G/H'$, where $G_{v_{0}}^{\circ}\subseteq H' \subseteq G_{v_{0}}$, because $\delta$ is finite. 
The group $G_{v_0}/H'\simeq (G_{v_0}/G_{v_0}^{\circ})/(H'/G_{v_0}^{\circ})$ is cyclic because $G_{v_0}/G_{v_0}^{\circ}$ is cyclic by Lemma~\ref{ray.lem}(\ref{ray.lem:2}). It follows that $\delta$ is a cyclic cover of degree $k:=|G_{v_0}/H'|$.  Now Lemma~\ref{quotient-of-Omin.lem} implies that $Z\simeq O_{\frac{\lambda}{k}}$.
\ps
(\ref{min-orbits.prop4})
We can assume that $O_{i} = G v_{i}$ and that $\eta(v_{1}) = v_{2}$. Then $G_{v_{2}}^{\circ}\subseteq G_{v_{1}} \subseteq G_{v_{2}}$ and so $G_{v_{1}}^{\circ}=G_{v_{2}}^{\circ}$. Thus, $O_{1}$ and $O_{2}$ have the same type by Remark~\ref{min-orbits.rem}(\ref{same-type}). Moreover, $\eta^{-1}(v_2)=G_{v_2}v_1\simeq G_{v_2}/G_{v_1}$ is cyclic by Lemma~\ref{ray.lem}(\ref{ray.lem:2}), hence the covering $\eta$ is cyclic.
\ps
(\ref{min-orbits.prop5}) Any ($G$-equivariant) isomorphism $O \simto O_\lambda$ extends to a ($G$-equivariant) isomorphism $\Obn \simto \Olamb$, because $\Olamb$ is normal.
If $\overline{O}$ is smooth, then $\Obn$ and hence $\overline{\Olam}$ are smooth, and so $\overline{\Olam}=V_{\lambda}$ by Lemma~\ref{hw-orbits.lem}(\ref{sing}). In particular, $\lambda$ is indivisible by Lemma~\ref{ray.lem}(\ref{ray.lem:4}). The other implication is obvious. 
\end{proof}

\ps
\subsection{The canonical \texorpdfstring{$\Kst$}{K*}-action on minimal orbits}\label{canonical.subsec}
In this section we show that there exists a unique $\Kst$-action on every minimal orbit $O$ with the following properties.
\be
\item[(a)] 
Every $G$-equivariant morphism $\eta\colon O \to O'$ between minimal orbits is also $\Kst$-equivariant.
\item[(b)]
If $O \subset X$ is a minimal orbit in an affine $G$-variety $X$, then the $\Kst$-action on $O$ extends to the closure $\overline{O}$.
\item[(c)]
If $O \subset X$ is as in (b), then the limit $\lim_{t\to 0} t y$ exists for all $y \in O$ and is equal to the unique fixed point $x_0\in \overline{O}$.
\item[(d)] If $O = \Olam$ where $\lambda$ is indivisible, then the canonical action is the scalar multiplication.
\ee
Let $O\simeq O_{\lambda}$ be a minimal orbit of type $\lambda_{0}$, i.e. $\lambda_{0}$ is indivisible and  $\lambda = \ell\lambda_{0}$ for some $\ell \in\NN$, see Definition~\ref{minimal-orbit.def}. 
Since $\Aut_{G}(O) \simeq \Kst$ by Lemma~\ref{hw-orbits.lem}(\ref{scalarmultiplication}), there are two faithful $\Kst$-actions on $O$ commuting with the $G$-action. Both extend to the normal closure $\Obn$, and for one of them we have that $\lim_{t\to 0}t y$ exists for all $y \in O$ and is equal to the unique fixed point in $\Obn$. This action corresponds to the scalar multiplication in case $O = \Olam \subset V_{\lambda}$.
We call it {\it the action by scalar multiplication} and denote it by $(t,y)\mapsto t\cdot y$. 
\begin{lem}\label{G-equivariant-maps.lem}
Let $O,O'$ be minimal orbits, and let $\eta\colon O \to O'$ be a $G$-equivariant morphism.
\be
\item 
$O$ and $O'$ are of the same type, and $\eta$ extends to a finite $G$-equivariant morphism 
$\tilde\eta\colon \Obn \to \overline{O'}^{\,n}$.
\item
The $G$-equivariant morphisms $\eta$ and $\tilde\eta$ are unique, up to scalar multiplication.
\item\label{covering}
If $O\simeq O_\lambda$ and $O'\simeq O_{\lambda'}$, then $\lambda' = k\lambda$ for an integer $k\geq 1$, and $\eta\colon O \to O'$ is a cyclic covering of degree $k$.
\item\label{formula}
For the scalar multiplication we have $\eta(t\cdot y) = t^k \cdot \eta (y)$ for all $y \in O$.
\item\label{scalar-mult}
The action by scalar multiplication on $\Olam$ corresponds to the representation of $\Kst$ on $\OOO(\Olam)$ which has weight $-n$ on the isotypic component $\OOO(\Olam)_{n \lambda^{\!\vee}}$:  
$$
t f = t^{-n}\cdot f \text{ for } t \in \Kst, \ f \in \OOO(\Olam)_{n\lambda^{\!\vee}}.
$$
\ee
\end{lem}
\begin{proof}
(1) This follows from Proposition~\ref{min-orbits.prop}(\ref{min-orbits.prop2}) and the fact that $\OOO(\Obn) = \OOO(O)$.
\ps
(2) This is clear since $\Aut_G(O) = \Aut_G(\Obn) \simeq \Kst$  which is generated by the scalar multiplication, see Lemma~\ref{hw-orbits.lem}(\ref{scalarmultiplication}).
\ps
(3) This follows from Proposition~\ref{min-orbits.prop}(\ref{min-orbits.prop4}) since $\eta$ is a finite morphism by (1).
\ps
(4) The formula obviously holds for the morphism $O_{\lambda} \to O_{k\lambda}$ induced by $V_{\lambda} \to V_{k\lambda}$, $v \mapsto v^k$ (see Lemma~\ref{quotient-of-Omin.lem}). It remains to see that every $G$-equivariant isomorphism $O \simto O'$ of minimal orbits commutes with the scalar multiplication. This follows from (1) and the fact that the scalar multiplication is the unique faithful $\Kst$-action on $O$ commuting with the $G$-action such that the limits $\lim_{t\to 0} t\cdot y$ exist in $\Obn$.
\ps
(5) This is clear from (\ref{covering})\&(\ref{formula}): the scalar multiplication on $V_\lambda$ induces the multiplication by $t^{-n}$ on the homogeneous component of $\OOO(V_\lambda)$ of degree $n$.
\end{proof}
Using this result we can now define the {\it canonical $\Kst$-action} on minimal orbits.
\begin{defn}\label{canonical-Kst.def}
Let $O \simeq O_{\lambda}$ be a minimal orbit of type $\lambda_{0}$,  where $\lambda = \ell\lambda_{0}$. The {\it canonical $\Kst$-action on $O$\/} is defined by
$$
(t,y)\mapsto t^{\ell}\cdot y \text{ for }t\in \Kst \text{ and }y \in O.
$$
It follows that this $\Kst$-action extends to $\Obn$ such that the limits $\lim_{t\to 0} t^\ell\cdot y$ exist in $\Obn$. If $\lambda$ is indivisible, 
then the canonical action on $\Olam$ coincides with the scalar multiplication, but it is not faithful if $\lambda$ is not indivisible.
\end{defn}

\begin{prop}\label{canonical-Kst.prop}
Let $O\simeq O_{\lambda}$ be a minimal orbit of type $\lambda_0$ where $\lambda=\ell\lambda_0$.
\be
\item\label{canonical-Kst.prop:1} 
The canonical $\Kst$-action on $O$ corresponds to the representation on $\OOO(O)$ which has weight $-n$ on the isotypic component $\OOO(O)_{n\lambda_0^{\!\vee}}$. In particular, it commutes with the $G$-action.
\item\label{canonical-Kst.prop:2} 
If $\eta\colon O \to O'$ is a $G$-equivariant morphism of minimal orbits, then $\eta$ is equivariant with respect to the canonical $\Kst$-action.
\ee
Assume that $O$ is embedded in an affine $G$-variety $X$ and that $\Ob = O\cup\{x_{0}\}\subseteq X$.
\be\addtocounter{enumi}{2}
\item\label{canonical-Kst.prop:3} 
The canonical $\Kst$-action on $O$ extends to $\Ob$.
\item\label{semigroup}
For any $x \in O$ the limit $\lim_{t\to 0}t^{\ell}\cdot x$ exists in $\Ob$ and is equal to $x_{0}$. In particular, the canonical $\Kst$-action on $\Ob$ extends to an action of the multiplicative semigroup $(\K,\cdot)$. 
\item\label{norm}
We have $\Norm_{G}(O^{U})=\Norm_{G}(\Ob^{U}) = P_{\lambda}$, and the action of $P_{\lambda}$ on $\Ob^{U}$ is given by  $p x = \lambda(p)\cdot x = \lambda_0(p)^\ell \cdot x$, i.e. it factors through the canonical $\Kst$-action.
\ee
\end{prop}
\begin{proof}
(\ref{canonical-Kst.prop:1}) The first claim follows from 
Lemma~\ref{G-equivariant-maps.lem}(\ref{scalar-mult}) and obviously implies the second.
\ps
(\ref{canonical-Kst.prop:2}) 
This is an immediate consequence of  Lemma~\ref{G-equivariant-maps.lem}, statements  (\ref{covering}) and (\ref{formula}).
\ps
(\ref{canonical-Kst.prop:3})
Since $\OOO(O) = \OOO(\Obn)$ the claim holds if the closure $\Ob$ is normal. By (1) the canonical $\Kst$-action on $\Obn$ corresponds to the grading of the coordinate ring $\OOO(\Obn) \simeq \bigoplus_{k\geq 0} V_{k\lambda^{\!\vee}}$. 
In the general case, $\OOO(\Ob)$ is a $G$-stable subalgebra of $\OOO(\Obn)$. Since the homogeneous components $V_{k\lambda^{\!\vee}}$ are simple and pairwise non-isomorphic $G$-modules we see that $\OOO(\Ob)$ is a graded subalgebra, hence stable under the canonical $\Kst$-action.
\ps
(\ref{semigroup}) 
This obviously holds for the scalar multiplication on $O_\lambda \subset V_\lambda$, hence in the case where  $\Ob$ is normal. By (\ref{canonical-Kst.prop:3}) it is true in general.
\ps
(\ref{norm})  We have $\Norm_{G}(O_\lambda^{\,U}) = P_{\lambda}$ and $p x = \lambda(p)\cdot x$ for $p\in P_{\lambda}$, $x \in O_{\lambda}^{\,U}$ (cf. Lemma~\ref{hw-orbits.lem}(\ref{parabolic})).
This shows that the action of $P_{\lambda}$ on $O^{U}$ is given by the canonical $\Kst$-action.
It follows from (\ref{canonical-Kst.prop:3}) that it extends to $\Ob$, hence $\Norm_{G}(\Ob^{U}) = \Norm_{G}(O^{U})$.
\end{proof}

\par\bigskip
\section{Graded \texorpdfstring{$G$}{G}-algebras}
Let $G$ be a semisimple group. An affine $G$-varieties whose nontrivial $G$-orbits are  minimal orbits is called a {\it small $G$-variety}. We will show that the coordinate ring of a small $G$-variety is a {\it graded $G$-algebra}, a structure  that we introduce and discuss in this paragraph. 

As in the previous section, we fix a Borel subgroup $B \subset G$, a maximal torus $T \subset B$ and denote by $U := B_{u}$ the unipotent radical of $B$ which is a maximal unipotent subgroup of $G$.

\ps
\subsection{\texorpdfstring{$G$}{G}-algebras and graded \texorpdfstring{$G$}{G}-algebras}
\begin{defn}
A finitely generated commutative $\K$-algebra $R$ with a unit $1=1_{R}$, equipped with a locally finite and rational action of $G$ by $\K$-algebra automorphisms is called a {\it $G$-algebra}. 

If $\lambda_{0}\in\Lambda_{G}$ is an indivisible dominant weight we say that the $G$-algebra $R$ is {\it of type $\lambda_{0}$} if the highest weight of any simple $G$-submodule of $R$ is a multiple of $\lambda_{0}$.

For any $G$-algebra $R$ we have the isotypic decomposition $R = \bigoplus_{\lambda\in\Lambda_{G}} R_{\lambda}$. If this is a grading, i.e. if $R_{\lambda}\cdot R_{\mu} \subseteq R_{\lambda+\mu}$ for all $\lambda,\mu\in\Lambda_{G}$, then $R$ is called a  {\it graded $G$-algebra}.
\end{defn}

\begin{exa}\label{Omin.exa}
Let $V$ be a simple $G$-module of highest weight $\lambda$ and $\Olam \subset V$ the highest weight orbit. Assume that $O_{\lambda}$ is of type $\lambda_{0}$, i.e. $\lambda_{0}$ is indivisible and $\lambda = k\lambda_{0}$ for a positive integer $k$. Then
$$
\OOO(\Olam) = \OOO(\overline{\Olam}) = \bigoplus_{j\geq 0}{V_{j\lambda}}^{\!\!\vee}
= \bigoplus_{j\geq 0}{V_{j k\lambda_{0}}}^{\!\!\vee}
$$ 
by Lemma~\ref{hw-orbits.lem}(\ref{hw-grading}), and so it is a graded $G$-algebra of type $\lambda_{0}^{\vee}$. Note that, by Definition~\ref{canonical-Kst.def}, this grading is induced by the canonical $\Kst$-action $(t,v)\mapsto t^{k} \cdot v$ on $O_{\lambda}$.
\end{exa}

\begin{defn}\label{W_U.def}
Let $H$ be a group, and let $W$ be an $H$-module. Define
$$
W_H:= W /  \langle h w - w \mid h \in H, w \in W\rangle, 
$$
and denote by $\pi_{H}\colon W \to W_{H}$ the projection. Then $\pi_H$ has the {\it universal property\/} that every $H$-equivariant linear map $\phi\colon W \to V$ where $V$ carries the trivial action of $H$ factors uniquely through $\pi_H$. We call $\pi_{H}\colon W \to W_{H}$ the {\it universal $H$-projection} or simply the {\it $H$-projection}.
\newline
If another group $N$ acts linearly on $W$ commuting with $H$, then $N$ acts linearly on $W_H$ and $\pi_H$ is $N$-equivariant.
Note that if $W$ is finite dimensional, then $\pi_H$ is the dual map to the inclusion $(W^\vee)^H \into W^\vee$.
\end{defn}
\begin{exa}\label{V_U-for-simple-V.exa}
Let $V$ be a simple $G$-module of highest weight $\lambda$ and consider the universal $U$-projection $\pi_U \colon V \to V_U$ with respect to the action of the maximal unipotent group $U \subset G$. Since $T$ normalizes $U$ we see that $\pi_U$ is $T$-equivariant and that the kernel is the direct sum of all weight spaces of weight different from the lowest weight $-\lambda^\vee$.
If $U^{-} \subset G$ denotes the maximal unipotent subgroup opposite to $U$, then $V^{U^{-}}$ is the lowest weight space and thus the composition $V^{U^{-}} \into V \onto V_{U}$ is a $T$-equivariant isomorphism.
\end{exa}
\begin{lem}\label{RU.lem}
Let $R$ be a graded $G$-algebra. Then the kernel of the universal $U$-projection $\pi_U\colon R \to R_{U}$ is a graded ideal, and 
the composition $R^{U^-}\hookrightarrow R\stackrel{\pi_U}\to R_U$ is a $T$-equivariant  isomorphism of $\K$-algebras.
\end{lem}
\begin{proof} For the isotypic component $R_{\lambda}$ of $R$ denote by $R_{\lambda}' \subset R_{\lambda}$ the direct sum of all weight spaces of weight different from the lowest weight.  Then $R_{\lambda}=(R_{\lambda})^{U^{-}} \oplus R_{\lambda}'$. Since $R_{\lambda}\cdot R_{\mu} \subseteq R_{\lambda+\mu}$ we get $R_{\lambda}\cdot R_{\mu}' \subset R_{\lambda+\mu}'$, because the lowest weight of $R_{\lambda+\mu}$ is equal to the sum of the lowest weights of $R_{\lambda}$ and $R_{\mu}$. It follow that $\bigoplus_{\mu}R_{\mu}' = \ker \pi_U \subset R$ is an ideal, and that the induced linear isomorphism $R^{U^{-}} \simto R_{U}$ is an isomorphism of $\K$-algebras.
\end{proof}

\begin{exa}\label{XU.exa}
Let $X$ be an affine $G$-variety and assume that $\OOO(X)$ is a graded $G$-algebra. 
Then $\OOO(X^{U})=\OOO(X)_{U}$ and quotient map $X\to X\quot U$ induces an isomorphism $X^U\simto X\quot U^-$.
\newline
In fact, we have $\OOO(X^{U}) = \OOO(X)/\sqrt{I}$ where $I$ is the ideal generated by the linear span $\langle g f - f\mid g\in U, f\in \OOO(X)\rangle = \ker(\OOO(X)\to\OOO(X)_{U})$. Now Lemma~\ref{RU.lem} implies that this kernel is an ideal, hence $\langle g f - f\mid g\in U, f\in \OOO(X)\rangle = I$, and since $\OOO(X)/I \simeq\OOO(X)^{U^{-}} \subseteq \OOO(X)$ we finally get $I = \sqrt{I}$. 
\newline
It follows that the restriction map $\rho\colon\OOO(X) \to \OOO(X^U)$ can be identified with the universal $U$-projection $\pi\colon\OOO(X) \to \OOO(X)_U$, and thus, by Lemma~\ref{RU.lem} above, 
the composition $\OOO(X)^{U^-}\hookrightarrow\OOO(X)\stackrel{\rho}\to\OOO(X^U)$ is an isomorphism. In particular, the quotient $X\to X\quot U^-$ induces an isomorphism $X^U\simto X\quot U^-$.
\end{exa}

\begin{lem}\label{RU-quotients.lem}
Let $\phi\colon R\to S$ be a $G$-equivariant linear map between $G$-modules. If the induced linear map $\phi^{U}\colon R^{U} \to S^{U}$ or $\phi_{U}\colon R_{U}\to S_{U}$ is injective, then $\phi$ is injective. In particular, 
\be
\item\label{RU-quotients.lem:1} if $\phi_U$ or $\phi^U$ is an isomorphism, then so is $\phi$;
\item\label{RU-quotients.lem:2}  if $\psi\colon R\to S$ is another $G$-equivariant linear map such that $\phi_U=\psi_U$ or $\phi^U=\psi^U$, then $\phi=\psi$.
\ee
\end{lem}
\begin{proof}
Let $V \subset \ker\phi$ be a nontrivial simple $G$-submodule. Then $V^U$ and $V_U$ are both nontrivial. The claims follow if we show that $V^U \subset \ker\phi^U$ and that $V_U \into \ker \phi_U$ is injective. The first statement is clear. For the second we remark that $V$ is a direct summand of $R$, $R = V \oplus R'$ as a $G$-module, hence $R_U = V_U \oplus R'_U$.
\end{proof}

Now consider the action of $G\times G$ on $G$ by left- and right-multiplication, i.e. 
$$
(g,h)\cdot x:=g x h^{-1}.
$$ 
With respect to this action one has the following well-known isotypic decomposition:
$$
\OOO(G) \simeq \bigoplus_{\lambda\in\Lambda_{G}} V_{\lambda}\otimes{V_{\lambda}}^{\vee}.
$$
This means that the only simple $G\times G$-modules occurring in $\OOO(G)$ are of the form $V \otimes V^{\vee}$, and they occur with multiplicity 1. The embedding $V \otimes V^{\vee} \into \OOO(G)$ is obtained as follows. The $G$-module structure on $V$ corresponds to a  representation $\rho_{V}\colon G \to \GL(V) \subset \End(V) \simeq V^{\vee}\otimes V$, and the comorphism $\rho_{V}^{*}$ induces a $G\times G$-equivariant embedding $V \otimes V^{\vee} \simto \End(V)^\vee \into \OOO(G)$. (The first map is defined by $(v\otimes\sigma)(\alpha) = \sigma(\alpha(v))$ for $v \in V$, $\sigma\in V^\vee$ and $\alpha \in \End(V)$.)

The action of $U\subset G$ on $G$ by right-multiplication induces 
a $G$-equivariant  isomorphism $\OOO(G/U)\simeq \OOO(G)^U$ with respect to the left-multiplication of $G$ on $G/U$ and on $G$, and we obtain the following 
isomorphisms of $G$-modules  
\begin{equation*}\tag{$\ast$}\label{G/U.grading.eq}
\OOO(G/U) \simeq \OOO(G)^{U} \simeq \bigoplus_{\lambda\in\Lambda_{G}} V_{\lambda}\otimes({V_{\lambda}}^{\vee})^{U}\simeq 
\bigoplus_{\lambda\in\Lambda_{G}} V_{\lambda},
\end{equation*}
giving the isotypic decomposition of $\OOO(G/U) = \OOO(G)^U$. Thus $\OOO(G/U)$ contains every simple $G$-module with multiplicity 1.

Since the torus $T$ normalizes $U$ there is also an action of $T$ on $\OOO(G)^U$ induced by the action of $G$ by right-multiplication, and this $T$-action  commutes with the $G$-action.
Thus we have a $G \times T$-action on $\OOO(G/U)=\OOO(G)^U$.

\begin{rem}\label{weight_T.rem}\strut
\be
\item\label{weight_T.rem:1}
The isomorphism \eqref{G/U.grading.eq} above is $G \times T$-equivariant where $T$ acts on $\OOO(G/U)_\lambda\simeq V_\lambda$ by scalar multiplication with the character $\lambda^\vee$. Thus the $T$-action on $\OOO(G/U)$ corresponds to the grading given by the isotypic decomposition. In particular, $\OOO(G/U)$ is a graded $G$-algebra.
\item\label{weight_T.rem:3} 
The universal  $U$-projection $\pi_U\colon\OOO(G/U)\to \OOO(G/U)_U$ is equivariant with respect to the $T\times T$-action. On the one-dimensional subspace $(\OOO(G/U)_\lambda)_U \subset \OOO(G/U)_U$ the action of $(s,t)\in T\times T$ is given by multiplication with $\lambda^\vee(s)^{-1}\lambda^\vee(t)$.
\ee
\end{rem}
Let $\eps\colon \OOO(G/U) \to \K$ denote the evaluation map $f \mapsto f(e U)$. This is the comorphism of the inclusion $\iota\colon\{e U\} \into G/U$. 

\begin{lem}\label{universal-property.lem}
The induced linear map $\eps_\lambda\colon \OOO(G/U)_\lambda \to \K$ is the universal $U$-projection $\pi_U\colon \OOO(G/U)_\lambda \to (\OOO(G/U)_\lambda)_U$,
and it induces an isomorphism $\bar\eps_\lambda\colon\OOO(G/U)_\lambda^{U^-} \simto \K$.
\end{lem}
\begin{proof}
We first consider the evaluation map $\tilde\eps\colon \OOO(G) \to \K$, $f\mapsto f(e)$, which is the comorphism of the inclusion $\tilde\iota\colon\{e\}\into G$. We claim that on the isotypic components $V_\lambda \otimes V_\lambda^\vee$ of $\OOO(G)$ the map $\tilde\eps$ is given by the formula $\tilde\eps(v\otimes\sigma) = \sigma(v)$. 
In fact, let $\rho_\lambda\colon G \to \GL(V_\lambda) \subset \End(V_\lambda)$ denote the representation on $V_\lambda$. Then the composition $\rho_\lambda\circ\tilde\iota$ sends $e$ to $\id_{V_\lambda}$, hence the comorphism 
$\End(V_\lambda)^\vee \to \K$ is given by $\ell \mapsto \ell(\id_{V_\lambda})$. We have mentioned above that the isomorphism $V \otimes V^\vee \simto \End(V)^\vee$ is defined by $(v\otimes\sigma)(\alpha) := \sigma(\alpha(v))$. This implies that $\tilde\eps\colon V_\lambda\otimes V_\lambda^\vee \simto \End(V)^\vee \to \K$ is given by $v\otimes\sigma\mapsto \sigma(v)$ as claimed.

For the restriction $\eps$ of $\tilde\eps$ to $\OOO(G/U)=\OOO(G)^U$ we thus find for $v \in V_\lambda \simeq \OOO(G/U)_\lambda$ that $\eps(v) = \sigma_0(v)$ where $\sigma_0$ is a highest weight vector in $V_\lambda^\vee$. As a consequence, $\eps(v) \neq 0$ if $v$ has weight $-\lambda^\vee$, i.e. if $v \in \OOO(G/H)^{U^-}$. Now the claims follow from Example~\ref{V_U-for-simple-V.exa}.
\end{proof}
One can use the isomorphisms $\bar\eps_\lambda$ to define elements $f_\lambda:=\bar\eps_\lambda^{\,-1}(1) \in \OOO(G/U)^{U^-}$ with the following properties:  $f_\lambda\cdot f_\mu = f_{\lambda+\mu}$ and $f_0 = 1$. This means that they form a multiplicative submonoid of $\OOO(G/U)^{U^-}$ isomorphic to $\Lambda_G$. In fact, there is a canonical isomorphism $\K[\Lambda_G] \simto \OOO(G/U)^{U^-}$, $x_\lambda \mapsto f_\lambda$.

\ps
\subsection{The structure of a graded \texorpdfstring{$G$}{G}-algebra}
It is a basic fact from highest weight theory that the structure of a $G$-module $M$ is completely determined by the $T$-module structure of $M^U$. In this section we show that the structure of a graded $G$-algebra $R$ is completely determined by the structure of $R_U$ or of $R^{U^-}$ as a $T$-algebra.
\begin{thm}\label{structure-Galgebra.thm}
Let $R$ be a $G$-module. Then there are two canonical $G$-equivariant isomorphisms
\[
\Psi\colon(\OOO(G/U)\otimes R_U)^T \simto R \quad \text{and}\quad \Psi'\colon(\OOO(G/U)\otimes R^{U^-})^T \simto R
\]
where the $T$-action on $\OOO(G/U)$ is by right-multiplication and on $R_U, R^{U^-}$ induced by the $G$-action on $R$. If $R$ is a graded $G$-algebra, then $\Psi$ and $\Psi'$ are isomorphisms of $\K$-algebras.
\end{thm}
For the proof we introduce an intermediate $T$-module $A_R$. 
If $R$ is a $G$-module, then, for every simple $G$-module $V$ of highest weight $\lambda$, there is a canonical $G$-equivariant  isomorphism 
\[
V \otimes \Hom_G(V, R) \simto R_\lambda, \text{ \ given by \ } v\otimes \alpha \mapsto \alpha(v).
\]
In particular, we have isomorphisms $\OOO(G/U)_\lambda\otimes \Hom_G(\OOO(G/U)_\lambda,R) \simto R_\lambda$ for any dominant weight $\lambda$. Recall that we have a $T$-action on $\OOO(G/U)$ by scalar-multiplication with the character $\lambda^\vee$ on $\OOO(G/U)_\lambda$, see Remark~\ref{weight_T.rem}(\ref{weight_T.rem:1}).

\begin{lem}\label{canonial-iso.lem}
There is a canonical $G$-equivariant isomorphism 
\[
(\OOO(G/U)\otimes \bigoplus_{\lambda\in\Lambda_G} \Hom_G(\OOO(G/U)_\lambda,R))^T \simto R.
\]
\end{lem}
\begin{proof}
The action of $T$ on $\OOO(G/U)_\mu\otimes\Hom_G(\OOO(G/U)_\lambda,R)$ is by scalar multiplication with the character $\mu^\vee - \lambda^\vee$, hence $(\OOO(G/U)_\mu\otimes\Hom_G(\OOO(G/U)_\lambda,R))^T = 0$ unless $\mu = \lambda$. For $\mu=\lambda$ the torus $T$ acts trivially and so $(\OOO(G/U)_\lambda\otimes\Hom_G(\OOO(G/U)_\lambda,R))^T\simto R_\lambda$ as we have seen above. Thus the left hand side is $\bigoplus_{\lambda\in\Lambda_G} \OOO(G/U)_\lambda\otimes\Hom_G(\OOO(G/U)_\lambda,R)^T$ which is canonically isomorphic to 
$\bigoplus_{\lambda\in\Lambda_G} R_\lambda = R$.
\end{proof}
Recall that we have natural $T$-actions on $R_U$ and $R^{U^-}$ and a $T$-equivariant isomorphism $R^{U^-} \simto R_U$ (Lemma~\ref{RU.lem}).

\begin{prop}\label{AR-RU-iso.prop}
Define the $T$-module $A_R:=\bigoplus_{\lambda\in\Lambda_G} \Hom_G(\OOO(G/U)_\lambda,R)$ where $T$ acts by right-multiplication on $\OOO(G/U)$. Then there are canonical $T$-equivariant isomorphisms
\[
\phi\colon A_R \simto R_U \quad \text{and} \quad \psi\colon A_R \simto R^{U^-}.
\]
\end{prop}  
\begin{proof}
(1) We first show that for every dominant weight $\lambda$ there is a canonical isomorphism $\phi_\lambda\colon \Hom_G(\OOO(G/U)_\lambda,R) \simto (R_\lambda)_U$. For
$\alpha \in \Hom_G(\OOO(G/U)_\lambda,R)$ consider the composition $\pi\circ\alpha\colon \OOO(G/U)_\lambda \to R_\lambda\to (R_\lambda)_U$. From the universal property of $\eps_\lambda\colon \OOO(G/U)_\lambda \to \K$ (Lemma~\ref{universal-property.lem}) we obtain a unique factorization
\[
\begin{CD}
\OOO(G/U)_\lambda @>{\alpha}>>  R_\lambda \\
@VV{\eps_\lambda}V @VV{\pi_\lambda}V \\
\K @>{\bar\alpha}>> (R_\lambda)_U
\end{CD}
\]
It is easy to see that the map $\phi_\lambda\colon\Hom_G(\OOO(G/U)_\lambda,R) \to (R_\lambda)_U$ defined by $\alpha\mapsto\bar\alpha(1)$ has the required properties.
\ps
(2) Next we show that for every dominant weight $\lambda$ there is a canonical isomorphism $\psi_\lambda\colon\Hom_G(\OOO(G/U)_\lambda,R) \simto (R_\lambda)^{U^-}$. Here we us the elements $f_\lambda :=\bar\eps_\lambda^{\,-1}(1)$ defined after Lemma~\ref{universal-property.lem}, and set $\psi_\lambda(\alpha):=\alpha(f_\lambda)$. Now the claim follows from (1), because $\eps_\lambda(f_\lambda) = 1$ and so 
$\pi_\lambda(\alpha(f_\lambda)) = \tilde\alpha(1)$, i.e. $\bar\pi_\lambda\circ\psi_\lambda = \phi_\lambda$ where $\bar\pi_\lambda\colon R_\lambda^{U^-} \simto (R_\lambda)_U$ is the $T$-equivariant isomorphism induced by $\pi_\lambda$  (see Lemma~\ref{universal-property.lem}).
\end{proof} 
 \begin{proof}[Proof of Theorem~\ref{structure-Galgebra.thm}]
From Lemma~\ref{canonial-iso.lem} above we obtain a canonical  isomorphism $(\OOO(G/U)\otimes A_R)^T \simto R$ of $G$-modules. Now the first part of the theorem follows from Proposition~\ref{AR-RU-iso.prop}. 

For the last claim, we have to work out the multiplication $*$ on $A=A_R$ given by the isomorphism $\psi\colon A_R \simto R^{U^-}$. If $\alpha\in A_{\mu}$ and $\beta\in A_\lambda$, then $\alpha*\beta \in A_{\mu+\lambda}$ is uniquely defined by $(\alpha*\beta)(f_{\mu+\lambda}) = \alpha(f_\mu)\cdot \beta(f_\lambda) \in R_\mu\cdot R_\lambda \subset R_{\mu+\lambda}$. The claim follows if we show that 
\[\tag{$\dagger$}\label{mult.equ}
(\alpha*\beta)(p\cdot q) = \alpha(p)\cdot \beta(q) \text{ for } p\in\OOO(G/U)_\mu \text{ and } q \in \OOO(G/U)_\lambda.
\]
Since $\OOO(G/U)_{\mu} \otimes \OOO(G/U)_{\lambda} \overset{\alpha\otimes \beta}{\longrightarrow} 
R_{\mu}\otimes R_{\lambda} \overset{\text{mult}}{\longrightarrow} R_{\mu+\lambda}$
is a $G$-equivariant linear map it factors uniquely through the multiplication map $\OOO(G/U)_{\mu} \otimes \OOO(G/U)_{\lambda} \to \OOO(G/U)_{\mu+\lambda}$:
\[
\begin{CD}
\OOO(G/U)_{\mu} \otimes \OOO(G/U)_{\lambda} @>{\alpha\otimes \beta}>> R_{\mu}\otimes R_{\lambda}\\
@VV{\text{mult}}V  @VV{\text{mult}}V\\
\OOO(G/U)_{\mu+\lambda}@>{\gamma}>> R_{\mu+\lambda}
\end{CD}
\]
By construction, $\gamma$ is $G$-equivariant and has the property that $\gamma(p\cdot q) = \alpha(p)\cdot\beta(q)$ for $p\in\OOO(G/U)_\mu,  q \in \OOO(G/U)_\lambda$. In particular, $\gamma(f_{\mu+\lambda}) = \gamma(f_\mu\cdot f_\lambda) = \alpha(f_\mu)\cdot\beta(f_\lambda) = (\alpha*\beta)(f_{\mu+\lambda})$, hence $\gamma = \alpha*\beta$ by uniqueness, and so \eqref{mult.equ} follows.
\end{proof} 

\begin{rem}\label{explicit-Psi.rem}
We will later need an explicit description of the isomorphism $\Psi$ from Theorem~\ref{structure-Galgebra.thm}. Let $f \in \OOO(G/U)_\lambda, h\in (R_{\lambda})_U$.  Proposition~\ref{AR-RU-iso.prop} shows that there is a unique $G$-equivariant homomorphism $\alpha\colon\OOO(G/U)_\lambda \to R_\lambda$ such that $\pi_\lambda(\alpha(f)) = h$, and then $\Psi(f\otimes h) = \alpha(f)$ by Lemma~\ref{canonial-iso.lem}:
\[
\begin{CD}
\OOO(G/U)_\lambda @>{\alpha}>>  R_\lambda \\
@VV{\eps_\lambda}V @VV{\pi_\lambda}V \\
\K @>{\bar\alpha}>> (R_\lambda)_U
\end{CD}
\]
Since $\eps_\lambda(f_\lambda) = 1$ we get $\bar\alpha(1) = h$ and so $\pi_\lambda(\Psi(f\otimes h)) = \pi_\lambda(\alpha(f)) = \bar\alpha(\eps_\lambda(f))= \eps_\lambda(f) h$. This shows that the diagram
\[
\begin{CD}
(\OOO(G/U)\otimes R_U)^T @>{\Psi}>> R \\
@VV{\eps\otimes \id}V  @VV{\pi_U}V \\
R_U @= R_U
\end{CD}
\]
commutes.
\end{rem}

\ps
\subsection{Deformation of \texorpdfstring{$G$}{G}-algebras}
In this subsection we give an application of the methods developed above. The results are interesting in their own, but they will not be used in the remaining part of the paper.%
\footnote{The result is cited in \cite{Br1981Sur-la-theorie-des} as an unpublished result due to the first author and is partially reproved there.}

Let $R$ be a $G$-algebra with isotypic decomposition $R = \bigoplus_{\lambda\in\Lambda_{G}} R_{\lambda}$. We define a {\it graded $G$-algebra 
$\gr R$} in the following way. As a $G$-module, we set $\gr R := \bigoplus_{\lambda\in\Lambda_{G}}R_{\lambda}$, and the multiplication is defined by the symmetric bilinear map
$$
\begin{CD}
R_{\lambda}\times R_{\mu}  @>\text{mult}>> R @>{\pr}>> R_{\lambda+\mu}.
\end{CD}
$$
It is not difficult to see that this multiplication is associative, hence defines a $\K$-algebra structure on $\gr R$ such that $\gr R$ becomes a graded $G$-algebra.
We now generalize Theorem~\ref{structure-Galgebra.thm} to non-graded $G$-algebras.

\begin{prop}\label{grR-iso.prop}
For any $G$-algebra $R$ there is a canonical $G$-equivariant isomorphism of $\K$-algebras
$$
(\OOO(G/U)\otimes R^{U^{-}})^{T} \simto \gr R
$$
\end{prop}
\begin{proof}
\ps\noindent
The definition of the multiplication on $\gr R$ implies that the subalgebra $(\gr R)^{U^-}\subset \gr R$ is equal to the subalgebra  $R^{U^-} \subset R$ since one has $R_\mu^{U^-} \cdot R_\lambda^{U^-} \subseteq R_{\mu+\lambda}^{U^-}$.
Applying Theorem~\ref{structure-Galgebra.thm} to the graded $G$-algebra $\gr R$, we get
$$
(\OOO(G/U)\otimes R^{U^{-}})^{T} = (\OOO(G/U)\otimes (\gr R)^{U^{-}})^{T}  \simto \gr R,
$$
hence the claim.
\end{proof}

Let $\Kt:=\Ktt$ be the polynomial ring in $r:=\rk G$ variables. Denote by  $\mm_0 := (t_1,\ldots,t_r) \subset \Kt$ the homogeneous maximal ideal and by $\Kt_\mathbf{t}$ the localization $\K[t_1,\ldots,t_r,t_1^{-1},\ldots,t_r^{-1}]$.
The following {\it Deformation Lemma\/} shows that there exists {\it a flat deformation of $\gr R$ whose general fiber is $R$}. 

\begin{lem}\label{deformation.lem}
Let $R$ be a $G$-algebra. There exists a $\Kt$-algebra $\tilde R$ with the following properties:
\be
\item 
$\tilde R$ is a free $\Kt$-module and, in particular, flat over $\Kt$;
\item\label{deformation.lem:3} 
$\tilde R/\mm_0\tilde R \simeq \gr R$;
\item\label{deformation.lem:2} 
$\tilde R_\mathbf{t}=\Kt_\mathbf{t}\otimes_{\Kt} \tilde{R} \simeq \Kt_{\mathbf{t}}\otimes_{\K} R$.
\ee
\end{lem}
\begin{proof}
For $\lambda = \sum_{i=1}^{r} m_{i}\omega_{i} \in \Lambda_{G}$ we put $t^{\lambda}:=t_{1}^{m_{1}}\cdots t_{r}^{m_{r}}$, so that $\Kt = \bigoplus_{\lambda} \K t^{\lambda}$. On $\Lambda_{G}$ we have a partial ordering: 
$$
\lambda\leq\mu=\sum_{i}n_{i}\omega_{i} \iff m_{i}\leq n_{i} \text{ for all } i.
$$
Define $R_{\leq\lambda}:=\bigoplus_{\mu\leq\lambda}R_{\mu}$ and $R_{<\lambda}:=\bigoplus_{\mu < \lambda}R_{\mu}$. 
Since $R$ is a $G$-algebra we obtain $R_{\leq\lambda}\cdot R_{\leq\mu} \subseteq R_{\leq\lambda+\mu}$. Therefore, the subspace
$$
\tilde R := \textstyle{\bigoplus_{\lambda\in\Lambda_{G}}\,} t^{\lambda}R_{\leq\lambda}\ \subset \Kt\otimes R = \bigoplus_{\lambda\in\Lambda_{G}}t^{\lambda}R
$$
is a subalgebra of $\Kt\otimes R$ which is a free $\Kt$-module with basis $\{t^{\lambda}r_{\lambda,j}\}_{\lambda,j}$ where $\{r_{\lambda,j}\}_j$ is a basis of $R_{\lambda}$, proving (1).

We have $R_{\leq \lambda} = R_{\lambda}\oplus R_{<\lambda}$ which implies that $\mm_0\tilde R = \bigoplus_{\lambda}t^{\lambda}R_{<\lambda}$, and hence 
$$
\tilde R/\mm_0\tilde R \simeq \textstyle{\bigoplus_{\lambda}\,}t^{\lambda}R_{\lambda} \simeq \gr R
$$
which gives (2).
Setting $\ZZ\Lambda_{G}:=\sum_{i}\ZZ\omega_{i}$ we finally get
$$
\tilde R_\mathbf{t} = \textstyle{\bigoplus_{\rho\in\ZZ\Lambda_{G}}} t^{\rho}R = \Kt_{\mathbf{t}}\otimes_{\K} R,
$$
proving (3).
\end{proof}

\begin{rem}
Let $X$ be variety. For simplicity we assume that $X$ is affine. Then {\it a flat family $(A_x)_{x \in X}$ of finitely generated $\K$-algebras} is a finitely generated and flat $\OOO(X)$-algebra $A$ such that $A_x:=A/\mm_x A$ where $\mm_x$ is the maximal ideal of $x \in X$.

The above lemma tells us that for a given $G$-algebra $R$ there is a flat family $(R_x)_{x \in \K^r}$ of finitely generated $G$-algebras  such that $R_0 \simeq \gr R$ and $R_x \simeq R$ for all $x$ from the dense open set $\K^r\setminus \bigcup_i\K e_i$, where $e_1,\dots,e_r$ is the standard basis of $\K^r$.

We say that a property $\PPP$ for finitely generated $\K$-algebras is {\it open} if for any flat family $A = (A_x)_{x \in X}$  of finitely generated $\K$-algebras the subset $\{x \in X\mid A_x  \text{ has property }\PPP\}$ is open in $X$.
\end{rem}

Proposition~\ref{grR-iso.prop} together with the Deformation Lemma~\ref{deformation.lem} allows to show that certain properties of the $U$-invariants $R^{U}$ also hold for $R$. 

\begin{exa}\label{normality.exa}
The following result is due to \name{Vust} \cite[\S1, Th\'eor\`eme 1]{Vu1976Sur-la-theorie-des}): {\it If $R$ is a finitely generated $G$-algebra such that $R^{U}$ is normal, then  $R$ is normal}. In fact, since $R^{U^{-}} \simeq R^{U}$ and $\OOO(G/U)$ are both normal, we see that $(\OOO(G/U)\otimes R^{U^{-}})^{T}$ is normal, hence $\gr R$ is normal, by  Proposition~\ref{grR-iso.prop}. Normality is an open property, i.e. in a flat family $(A_x)_{x \in X}$ of finitely generated $\K$-algebras the set $\{x \in X \mid A_x \text{ is normal}\}$ is open, see \cite[Corollaire~12.1.7(v)]{Gr1966Elements-de-geomet}. Since $\gr R \simeq R_0$ is normal, the Deformation Lemma implies that $R_x$ is normal for all $x$ in an open neighborhood $W$ of $0 \in \K^r$. Since $W$ meets $\K^r\setminus \bigcup_i\K e_i$ it follows that $R$ is normal.
\end{exa}
The argument from this example can be formalized in the following way. 

\begin{prop}\label{property-P.prop}
Let $\PPP$ be a property for finitely generated $\K$-algebras which satisfies the following conditions.
\be
\item[(i)]
$\PPP$ is open;
\item[(ii)]
$\OOO(G/U)$ has property $\PPP$;
\item[(iii)]
If $R$ and  $S$ have property $\PPP$, then so does $R\otimes S$;
\item[(iv)]
If $R$ is a $T$-algebra with property $\PPP$, then $R^{T}$ has property $\PPP$.
\ee
Then a finitely generated $G$-algebra $R$ has property $\PPP$ if $R^{U}$ has property $\PPP$.
\end{prop}
\begin{proof}
If $R^U$ has property $\PPP$, then so does $R^{U^-}$. Hence, assumptions (ii)-(vi) imply that $(\OOO(G/U)\otimes R^{U^-})^T$ has property $\PPP$. In particular, $\gr R$ has property $\PPP$ by Proposition~\ref{grR-iso.prop}. Now (i) implies that $R$ has property $\PPP$ as well.
\end{proof}

Another very interesting property satisfying the assumption of the proposition above  is that of {\it rational singularities}, see \cite{Bo1987Singularites-ratio}.

\par\bigskip
\section{Small \texorpdfstring{$G$}{G}-varieties}
Recall that a affine $G$-variety is {\it small} if every nontrivial orbit is a minimal orbit. We will show that the coordinate ring of a small $G$-variety is a graded $G$-algebra and then use the results of the previous section to obtain important properties of small $G$-varieties and a classification.

\begin{rem}\label{fixpointed.rem}
The $G$-action on a small $G$-variety $X$ is {\it fix-pointed} which means that the closed orbits are fixed points. This has some interesting consequences. For example, it is not difficult to see that for a fix-pointed action the algebraic quotient $\pi\colon X \to X\quot G$ induces an isomorphism $X^G \simto X\quot G$, cf. \cite[\S10, p. 475]{BaHa1987Some-equivariant-K}.
\end{rem}

\ps
\subsection{A geometric formulation}
We  first translate Theorem~\ref{structure-Galgebra.thm}  into the geometric setting. 
By a result of \name{Hadziev} (\cite{Ha1967Certain-questions-}, cf. \cite[3.2 Lemma]{Kr1984Geometrische-Metho}) the $U$-invariants $\OOO(G)^{U}$ are finitely generated, hence define an affine $G$-variety $G\quot U$ with a $G$-equivariant quotient map $\eta\colon G \to G\quot U$. Since $\OOO(G/U) = \OOO(G)^U = \OOO(G\quot U)$ the canonical $G$-equivariant map $G/U \to G\quot U$, $g U\mapsto \eta(g)$, is birational, hence an open immersion: $G/U = G\eta(e) \subset G\quot U$.  Moreover,  the $T$-action on $G/U$ by right-multiplication extends to a $T$-action on $G\quot U$ commuting with the $G$-action.

For an affine $G$-variety $X$ we have a canonical $G$-equivariant morphism 
\[
G/U \times X^{U} \to X,\quad (g U,x) \mapsto g x,
\] 
and a $T$-action on $G/U \times X^{U}$ given by $(t,(g U,x))\mapsto (g t^{-1} U, t x)$. They respectively extend to a morphism $\phi\colon G\quot U\times X^{U} \to X$ and a $T$-action on $G\quot U\times X^{U}$. It follows that $\phi$ is constant on the $T$-orbits, and thus induces a $G$-equivariant morphism
$$
\Phi \colon G\quot U \times^{T} X^{U} :=(G\quot U \times X^{U})\quot T \to X.
$$
\begin{prop}\label{geometric-iso.prop}
Let $X$ be an affine $G$-variety and assume that $\OOO(X)$ is a graded $G$-algebra. Then the canonical morphism
$$
\Phi \colon G\quot U \times^{T} X^{U} \to X
$$
is a $G$-equivariant  isomorphism. Its comorphism is the  inverse of the isomorphism $\Psi$ from Theorem~\ref{structure-Galgebra.thm}. Moreover, the quotient $X \to X\quot U^{-}$ induces an isomorphism $X^{U}\simto X\quot U^{-}$.
\end{prop}

\begin{proof}
By definition, the comorphism 
$\Phi^{*}\colon \OOO(X) \longrightarrow (\OOO(G/U)\otimes \OOO(X^{U}))^{T}$ 
is given as follows: if $\Phi^{*}(f) = \sum_{j} f_{j}\otimes h_{j}$, then $\Phi^{*}(f)(g U,x)=f(g x) = \sum_{j}f_{j}(g U) h_{j}(x)$. 
Consider the evaluation map $\varepsilon\colon\OOO(G/U)\to\K$, $f\mapsto f(e U)$. Then $f(x)=\sum_{j}\eps(f_{j})h_{j}(x)$ for all $x\in X^U$, which shows that the diagram
$$
\begin{CD}
\OOO(X) @>{\Phi^{*}}>>  (\OOO(G/U)\otimes \OOO(X^{U}))^{T}\\
@VV{\rho}V   @VV{\eps\otimes\id}V \\
\OOO(X^{U}) @=    \OOO(X^{U})
\end{CD}
$$
commutes, where $\rho$ is the restriction map, i.e. $\rho(f) = \sum_{j}\eps(f_{j})h_{j} = (\eps\otimes\id)(\Phi^{*}(f))$.
Since $\OOO(X)$ is a graded $G$-algebra it follows from Example~\ref{XU.exa} that the restriction map $\rho$ is equal to the universal $U$-projection $\pi_{U}\colon \OOO(X) \to \OOO(X)_{U}$. If we show  that $\eps\otimes\id$ is also equal to the $U$-projection ${\pi}_U\colon  (\OOO(G/U)\otimes \OOO(X^{U}))^{T}\to \left((\OOO(G/U)\otimes \OOO(X^{U}))^{T}\right)_U$, then $\Phi^*$ is an isomorphism by Lemma~\ref{RU-quotients.lem}.
We have 
$$
(\OOO(G/U)\otimes \OOO(X^{U}))^{T} = \textstyle \bigoplus_{\lambda\in\Lambda_G}\OOO(G/U)_\lambda \otimes \OOO(X^U)_{-\lambda}
$$
where $ \OOO(X^U)_{\mu}$ is the $T$-weight space of $\OOO(X^U)$ of weight $\mu$. Since the evaluation map $\eps_\lambda\colon \OOO(G/U)_\lambda \to \K$, $f\mapsto f(e U)$,  is the universal $U$-projection (Lemma~\ref{universal-property.lem}), we see that the linear map 
$\OOO(G/U)_\lambda \otimes \OOO(X^U)_{-\lambda} \to \OOO(X^U)_{-\lambda}$, $\sum_j f_j\otimes h_j \mapsto \sum_j\eps(f_j)h_j$, is the $U$-projection as well, and the claim follows.
\ps
It remains to see that $\Phi^*$ is equal to the inverse of $\Psi$ from Theorem~\ref{structure-Galgebra.thm}. Using again Lemma~\ref{RU-quotients.lem} it suffices to show that the diagram
$$
\begin{CD}
(\OOO(G/U)\otimes \OOO(X^{U}))^{T}@>{\Psi}>>  \OOO(X)\\
@VV{\eps\otimes\id}V  @VV{\rho}V   \\
\OOO(X^{U}) @=    \OOO(X^{U})
\end{CD}
$$
commutes. This is stated in Remark~\ref{explicit-Psi.rem}. 
The last claim is proved in Example~\ref{XU.exa}.
\end{proof}
proved
\ps
\subsection{The structure of small \texorpdfstring{$G$}{G}-varieties}

\begin{prop}\label{main.lem}
Let $X$ be an irreducible small $G$-variety. Then the following holds.
\be
\item\label{main.lem:1} The $G$-action is fix-pointed, and all minimal orbits in $X$ have the same type $\lambda$.
\item\label{main.lem:2} $\OOO(X)$ is a graded $G$-algebra of type $\lambda^{\!\vee}$.
\item\label{main.lem:3} The quotient $X\to X\quot U^-$ restricts to an isomorphism $X^U\simto X\quot U^-$.  In particular, $X$ is normal if and only if $X^U$ is normal.
\ee 
\end{prop}
\noindent
We call such a variety $X$  {\em a small $G$-variety of type $\lambda$.}
\begin{proof}
(\ref{main.lem:1}) By hypothesis, any non-trivial orbit $O \subset X$ is minimal, so $\overline{O}=O\cup\{x_{0}\}$ where $x_{0}\in X^{G}$  by Proposition~\ref{min-orbits.prop}(\ref{min-orbits.prop1}). In particular, the $G$-action is fix-pointed.

We can assume that $X$ is a closed $G$-stable subvariety of a $G$-module $W$. Let $O \subset X$ be a non-trivial orbit. There is a linear projection $p\colon W \to V$ onto a simple $G$-module $V$ of highest weight $\lambda$ such that $O \nsubseteq \ker p$. Proposition~\ref{min-orbits.prop}(\ref{min-orbits.prop5}) implies that $p(O) = O_{\lambda}$ and that $O$ is of the same type as $O_{\lambda}$. The same is true for all orbits $O'$ from the open subset $X':= X \setminus \ker p$ of $X$. Since $X$ is irreducible all minimal orbits are of type $\lambda$.
\ps
(\ref{main.lem:2}) Let $\lambda$ be the type of the minimal orbits of $X$, and let $O \simeq O_{\mu}$ be an orbit in $X$, $\mu=\ell\lambda$. As explained in Example~\ref{Omin.exa}, we have $\OOO(O) = \OOO(\overline{O})=\bigoplus_{k\geq0}V_{k\mu^{\vee}}$. If $V \subset \OOO(X)$ is a non-trivial simple submodule and $0\neq f \in V$, then there is a nontrivial orbit $O$ such that $f|_{O}\neq 0$. This implies that the $G$-equivariant restriction $\rho\colon \OOO(X) \to \OOO(\overline{O})$ induces an isomorphism $V \simto \rho(V)$ of simple modules. It follows that $V$ has highest weight $k\lambda^{\!\vee}$ for some $k\geq1$. 

It remains to see that $\OOO(X)$ is a graded $G$-algebra. Let  $V_{1}\subset \OOO(X)_{k_{1}\lambda^{\!\vee}}, V_{2}\subset \OOO(X)_{k_{2}\lambda^{\!\vee}}$ be simple submodules, and assume that $V_{1}\cdot V_{2} \nsubseteq \OOO_{(k_{1}+k_{2})\lambda^{\!\vee}}$. Then we can find $f_{i}\in V_{i}$ such that the product $f_{1}\cdot f_{2} \in \OOO(X)$ has the form $f_{1}\cdot f_{2} = \sum_{j}h_{j}$ where $h_{j}\in \OOO(X)_{\ell_{j}\lambda^{\!\vee}}\setminus\{0\}$ with distinct $\ell_{j}$ and at least one $\ell_{j}\neq k_{1}+ k_{2}$.  There exists a non-trivial orbit $O$ such that all functions $f_{1},f_{2},h_{j}$ do not vanish on $O$. It follows that the restrictions $\bar f_{1}, \bar f_{2}, \bar h_{j} \in \OOO(\overline{O})$ are nonzero, that $\bar f_{1}\cdot \bar f_{2}= \sum_{j}\bar h_{j}$, $\bar f_{i}\in\OOO(\overline{O})_{k_{i}\lambda^{\!\vee}}$ and $\bar h_{j}\in \OOO(\overline{O})_{\ell_{j}\lambda^{\!\vee}}$.  Since $\OOO(\overline{O})$ is a graded $G$-algebra (Example~\ref{Omin.exa}) we have $\bar f_{1} \cdot \bar f_{2} \in \OOO(\overline{O})_{(k_{1}+k_{2})\lambda^{\!\vee}}$ which contradicts our assumptions on the $\ell_j$.
\ps
(\ref{main.lem:3}) This follows from (\ref{main.lem:2}), Proposition~\ref{geometric-iso.prop} and the fact that $X$ is normal if and only if $X\quot U^-$ is normal (\cite[\S1, Th\'eor\`eme 1]{Vu1976Sur-la-theorie-des}, cf. Example~\ref{normality.exa}).
\end{proof}

\begin{prop}\label{small-var.thm}
Let $X$ be an irreducible small $G$-variety of type $\lambda$.
\be
\item\label{small-var.thm:2} 
There is a unique $\Kst$-action on $X$ which induces the canonical $\Kst$-action on each minimal orbit and commutes with the $G$-action.
The action on $X^{U}$ is fix-pointed, and $X^{U}\quot\Kst \simto X\quot G \simot X^{G}$.
\item\label{G-iso}
The morphism $G\times X^U \to X$, $(g,x)\mapsto g x$, induces a $G$-equivariant isomorphism
\[
\Phi\colon\Olamb\times^{\Kst} X^{U}\simto X
\]
where $\Kst$ acts on $\Olamb$ by the inverse of the scalar multiplication: $(t,x)\mapsto t^{-1}\cdot x$.
\item\label{small-var.thm:4}  
We have $\mathrm{Stab}_{G}(X^{U})=P_{\lambda}$, and
the $G$-equivariant morphism 
\[
\Theta\colon G\times^{P_{\lambda}}X^{U}\rightarrow X,\quad [g,x]\mapsto g x,
\] 
is proper, surjective and birational and induces an isomorphism between the algebras of regular functions.  
\ee
\end{prop}

\begin{proof}
(\ref{small-var.thm:2})
By  Proposition~\ref{main.lem}(\ref{main.lem:2}) $\OOO(X)$ is a graded $G$-algebra of type $\lambda^{\vee}$. If we define the $\Kst$-action on $\OOO(X)$ such that the isotypic component of type $n \lambda^{\vee}$ has weight $n$, then this action is fix-pointed and restricts to the the canonical $\Kst$-action on the closure of each minimal orbit  (Proposition~\ref{canonical-Kst.prop}(\ref{canonical-Kst.prop:3})). Since $X$ is the union of the closures of the minimal orbits, this $\K^*$-action is unique. By Proposition~\ref{canonical-Kst.prop}(\ref{canonical-Kst.prop:1}) the $\K^*$-action and the $G$-action commute on the closure of every minimal orbit, hence they commute on $X$. We have $X^G= (X^U)^{\K^*}$ and the $\Kst$-action on $X^U$ is fix-pointed, since this holds for the closure of a minimal orbit.
This implies that $X^G = (X^U)^{\Kst} \simto X^U\quot \Kst$, and  $X^G \simto X\quot G$ since the $G$-action is fix-pointed
which yields the remaining claims.
\ps
(\ref{G-iso}) 
Choose $x_{0} \in O_{\lambda}^{U}$ and consider the $G$-equivariant morphism $\eta\colon G\quot U \to \Olamb$ induced by $g U\mapsto g x_0$. Define $D:=\ker \lambda \subset T$. We claim that $\eta$ is the algebraic quotient under the action of $D$. In fact, the action of $t\in T$ on $\OOO(G/U)_\mu$ is by scalar multiplication with $\mu^\vee(t)$ (Remark~\ref{weight_T.rem}(\ref{weight_T.rem:1})). Hence, the action of $D$  is trivial if and only if $\mu$ is a multiple of $\lambda^{\!\vee}$. This implies that
$$
\OOO(G/U)^{D} = \textstyle\bigoplus_{\mu\in\Lambda_G} \OOO(G/U)_\mu^D = \bigoplus_{k\geq 0} V_{k\lambda^\vee} \simeq \OOO(\Olamb),
$$
see Lemma~\ref{hw-orbits.lem}(\ref{hw-grading}). In particular, the $T$-action on $\Olamb$ factors through $\lambda\colon T \to \Kst$ and the induced $\Kst$-action is the canonical $\Kst$-action. Since $D$ acts trivially on $X^U$ we get $(G\quot U \times X^U)\quot D = \Olamb\times X^U$, hence
$$
G\quot U \times^T X^U = (G\quot U \times^D X^U)\quot T = (\Olamb\times X^U)\quot T.
$$
By construction, the $T$-action on $\Olamb\times X^U$ is given by $t(v,x)= (\lambda(t)^{-1}\cdot v, t x)$, i.e. by the inverse of the canonical $\Kst$-action on $\Olamb$ and the given action on $X^U$. Hence $ (\Olamb\times X^U)\quot T = \Olamb \times^{\Kst} X^U$, and the claim follows from Proposition \ref{geometric-iso.prop}.
\ps
(\ref{small-var.thm:4})
Consider the action of $P_{\lambda}$ on $G \times X^{U}$ given by $p(g,x) = (g p^{-1},p x)$. Then the action map $G \times X^{U}\to X$, $(g,x)\mapsto g x$, factors through the quotient 
\begin{equation*}
G\times^{P_{\lambda}}X^{U}:=(G\times X^{U})\quot P_{\lambda}.
\end{equation*}
For $\Theta$ we have the following factorization:
$$
\begin{CD}
\Theta\colon\ G\times^{P_{\lambda}} X^{U} @>{\subseteq}>> G \times^{P_{\lambda}} X 
@>{[g,x]\mapsto (g P_{\lambda},g x)}>\simeq> G/P_{\lambda} \times X @>{\pr_{X}}>> X
\end{CD}
$$
where the first map is a closed immersion and the second an isomorphism.
Since $G/P_{\lambda}$ is complete it follows that $\Theta$ is proper. Moreover, $\Theta$ is surjective, because every $G$-orbit meets $X^{U}$. We claim that $\Theta$ induces a bijection $G\times^{P_{\lambda}}(X^{U}\setminus X^{G}) \to X\setminus X^{G}$ which implies that $\Theta$ is birational. 
Indeed, if $x\in X^{U}\setminus X^{G}$, then $x \in O^{U}$ for a minimal orbit $O \subset X$. If $g x = g' x'$ for some $x'\in X^{U}$, $g'\in G$, then $x'\in O^{U}$, hence $x' = q x$ for some $q \in P_{\lambda}$, because the action of $P_{\lambda}$ on $O^{U}$ is transitive. It follows that $g^{-1}g' q\in G_{x}\subset P_{\lambda}$, hence $p:=g^{-1} g'\in P_{\lambda}$. Thus $[g',x']=[g p,x']=[g,p x'] = [g,x]$. 

It remains to see that the comorphism of $\Theta$ is an isomorphism on the global functions.
Let $K_{\lambda}$ be the kernel of the character $\lambda\colon P_{\lambda}\to\K^*$. Then $G/K_\lambda \simeq O_\lambda$, and the action of $P_{\lambda}$ on $G$ by right-multiplication induces an action of $\K^* = P_{\lambda}/K_{\lambda}$ on $G/K_{\lambda}$ by right-multiplication corresponding to the canonical action on $O_\lambda$. This gives the $G$-equivariant isomorphisms
\[
G\times^{P_{\lambda}}X^{U} \simto G/G_{\lambda}\times^{\Kst} X^{U} \simto O_\lambda \times^{\Kst} X^U,
\]
and the claim follows from (\ref{G-iso}).
\end{proof}

\begin{exa}
Let $X:=\overline{O_{\mu}} \subset V_{\mu}$ be the closure of the minimal orbit in $V_{\mu}$, and let $\mu=\ell\lambda$ where $\lambda$ is indivisible. Then $X^{U}=\K$, and from Proposition~\ref{small-var.thm}(\ref{G-iso}) we get an isomorphism
$$
\overline{O_{\lambda}} \times^{\Kst} \K \simeq X = \overline{O_{\mu}},
$$ 
where $\K^*$ acts on $\overline{O_{\lambda}}$ by the inverse of the canonical action, $(t,x)\mapsto\lambda(t^{-1})\cdot x$, and by the canonical action on $\K = \overline{O_\mu}$ which is the scalar multiplication with $\mu(t)$.
\end{exa}

The second statement  of Proposition~\ref{small-var.thm} says that a small $G$-variety $X$ can be {\it reconstructed} from the $\Kst$-variety $X^U$. In order to give a more precise statement we introduce the following notion. 
A $\Kst$-action on an affine variety $Y$ is called  {\it positively fix-pointed} if for  every $y \in Y$ the limit $\lim_{t\to 0}t y$ exists and is therefore a fixed point. 

For a fix-pointed $\Kst$-action on an irreducible affine variety $Y$ either the action is positively fix-pointed or the inverse action $(t,y)\mapsto t^{-1}y$ is positively fix-pointed. In fact, for any $y \in Y$ either $\lim_{t\to 0}t y$ or $\lim_{t\to \infty}t y$ exists. Embedding $Y$ equivariantly into a $\Kst$-module one sees that the subsets $Y_+:=\{y\in Y\mid \lim_{t\to 0}t y \text{ exists}\}$ and $Y_-:=\{y\in Y\mid \lim_{t\to \infty}t y \text{ exists}\}$ are closed. As $Y$ is irreducible, this yields the claim.
(The claim does not hold for connected $\Kst$-varieties, as the example of the union of the coordinate lines in the two-dimensional representation $t(x,y):=(t x,t^{-1} y)$ shows.)

\begin{rem}\label{K*K.rem}
A positively fix-pointed $\Kst$-action on $Y$ extends to an action of the multiplicative semigroup $(\K,\cdot)$, and the morphism $\K \times Y \to Y$, $(s,y)\mapsto s y$, induces an isomorphism $\K \times^{\Kst} Y \simto Y$. This follows from the commutative diagram
$$
\begin{CD}
Y @>{y\mapsto(1,y)}>> \K \times Y @>{(s,y)\mapsto s y}>> Y \\
@VV{\id_Y}V @VV{\pi}V @VV{\id_Y}V\\
Y @>>> \K\times^{\Kst} Y @>>> Y
\end{CD}
$$
where the compositions of the horizontal maps are  the identity.
\end{rem}
\begin{lem}\label{construction.prop} 
Let $Y$ be a positively fix-pointed affine $\Kst$-variety and let $\lambda\in\Lambda_G$ be indivisible. Consider the $\Kst$-action on $\Olamb\times Y$ given by $t(v,y):=(\lambda(t)^{-1}\cdot v, t y)$. 
Then 
$$
X:=\Olamb\times^{\Kst}Y=(\Olamb \times Y)\quot\Kst
$$
is a small $G$-variety of type $\lambda$ where the action of $G$ is induced by the action on $\Olamb$. Moreover, there is canonical $\Kst$-equivariant isomorphism $X^{U}\simto Y$.
\end{lem}
\begin{proof}
By definition, $X$ is an affine $G$-variety.
For $x = [v,y]\in \Olamb\times^{\Kst}Y$, $v\neq 0$, the $G$-orbit $G x \subset X$ is the image of $O_{\lambda} \times \{y\}$ in $X$, hence a minimal orbit of type $\lambda$ (Proposition~\ref{min-orbits.prop}(\ref{min-orbits.prop5})). As a consequence, $X$ is a small $G$-variety of type $\lambda$. Furthermore, since the canonical $\K^*$-action on $\Olbn$ commutes with the $G$-action, we have
$$
X^{U}=(\Olamb\times^{\Kst}Y)^{U} = \Olamb^{U}\times^{\Kst}Y \simeq \K\times^{\Kst} Y \simto Y,
$$
where the last morphism is given by $[t,y]\mapsto t y$ which is an isomorphism, as explained in Remark~\ref{K*K.rem} above.
\end{proof}

\begin{proof}[Proof of Corollary~\ref{equivalence.cor}] This corollary follows from Proposition~\ref{small-var.thm}(\ref{G-iso}) and Lemma~\ref{construction.prop}.
\end{proof}

\ps
\subsection{Smoothness of small \texorpdfstring{$G$}{G}-varieties}\label{smoothness.subsec}
Before describing the smoothness properties of small varieties, let us look at some examples. As before, $G$ is always a semisimple algebraic group.

\begin{rem}\label{small-G-module.rem}
Let $W$ be a $G$-module whose non-trivial orbits are all minimal. We claim that $W$ is a simple $G$-module and contains a single non-trivial orbit which is minimal. In particular, the highest weight of $W$ is indivisible.

Indeed, all minimal orbits in $W$ have the same type by Lemma~\ref{main.lem}(\ref{main.lem:1}) and therefore the same dimension $d>1$ by Remark~\ref{min-orbits.rem}(\ref{covering2}), and every minimal orbit meets $W^{U}$ in a punctured line, by Lemma~\ref{hw-orbits.lem}(\ref{U-fixed}). This implies that $\dim W = \dim W^{U}-1 + d$. Let $W = \bigoplus_{i=1}^{m}V_{i}$ be the decomposition into simple $G$-modules. 
Every factor contains a dense minimal orbit, all non-trivial orbits are minimal and hence of the same type by Lemma~\ref{main.lem}(\ref{main.lem:1}). By Lemma~\ref{hw-orbits.lem}(\ref{minimal}), a simple $G$-module contains at most one minimal orbit, hence
$\dim W = m d$. Since $\dim W^{U}= m$ we find $m d = m-1+d$, and so $m=1$.
\end{rem}

\begin{rem}\label{smooth-small-G-var.rem}
If a small $G$-variety $X$ is smooth and contains exactly one fixed point, then $X$ is a simple $G$-module $V_\lambda$ containing a dense minimal orbit, and $\lambda$ is indivisible. Indeed, smoothness and having exactly one fixed point imply by {\sc Luna}'s Slice theorem \cite[\S III.1 Corollaire 2]{Lu1973Slices-etales} that $X$ is a $G$-module, and the rest follows from the remark above. 
\end{rem}

\begin{exa}
Let $\Kn$ be the standard representation of $\SL_{n}$, and set  $W := (\Kn)^{\oplus m}$. Define
$Y:=\K e_1\oplus\K e_1 \oplus\cdots\oplus\K e_1 \subset W$ where $e_1=(1,0,\dots,0)$, and set $X:=\SL_{n} Y \subseteq W$. Since 
$Y$ is $B$-stable and closed it follows that $X$ is a closed and $\SL_{n}$-stable subvariety of $W$ with the following properties (cf. Example~\ref{SLn.exa}).
\be
\item 
$X$ contains a single closed $\SL_{n}$-orbit, namely the fixed point $\{0\}$.
\item 
Every nontrivial orbit $O \subset X$ is minimal of type $\eps_1$, and $\overline{O} \simeq \Kn$ as an $\SL_n$-variety. 
In particular, $X$ is a small $\SL_n$-variety.
\item 
Since $X^U=Y$ is normal (even smooth),
$X$ is also normal, by Lemma~\ref{main.lem}(\ref{main.lem:3}).
\ee
However, by Remark~\ref{smooth-small-G-var.rem} and (2), $X$ is not smooth if $m>1$.
\end{exa}

\begin{exa} Let $W:=\K^{3}$ be the $\Kst$-module with weights $(2,1,0)$, i.e. $t(x,y,z) := (t^{2}\cdot x, t \cdot y,z)$. 
The homogeneous function $f:=x z-y^{2}$ defines a normal $\Kst$-stable closed subvariety $Y=\VVV(f) \subset \K^{3}$ with an isolated singularity at $0$. The invariant $z$ defines the quotient $\pi=z\colon Y \to \K=Y\quot\Kst$.
The (reduced) fibers of $\pi$ are isomorphic to $\K$, but $Y$ is not a line bundle, because the zero fiber is not reduced. 
The action of $\Kst$ is given by $(t,s)\mapsto t\cdot s$ on the fibers over $\K\setminus\{0\}$ and by $(t,s)\mapsto t^{2}\cdot s$ on the zero fiber. In fact, the zero fiber contains the point $(1,0,0)$ which is fixed by $\{\pm 1\}$, but not by $\Kst$.

By Lemma~\ref{construction.prop}, $X:=\overline{O_{\eps_1}} \times^{\Kst} Y$ is a small $G$-variety and $X^U\simeq Y$, hence $X$ is normal (Lemma~\ref{main.lem}(\ref{main.lem:3})). Moreover, $X\quot G\simeq Y\quot\K^*=\K$ by Proposition~\ref{small-var.thm}(\ref{small-var.thm:2}).
All fibers of the quotient map $\pi\colon X \to X\quot G=\K$ different from the zero fiber are isomorphic to $\K^3=\overline{O_{\eps_{1}}}$, but $\pi^{-1}(0) \simeq \overline{O_{2\eps_{1}}}$.
\end{exa}
Concerning the smoothness of small $G$-varieties we have the following rather strong result, cf. Theorem~\ref{mainthm2}.

\begin{thm}\label{mainthm2b}
Let $X$ be an irreducible small $G$-variety of type $\lambda$, and consider the following statements.\renewcommand{\theenumi}{\roman{enumi}}
\be
\item\label{vectorbundle2}
The quotient $\pi\colon X \to X\quot G$ is a $G$-vector bundle with fiber $V_{\lambda}$.
\item\label{linebundle2}
$\Kst$ acts faithfully on $X^{U}$, the quotient $X^{U} \to X^{U}\quot\Kst$ is a line bundle, and $V_{\lambda}=\overline{O_{\lambda}}$.
\item\label{principal2}
The quotient $X^{U}\setminus X^{G}\to X^{U}\quot\Kst$ is a principal $\Kst$-bundle, and $V_{\lambda}=\overline{O_{\lambda}}$.
\item\label{disjoint2}
The closures of the minimal orbits of $X$ are smooth and pairwise disjoint.
\item\label{smooth2}
The quotient morphism $\pi\colon X \to X\quot G$ is smooth.
\ee
Then the assertions {\rm(\ref{vectorbundle2})} and {\rm(\ref{linebundle2})} are equivalent and imply {\rm(\ref{principal2})--(\ref{smooth2})}. If $X$ (or $X^U$) is normal, all assertions are equivalent.

Furthermore, $X$ is smooth if and only if $X\quot G$ is smooth and $\pi\colon X \to X\quot G$ is a $G$-vector bundle.
\end{thm}

\begin{proof}
(i) $\Rightarrow$ (ii): 
If $X \to X\quot G$ is a $G$-vector bundle with fiber $V_\lambda$, then the induced morphism $X^U \to X\quot G = X^U \quot \Kst$ is a subbundle with fiber $V_\lambda^U \simeq \K$, hence a line bundle.
\ps
(ii) $\Rightarrow$ (i):
Since $\Olb = V_\lambda$ we have a canonical isomorphism $V_\lambda\times^{\Kst} X^U \simto X$ where $\Kst$ acts by the inverse of the scalar multiplication on $V_\lambda$, see Proposition~\ref{small-var.thm}(\ref{G-iso})).
If $X^U \to X^U\quot \Kst$ is a line bundle, then it looks locally like $\K \times W \overset{\pr_W}{\longrightarrow} W$, and $\Kst$ acts by scalar multiplication on $\K$.
Hence $V_\lambda \times^{\Kst} X^U$ looks locally like 
$$
V_\lambda\times^{\Kst} (\K \times W) = (V_\lambda\times^{\Kst} \K) \times W \simeq V_\lambda \times W
$$
where we use the canonical isomorphism $V_\lambda\times^\Kst \K \simto V_\lambda$, $[v,s] \mapsto s\cdot v$, see Remark~\ref{K*K.rem}. This shows that $V_\lambda\times^{\Kst} X^U \simto X$ is a $G$-vector bundle over $X^U\quot\Kst = X \quot G$.
\ps
(i) $\Rightarrow$ (v): This is obvious.
\ps
(v) $\Rightarrow$ (iv): The (reduced) fibers of $\pi\colon X \to X\quot G$ are small $G$-varieties with a unique fixed point.  If such a fiber $F$ is smooth, then $F \simeq V_\lambda$ and $V_\lambda = \Olb$ by Remark~\ref{smooth-small-G-var.rem}.
\ps
(iv) $\Rightarrow$ (iii):
If the closure of a minimal orbit $O$ is smooth, then $O\simeq O_\lambda$ and $\Olb = V_\lambda$, again by Remark~\ref{smooth-small-G-var.rem}. It follows that the action of $\Kst$ on $X^U \setminus X^G$ is free and so $P:=X^U\setminus X^G \to X^U\quot \Kst$ is a principal $\Kst$-bundle. 
\ps
(iii) $\Rightarrow$ (ii) if $X^U$ is normal:
If $P:=X^U \setminus X^G \to X^U\quot \Kst$ is a principal $\Kst$-bundle and $L:=\K\times^\Kst P \to X^U\quot \Kst$ the associated line bundle, then there is a canonical morphism (see Remark~\ref{K*K.rem})
$$
\sigma\colon L=\K \times^\Kst (X^U\setminus X^G) \longrightarrow \K \times^\Kst X^U \simeq X^U.
$$
By construction, $\sigma$ is bijective, hence an isomorphism, because  $X^U$ is normal.
\ps
It remains to prove the last statement where one implication is clear. Assume that $X$ is smooth. Since the $G$-action is fix-pointed, it follows from \cite[(10.3)~Theorem]{BaHa1985Linearizing-certai} that $\pi\colon X \to X \quot G$ is a $G$-vector bundle.
\end{proof}

\par\bigskip
\section{Computations}
In this paragraph we calculate the invariants $m_G$, $d_G$ and $r_G$ which are defined for any simple algebraic group $G$ in the following way:
\begin{align*}
m_{G} &:=\min\{\dim O \mid O \text{ a minimal orbit}\},\\
d_{G} &:=\min\{\dim O \mid O \text{ a non-minimal quasi-affine non-trivial orbit}\},\\
r_{G} &:=\min\{\codim H \mid H\subsetneqq G \text{ reductive}\}.
\end{align*}

For any orbit $O$ in an affine $G$-variety $X$ we have $\dim O\geq\dim\Olam$ (Lemma~\ref{Omin.lem}). An orbit $O\simeq G/H$ with $H$ reductive is affine and thus cannot be minimal (Lemma~\ref{hw-orbits.lem}(\ref{hw-grading})). 

If $O \subset X$ is an orbit of dimension $m_G$, then it is either minimal or closed. In fact, if $O$ is not closed, then $\overline{O} \setminus O$ must be a fixed point since it cannot contain an orbit of positive dimension. This implies, by Proposition~\ref{min-orbits.prop}(\ref{min-orbits.prop1}), that $O$ is minimal. This shows that if $d_G=m_G$, then  $d_G=r_G$.
Hence we get 
\begin{equation}
\label{ineq.eqn}
r_G\geq d_G\geq m_G, \text{\ \ and \ } d_G > m_G \text{ in case } r_G > m_G.
\end{equation}
If $d_{G}>m_{G}$, then an irreducible $G$-variety $X$ of dimension $< d_{G}$ is small and  we can apply our results about small $G$-varieties.

For simplicity, we assume from now on that $G$ is {\it simply connected}.
\ps
\subsection{Notation} Let $G$ be a simple group.
As before, we fix a Borel subgroup $B\subset G$, a maximal torus $T\subset B$, and denote by $W:= \Norm_G(T)/T$ the Weyl group.
The monoid of dominant weights $\Lambda_G\subset X(T):=\Hom(T,\K^*)$ is freely generated by the fundamental weights $\omega_1,\dots,\omega_r$, i.e. $\Lambda_G=\bigoplus_{i=1}^r \NN\omega_i$ (see Section~\ref{min-orbits.subsec}). 
We denote by $\Phi=\Phi_G\subset X(T)$ the root system of $G$, by $\Phi^+=\Phi_G^+\subset \Phi$ the set of positive roots corresponding to $B$ and by $\Delta=\Delta_G\subset\Phi^+$ the set of simple roots. 
Furthermore, $\gg:=\Lie G$, $\bb:=\Lie B$ and $\hh:=\Lie T$ are the Lie algebras of $G$, $B$ and $T$, respectively,  $\gg_{\alpha}\subset\gg$ is the root subspace of $\alpha\in\Phi$ and $G_{\alpha}\subset G$ the corresponding root subgroup of $G$, isomorphic to $\K^+$.

The nodes of the Dynkin diagram of $G$ are the simple roots $\Delta_G$. We will use the  \name{Bourbaki}-labeling of the nodes:
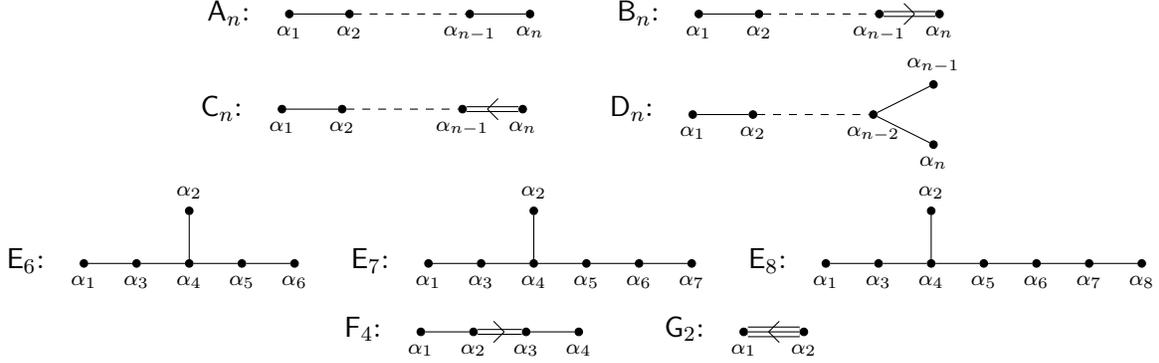
\begin{center}
$\An$: \DynkinAn \qquad $\Bn$: \DynkinBn \\

$\Cn$: \DynkinCn \qquad $\Dn$: \DynkinDn\\

$\Esix$: \DynkinEsix \quad $\Eseven$: \DynkinEseven \quad $\Eeight$: \DynkinEeight \\  

\hskip.5cm

$\Ffour$: \DynkinF \qquad $\Gtwo$: \DynkinG
\end{center}
We also have a  canonical bijection between the simple roots $\Delta_{G}=\{\alpha_1,\ldots,\alpha_r\}$ and the fundamental weights $\{\omega_1,\ldots,\omega_r\}$ induced by the Weyl group invariant scalar product  $(\cdot,\cdot)$ on $X(T)_\RR:=X(T)\otimes_{\ZZ}\RR$:
\[
\frac{2(\omega_i,\alpha_j)}{(\alpha_j,\alpha_j)}=\delta_{i j}.
\]
For any root $\alpha$ we denote by $\sigma_\alpha$ the corresponding reflection of $X(T)_\RR$:
\[
\sigma_\alpha(\beta) := \beta - \frac{2(\beta,\alpha)}{(\alpha,\alpha)}\alpha
\]
\ps
\subsection{Parabolic subgroups}\label{parabolic.subsec}
We now recall some classical facts about parabolic subgroups of $G$, cf. \cite[\S29-30]{Hu1975Linear-algebraic-g}. 
\ps

If $R\subset\Delta$ is a set of simple roots and $I:=\Delta\setminus R$ the complement we define $P(R):=B W_I B \subseteq G$ where $W_I\subseteq W$ is the subgroup generated by the reflections $\sigma_i$ corresponding to the elements of $I$. 
Any parabolic subgroup of $G$ containing $B$ is of the form $P(R)$, and we have $R\subseteq S$ if and only if $P(S)\subseteq P(R)$, with $R=S$ being equivalent to $P(R)=P(S)$. 
In particular, $P(\emptyset)=G$ and $P(\Delta)=B$, and the $P(\alpha_i):=P(\{\alpha_i\})$ are the maximal parabolic subgroups of $G$ containing $B$.

Consider the {\it Levi decomposition\/} $P(R)=L(R)\ltimes U(R)$, where $U(R)$ is the unipotent radical of $P(R)$ and $L(R)$ the Levi part of $P(R)$ containing $T$, i.e.
$L(G) = \Cent_G(Z)$ where $Z := \bigcap_{\alpha\in I}\ker\alpha \subseteq T$. In particular,  $L(R)$ is reductive, and so its derived subgroup $(L(R),L(R))$ is semisimple. 
The connected center $\Cent(L(R))^{\circ}$ of $L(R)$ is equal to $Z$, and hence 
\[
\dim \Cent(L(R))=\dim Z=\dim T-|\Delta\setminus I|=| I |.
\]
 It follows that
\begin{equation}
\label{L.eqn}
\dim(L(R),L(R))=\dim L(R)-\dim \Cent(L(R))=\dim L(R)-|I|.
\end{equation}
On the level of Lie algebras, we have 
$$
\textstyle
\pp(R):=\Lie P(R) =\hh\oplus\bigoplus_{\alpha\in\Theta}\gg_{\alpha} \text{ \ where \ } \Theta:=\Phi^+\cup(\Phi^-\cap\sum_{\alpha\in I}\ZZ\alpha).
$$
If $\Phi_I \subseteq \Phi$ is the subsystem generated by $I$ we get
\begin{gather*}
\pp(R)=\ll(R)\oplus\uu(R), \quad \ll(R):=\Lie L(R)=\hh\oplus\bigoplus_{\alpha\in\Phi_I} \gg_{\alpha},\\ 
\uu(R):=\Lie U(R)=\bigoplus_{\alpha\in\Phi^+\setminus\Phi_I}\gg_{\alpha}
\end{gather*}
Moreover, if $R\subseteq S$, we have $\ll(S)\subseteq\ll(R)$ and $\uu(S)\subseteq\uu(R)$. 
\ps
Setting $\uu^-:=\bigoplus_{\gg_\alpha \subset \uu}\gg_{-\alpha}$, we get $\dim\uu^-=\dim\uu$ and $\gg=\uu^-\oplus\pp=\uu^-\oplus\ll\oplus \uu$, hence
\begin{equation}
\label{G-P.eqn}
\dim\gg-\dim\ll=2\dim\uu =2\dim(\gg/\pp).
\end{equation}
Furthermore, (\ref{L.eqn}) and (\ref{G-P.eqn}) yield
\begin{equation}
\label{U.eqn}
\dim\uu=\frac{1}{2}(\dim \gg-\dim[\ll,\ll]-|I|).
\end{equation}
\begin{rem}\label{dynkin.rem}
The following facts will be important in our calculations of the invariants $m_G$ and $d_G$.
From the Dynkin diagram of $G$ we can read off the semisimple type of $L(R)$ by simply removing the nodes corresponding to the roots in $R$. Moreover,  we have $\gg_{\alpha} \subset \uu(R)$ for any $\alpha \in R$, and one can determine the irreducible representation $V(\alpha)\subseteq \uu(R)$ of $L(R)$ generated by $\gg_{\alpha}$, because $\gg_{\alpha}\subset V(\alpha)$ is the lowest weight space. 

It is easy to see that the Lie subalgebra generated by $V(\alpha)$ consists of all root spaces $\gg_{\beta}$ where $\beta$ is a positive root containing $\alpha$. In the special case $R = \{\alpha_i\}$ this implies that $\uu(\alpha_i)$ is equal to the Lie subalgebra generated by $V(\alpha_i)$.
\end{rem}

\ps
\subsection{The parabolic \texorpdfstring{$P_\lambda$}{P-lambda}}\label{P-lambda.subsec}
Recall that for a simple $G$-module  with highest weight $\lambda$ the subgroup 
\begin{equation*}
\label{parabolic.eqn}
P_\lambda:=\Norm_G(V_{\lambda}^U)=\Norm_G(\Olam^U)
\end{equation*}
is a parabolic subgroup of $G$, and $\lambda$ induces a character $\lambda\colon P_\lambda \to \Kst$. For $v \in V_\lambda^U$, $v\neq 0$, we have
$$
\Olam = G v \text{ \ and \ } G_v = \ker(\lambda\colon P_\lambda \to \Kst).
$$
In particular, $\dim\Olam=\codim_G P_{\lambda}+1$  (Lemma~\ref{hw-orbits.lem}). 
As above there is a well-defined Levi decomposition $P_\lambda = L_\lambda \ltimes U_\lambda$ where $T \subseteq L_\lambda$, which carries over to the Lie algebra:
$$
\pp_\lambda:=\Lie P_\lambda = \ll_\lambda \oplus \uu_\lambda, \quad \ll_\lambda:=\Lie L_\lambda, \  \uu_\lambda:=\Lie U_\lambda.
$$
Since $P_\lambda$ contains $B$ it is of the form $P(R)$ where the subset  $R\subseteq \Delta_G$ has the following description.

\begin{lem}\label{P-lambda.lem}
\strut
\be
\item
If $\lambda = \sum_{i=1}^r m_i\omega_i$, then
$P_{\lambda}= P(R)$ where $R :=\{\alpha_i\in\Delta_G\mid m_i=0\}$.
\item
We have $P_{\lambda}=P_{\lambda'}$ if the same $\omega_i$ appear in $\lambda$ and $\lambda'$. More generally, if every $\omega_i$ appearing in $\lambda'$ also appears in $\lambda$, then $P_\lambda \subseteq P_{\lambda'}$, $L_\lambda \subseteq L_{\lambda'}$ and $U_\lambda \subseteq U_{\lambda'}$. 
\item
$P_{k\omega_i}=P(\alpha_i)$ for all $k>0$, and these are the maximal parabolic subgroups of $G$ containing $B$.
\ee
\end{lem}
\begin{proof}
(1)
Set $I:=\{\alpha_i\in \Delta_G\mid m_i\neq 0\}$ and $R:=\Delta_G\setminus I$.  
Since $G_v^\circ = \ker\lambda^\circ = \bigcap_{\alpha\in I}\ker\alpha$ we get from Lemma~\ref{hw-orbits.lem}(\ref{parabolic}) that 
$P_\lambda = \Norm_G(G_v^\circ) = \Norm_G(\bigcap_{\alpha\in I}\ker\alpha) = P(R)$.
\ps
(2) follows from (1) and (3) from (2).
\end{proof}

\ps
\subsection{The invariant \texorpdfstring{$m_G$}{mG}}\label{m-G.subsec}
It follows from (\ref{parabolic.eqn}), that 
\[m_{G} = \min_{\lambda\in\Lambda_{G}}\dim\Olam = \min_{\lambda\in\Lambda_{G}}\codim_{G}P_{\lambda}+1.\]
So, it suffices to calculate
\begin{equation}
p_{G}:=\min\{\dim G/P \mid P \subsetneqq G \text{ a parabolic subgroup}\} = m_G-1.
\end{equation}
For this it is clearly sufficient to consider the maximal parabolic subgroups $P_{\omega_i}=P(\alpha_i)$. 

\begin{lem}\label{min-parabolics.lem}
The following table lists the invariants $m_G$ and $p_G$ for the simple groups $G$, the corresponding maximal parabolic subgroups $P_\omega$ as well as the dimensions of the fundamental representations $V_\omega$. The last column gives some indication about $\overline{O_\omega}$ where the null cone $\NNN_V$ appears only in case $\NNN_V \subsetneqq V$.
\ps
\begin{center}
\begin{tabular}{ c | c | c | c | c | c} 
$G$ & $m_G$ & $p_G$ &  \text{maximal}\ $P_{\omega}$ & $\dim V_{\omega}$ & $\overline{O_\omega}$  \\ 
\hline \hline
$\An,n\geq1$ & $n+1$ & $n$ & $P_{\omega_1},\ P_{\omega_n}$   & $n+1,n+1$ & $\K^{n+1},(\K^{n+1})^{\vee}$ \\
\hline 
$\Btwo$ & $4$ & $3$ & $P_{\omega_1},P_{\omega_2}$ & $5,4$ & $\NNN_{V_{\omega_1}},V_{\omega_2}$  \\
\hline
$\Bn,n\geq3$ & $2n$ & $2n-1$ & $P_{\omega_1}$ & $2n+1$ & $\NNN_{V_{\omega_1}}$ \\
\hline
$\Cn,n\geq3$ & $2n$ &  $2n-1$ & $P_{\omega_1}$ &  $2n$ & $V_{\omega_1}$ \\
\hline
$\Dfour$ & $7$ & $6$ & $P_{\omega_1},P_{\omega_3},P_{\omega_4}$ & $8,8,8$ & $\NNN_{V_{\omega_1}},\NNN_{V_{\omega_3}},\NNN_{V_{\omega_4}}$\\
\hline
$\Dn,n\geq5$ & $2n-1$ & $2n-2$ & $P_{\omega_1}$ & $2n$ & $\NNN_{V_{\omega_1}}$\\
\hline
$\Esix$ & $17$ & $16$ & $P_{\omega_1},P_{\omega_6}$ &  $27,27$ & $\subsetneqq\NNN_{V_{\omega_i}}$, $i=1,6$\\
\hline
$\Eseven$ & $28$ & $27$ & $P_{\omega_7}$ & $56$ & $\subsetneqq\NNN_{V_{\omega_7}}$\\
\hline
$\Eeight$ & $58$ & $57$ & $P_{\omega_8}$ & $248$ & $\subsetneqq \NNN_{\Lie\Eeight}$\\
\hline
$\Ffour$ & $16$ & $15$ & $P_{\omega_1},P_{\omega_4}$ & $52,26$ & $\subsetneqq \NNN_{V_{\omega_i}}$, $i=1,4$\\
\hline
$\Gtwo$ & $6$ & $5$ & $P_{\omega_1},P_{\omega_2}$ & $7,14$ & $\NNN_{V_{\omega_1}}$, $\subsetneqq\NNN_{\Lie\Gtwo}$
\end{tabular}
\ps
\captionof{table}{\label{tab3}Minimal orbits for the simple groups}
\end{center}
\end{lem}
\begin{proof}
By (\ref{G-P.eqn}), we have to find the maximal dimensional Levi subgroups $L_{\omega_i}$. For this it suffices to compute the maximum of $\dim (L_{\omega_i},L_{\omega_i})$. The Dynkin diagram of $\dd_i:=[\ll_{\omega_i},\ll_{\omega_i}]$ is obtained by removing the $i$-th node from the Dynkin diagram of $G$.  Now a short calculation in each case will give the possible  $\omega_i$  from which we will obtain columns 2--5 of Table~\ref{tab3}.
For the last column, we use that 
\[
O_{\omega_i}\subseteq\NNN_{V_{\omega_i}}\quad\text{where}\quad \dim O_{\omega_i}=\codim P_{\omega_i}+1=\frac{1}{2}(\dim G-\dim\dd_i-1)+1,
\]
see Lemma~\ref{hw-orbits.lem}(\ref{normal}) and (\ref{parabolic}) and Section~\ref{parabolic.subsec} above.

We now apply the above strategy to each simple group $G$. In each case, $\dim\dd_i$ turns out to be quadratic in $i$ and achieves its minimum on the interval $[1,n]$. Hence, if $\dd_i$ is of maximal dimension, then $i$ is either $1$ or $n$. 
\ps
{\bf(row $\An$)} 
For $i=1,\dots,n$, we obtain $\dd_i=\sL_i\oplus\sL_{n-i+1}$. It is of maximal dimension for $i=1,n$. Furthermore, $V_{\omega_1}=\K^{n+1}$ and $V_{\omega_n}= (\K^{n+1})^\vee$ are the the standard representation of $\SL_{n+1}$ and its dual which yields $\codim P_{\omega_1}=\codim P_{\omega_2}=n$ and $\overline{O_{\omega_i}}=V_{\omega_i}$. 
\ps
{\bf(rows $\Btwo=\C_2$ and $\Bn$)}
For $i=1,\dots,n$, we obtain $\dd_i=\sL_i\oplus \sO_{2(n-i)+1}$. It is of maximal dimension for $i=1$ if $n\geq3$ and for $i=1,2$ if $n=2$.  
Furthermore, $V_{\omega_1}=\K^{2n+1}$ is the standard representation of $\SO_{2n+1}$, and the quotient $V_{\omega_1}\quot\SO_{2n+1}\simeq\K$ is given by the invariant quadratic form. In particular, $\dim\NNN_{V_{\omega_1}} = 2n$, and $\SO_{2n+1}$ acts transitively on the isotropic vectors $\NNN_{V_{\omega_1}}\setminus\{0\}$, hence $\overline{O_{\omega_1}} = \NNN_{V_{\omega_1}}$. This gives the row  $\Bn$, $n\geq3$, and half of the row $\Btwo$. 

If $n=2$, then $V_{\omega_2}$ is the standard representation $\K^4$ of $\Sp_4$, hence $O_{\omega_2} = \K^4 \setminus\{0\}$, giving the other part of the row $\Btwo$.
\ps
{\bf(row $\Cn$)} Here we get $\dd_i=\sL_i\oplus \sP_{2(n-i)}$ which is of maximal dimension for $i=1$.
Furthermore, $V_{\omega_1}=\K^{2n}$ is the standard representation of $\Sp_{2n}$, and $\overline{O_{\omega_1}}=V_{\omega_1}$, hence $m_{\Sp_{2n}} = 2n$.
\ps
{\bf(rows $\Dfour$ and $\Dn$)} 
For $i=1,\dots,n-3$, we get $\dd_i=\sL_i\oplus \sO_{2(n-i)}$. Moreover, $\dd_{n-2}=\sL_{n-2}\oplus\sL_2\oplus\sL_2 $ and $\dd_{n-1} = \dd_{n}=\sL_n$. They are maximal dimensional for $i=1$ if $n\geq5$ and for  $i=1,3$ and $4$ if $n=4$.
Furthermore, $V_{\omega_1}=\K^{2n}$ is the standard representation of $\SO_{2n}$, and we get the claim for $\Dn$, $n\geq5$ and for $V_{\omega_1}$ in case $n=4$. 
In this case, $V_{\omega_3}$ and $V_{\omega_4}$ are conjugate to the standard representation $V_{\omega_1}=\K^8$ by an outer automorphism of $\Dfour$. For the  standard representation $V$ we have $V\quot G = \K$, given by the invariant quadratic form, and the nullcone consists of two orbits, $\{0\}$ and the minimal orbit of nonzero isotropic vectors.
\ps
{\bf(row $\Esix$)}
Here we find $\dd_1=\dd_6=\sO_{10}$, $\dd_2=\sL_6$, $\dd_3=\dd_5=\sL_2\oplus\sL_5$, $\dd_4=\sL_3\oplus\sL_2\oplus\sL_3$. 
The maximal dimension is reached for $i=1,6$, and we get $p_{\Esix}=16$. The representations $V_{\omega_1}$ and $V_{\omega_6}$ of dimension $27$ are dual to each other. The quotient $V_{\omega_1}\quot \Esix=\K$ is given by the cubic invariant of $V_{\omega_1}$ (see \cite[Table~5b]{Sc1978Representations-of}), and so $\dim\NNN_{V_{\omega_1}}=26$. It follows that  $\overline{O_{\omega_i}}\subsetneqq\NNN_{V_{\omega_i}}$, $i=1,6$.
\ps
{\bf(row $\Eseven$)} We have $\dd_1=\sO_{12}$, $\dd_2=\sL_7$, $\dd_3=\sL_2\oplus\sL_6$, $\dd_4=\sL_3\oplus\sL_2\oplus\sL_4$, $\dd_5=\sL_5\oplus\sL_3$, $\dd_6=\sO_{10}\oplus\sL_2$, $\dd_7=\Esix$.
The maximal dimension is reached for $i=7$, and we get $p_{\Eseven}=27$. 
We have $\dim V_{\omega_7}=56$ and $\dim V_{\omega_7}\quot \Eseven=1$ (see for instance \cite[Table 5a]{Sc1978Representations-of}), hence 
$\NNN_{V_{\omega_7}}\subset V_{\omega_7}$ has codimension 1 and so $\overline{O_{\omega_7}}\subsetneqq\NNN_{V_{\omega_7}}$. 
\ps
{\bf(row $\Eeight$)} Here we obtain $\dd_1=\sO_{14}$, $\dd_2=\sL_8$, $\dd_3=\sL_2\oplus\sL_7$, $\dd_4=\sL_3\oplus\sL_2\oplus\sL_5$, $\dd_5=\sL_5\oplus\sL_4$, $\dd_6=\sO_{10}\oplus\sL_3$, $\dd_7=\Esix\oplus\sL_2$, $\dd_8=\Eseven$. 
The maximal dimension is reached for $i=8$, and we get $p_{\Eeight}=57$. Moreover, $V_{\omega_8}$ is the adjoint representation of dimension $248$, $\dim\NNN_{V_{\omega_8}}=\dim\Eeight-\rk\Eeight=240$ (\cite[Example 2.1]{KrWa2006On-the-nullcone-of}), and thus $\overline{O_{\omega_8}}\subsetneqq \NNN_{V_{\omega_8}}$. 
\ps
{\bf(row $\Ffour$)} We have $\dd_1=\sP_6$, $\dd_2=\dd_3=\sL_2\oplus\sL_3$, $\dd_4=\sO_7$, and the maximal dimension is reached for $i=1,4$. This yields $p_{\Ffour}=15$. Moreover, $V_{\omega_1}$ is the adjoint representation of dimension $52$, thus $\dim\NNN_{V_{\omega_1}}=\dim\Ffour-\rk\Ffour=48$, and so $\overline{O_{\omega_1}}\subsetneqq\NNN_{V_{\omega_1}}$. The other representation $V_{\omega_4}$ has dimension $26$, is cofree and $\dim V_{\omega_4}\quot G = 2$, cf. \cite[Table 5a]{Sc1978Representations-of}. Hence $\dim \NNN_{\omega_4} = 24$ 
and thus $\overline{O_{\omega_4}}\subsetneqq \NNN_{V_{\omega_4}}$.

{\bf(row $\Gtwo$)} We have $\dd_1=\dd_2=\sL_2$ and hence $p_{\Gtwo}=5$ and $\dim O_{\omega_i}=6$. Furthermore $\dim V_{\omega_1}=7$, $\dim V_{\omega_2}=14$ and $\Gtwo$ preserves a quadratic form on $V_{\omega_1}$ (see \cite[\S22.3]{FuHa1991Representation-the}) which implies that $\overline{O_{\omega_1}}=\NNN_{V_{\omega_1}}$. Moreover, $V_{\omega_2}$ is the adjoint representation, $\dim\NNN_{V_{\omega_2}}=\dim\Gtwo-\rk\Gtwo=12$  and hence $\overline{O_{\omega_2}}\subsetneqq\NNN_{V_{\omega_2}}$. 
\end{proof}

\begin{rem}\label{smooth.rem}
The lemma above has the following consequence. Let $G$ be a simple group. {\it If $\Olb$ is smooth, then we are in one of the following cases:
\be
\item 
$G = \SL_n$ and $\lambda = \omega_1$ or $\lambda=\omega_n$, i.e. $\Olb$ is the standard representation or its dual.
\item
$G = \Sp_{2n}$ and $\lambda = \omega_1$, i.e. $\Olb$ is the standard representation.
\ee}
In fact, if $\Olb$ is smooth, then $\Olb = V_\lambda$ by Lemma~\ref{hw-orbits.lem}(\ref{sing}), and the claim follows from the last column of Table~\ref{tab3} in Lemma~\ref{min-parabolics.lem}.
\end{rem}

Now we can prove the first theorem from the introduction.

\begin{proof}[Proof of Theorem~\ref{mainthm.cor}]
Theorem~\ref{mainthm2} implies that $X\to X\quot G$ is a $G$-vector bundle with fiber $V_{\lambda}$, where $\lambda$ is the type of $X$, and the minimal orbits are smooth. This means that $\Olb = V_{\lambda}$, by Lemma~\ref{hw-orbits.lem}(\ref{sing}),  and the claim follows from Remark~\ref{smooth.rem} above.
\end{proof}

\ps
\subsection{The invariant \texorpdfstring{$r_G$}{rG}}\label{r-G.subsec}
In this section, we compute the invariant 
\[
r_{G} =\min\{\codim_G H \mid H\subsetneqq G \text{ reductive}\}
\] 
which is the minimal dimension of a nontrivial affine $G$-orbit. These orbits are never minimal orbits, by Lemma~\ref{hw-orbits.lem}(\ref{hw-grading}).

\begin{lem}\label{rG.lem} The following table lists the types of the proper reductive subgroups $H$ of the simple groups $G$ of maximal dimension, their codimension $r_G=\codim_G H$ and the invariant $m_G$ from Lemma~\ref{min-parabolics.lem}. (In the table  $\Tone$ denotes the $1$-dimensional torus.)  
\ps
\begin{center}
\begin{tabular}{ c || c | c | c | c | c | c | c | c | c | c } 
$G$ &  $\mathsf{A}_{3}$ & $\An, n\neq 3$ &$\Bn$ &$\Cn, n\geq 3$ & $\Dn, n\geq 4$ & $\Esix $ & $\Eseven$ & $\Eeight$ & $\Ffour$ & $\Gtwo$ \\ 
\hline\hline
$H$ &  $\Btwo$ & $ \A_{n-1}\times \Tone$ & $\Dn$ & $ \mathsf{C}_{n-1}\times \Aone$ & 
$ \mathsf{B}_{n-1}$ & $\Ffour$ & $\Esix\times \Tone$ & $\Eseven\times\Aone$ & $\mathsf{B}_{4}$ & $\mathsf{A}_{2}$ \\ \hline
$r_G$ & $5$ & $2n$ & $2n$ & $4(n-1)$ &  $ 2n-1$ & $26$ & $54$ & $112$ & $16$ & $6$ \\ \hline
$m_G$ & $4$ & $n+1$ & $2n$ & $2n$ &  $ 2n-1$ & $17$ & $28$ & $58$ & $16$ & $6$ 
\end{tabular}
\ps
\captionof{table}{Maximal reductive subgroups of simple groups}\label{tab4}
\end{center}
\end{lem}

\begin{proof}
The classification of maximal subalgebras $\hh$ of a simple Lie algebra $\gg$ is due to \name{Dynkin}, see  \cite{Dy1952Maximal-subgroups-,Dy1952Semisimple-subalge}. His results are reformulated in  \cite[chap.6, \S 1 and \S 3]{GoOnVi1994Structure-of-Lie-g}. 

(a) If $\hh$ is maximal reductive of maximal rank  $\ell:=\rk \gg$, then the classification is given in \cite[Corollary to Theorem~1.2, p.186]{GoOnVi1994Structure-of-Lie-g} (the results are listed in Tables 5 and 6, pp.~234--235). From these tables one gets the following candidates for reductive subalgebras of minimal codimension.\footnote{One has to be careful since the tables contain several errors.}
\begin{center}
{\footnotesize
\begin{tabular}{ c || c | c | c | c | c | c | c | c | c | c } 
$\gg$ &  $\An, n\geq 1$ &$\Bn, n\geq 2$ &$\Cn, n\geq 3$ & $\Dn, n\geq 4$ & $\Esix $ & $\Eseven$ & $\Eeight$ & $\Ffour$ & $\Gtwo$ \\ 
\hline\hline
$\hh$ & $ \A_{n-1}\times \Tone$ & $\Dn$ & $ \mathsf{C}_{n-1}\times \Aone$ & 
$ \mathsf{D}_{n-1}\times\Tone$ & $\A_5\times\Aone$ & $\Esix\times \Tone$ & $\Eseven\times\Aone$ & $\mathsf{B}_{4}$ & $\mathsf{A}_{2}$ \\ \hline
$\codim$ & $2n$ & $2n$ & $4(n-1)$ &  $ 4(n-1)$ & $40$ & $54$ & $112$ & $16$ & $6$ 
\end{tabular}
}
\end{center}

(b) If $\hh \subset \gg$ is a maximal subalgebra, then it is either semisimple or parabolic, \cite[Theorem~1.8]{GoOnVi1994Structure-of-Lie-g}. Since the Levi parts of the parabolic subalgebras have maximal rank the second case does not produce any new candidate. It is therefore sufficient to look at the maximal semisimple subalgebras. 

For the exceptional groups $G$ the classification is given in \cite[Theorem~3.4]{GoOnVi1994Structure-of-Lie-g}, and one finds one new case, namely $\Ffour\subset \Esix$ which has codimension 26. Thus the claim is proved for the exceptional groups. 
\ps
(c) From now on $G$ is a classical group and we can use \cite[Theorems~3.1--3.3]{GoOnVi1994Structure-of-Lie-g}. From the first two theorems one finds the new candidates $\B_{n-1} \subset \Dn$ of codimension $2n-1$, including $\Btwo \subset \Athree$ of codimension $5$. This gives the following table.

\begin{center}
{
\begin{tabular}{ c || c | c | c | c  } 
$G$ &  $\SL_4 $ & $\SL_n, n\neq 4$ & $\SO_n, n\geq 4$ & $\Sp_{n}, n=2m\geq 4$ \\ 
\hline\hline
$H$ & $\Sp_4$ & $\GL_{n-1}$ & $\SO_{n-1}$ & $\Sp_{n-2}\times\SLtwo$\\
\hline
$\dim H$ & $10$ & $(n-1)^2$ & $\frac{n^2-3n}{2}+1$ & $\frac{n^2-3n}{2}+4$ \\
\hline
$c_G:=\codim_G H$ & $5$ & $2n-2$ & $n-1$ & $2n-4$\\
\end{tabular}
}
\end{center}

Our claim is that $c_G = r_G$, i.e. that we have found the minimal codimensions of reductive subgroups of the classical groups. In order to prove this we have to show that \cite[Theorems~3.3]{GoOnVi1994Structure-of-Lie-g} does not give any reductive subgroup of smaller codimension:
\ps
{\it If $H \subsetneqq G$ is an irreducible simple subgroup of a classical group $G = \SL_n, \SO_n, \Sp_{n}$, then $\codim_G H\geq c_G$} (irreducible means that the representation of $H \into \GL_n$ is irreducible).
\ps
Now the table above implies the following. Assume $n\geq 4$. If $H \subsetneqq G \subset \GL_n$ is an irreducible subgroup of a classical group $G = \SL_n, \SO_n, \Sp_n$ and $\dim H < d(n):=\frac{n^2-3n}{2}+1$, then $\codim_G H > c_G$, and so $H$ can be omitted.

The following table contains the minimal dimensions of irreducible representations of the simply connected exceptional groups. They have been calculated using \cite[Exercise~24.9]{FuHa1991Representation-the} which says that one has only to consider the fundamental representations.
\begin{center}
{
\begin{tabular}{ c || c | c | c | c | c } 
$H$ &  $\Esix $ & $\Eseven$ & $\Eeight$ & $\Ffour$ & $\Gtwo$ \\ 
\hline\hline
$\dim H$ & $78$ & $133$ & $248$ & $52$ & $14$ \\
\hline
$\lambda$ &  $\omega_1,\omega_6$ & $\omega_7$ & $\omega_8$ & $\omega_4$ & $\omega_1$ \\ 
\hline
$n=\dim V_\lambda$ & $27$ & $56$ & $248$ & $26$ & $7$ 
\end{tabular}
}
\end{center}
In all cases we have $\dim H < d(n) = \frac{n^2-3n}{2}+1$, so that $\codim_G H > c_G$   for an exceptional group $H$.
\ps
(d) It remains to consider the simple subgroups $H \subsetneqq G$  of classical type where $G = \SL_n, \SO_n, \Sp_n$.
\ps
(d$_1$) The irreducible representations $H \to \SL_n$  of minimal dimension of a group $H$ of classical type are given by the following table. It is obtained by using again the fact that one has only to consider the fundamental representations, see \cite[Exercise~24.9]{FuHa1991Representation-the}.
\begin{center}
{
\begin{tabular}{ c || c | c | c | c | c | c } 
$H$ &  $\A_\ell$ & $\Btwo$ & $\B_\ell, \ell\geq 3$ & $\C_\ell, \ell\geq 3$ & $\Dfour$ & $\D_\ell, \ell\geq 5$ \\ 
\hline\hline
$\dim H$ & $\ell(\ell+2)$ & $10$ & $\ell(2\ell+1)$ & $\ell(2\ell+1)$ & $28$  & $\ell(2\ell-1)$ \\
\hline
$\lambda$ & $\omega_1,\omega_\ell$ & $\omega_2$ & $\omega_1$ & $\omega_1$ & 
$\omega_1,\omega_3,\omega_4$ & $\omega_1$  \\ 
\hline
$n=\dim V_\lambda$ & $\ell+1$ & $4$ & $2\ell+1$ & $2\ell$ & $8$ & $2\ell$\\ 
\end{tabular}
}
\end{center}
They correspond to the standard representations $\SL_n\subset \GL_n$, $\SO_n\subset\GL_n$ and $\Sp_n\subset \GL_n$, except for $\Btwo = \Ctwo$ where it is $\Sp_4 \subset \GL_4$. If $H$ is not of type $\A$ we have  $\codim_{\SL_n} H > c_{\SL_n} = 2n-2$ except for type $\Btwo$ where $\codim_{\SL_4} \Sp_4 = 5 = c_{\SL_4}$. Moreover, if $\SL_k \to \SL_n$ is not an isomorphism, then $k<n$ and $\codim_{\SL_n} \SL_k > c_{\SL_n}$. 
\ps
(d$_2$) Next we consider irreducible orthogonal representations $\rho\colon H \to \SO_n$ for $H$ of classical type where $n\geq 5$. 
If $H$ is a candidate not already in (a), then $\rk H<\rk\SO_n$, and one calculates straight forwardly that $\codim_{\SO_n}H>c_{\SO_n} = n-1$.
\ps
(d$_3$) Finally, we consider irreducible symplectic representations $\rho\colon H \to \Sp_{2m}$ for $H$ of classical type where $m\geq 2$. 
As above, if $H$ is a candidate not already in (a), then $\rk H<\rk\Sp_{2m}=m$. Again, an easy calculation shows that  $\codim_{\Sp_{2m}}H>c_{\Sp_{2m}} = 4m-4$.
\end{proof}

\ps
\subsection{The invariant \texorpdfstring{$d_G$}{d G}}\label{d-G.subsec}
In this section, we compute the invariant 
\[
d_G=\min\{\dim O \mid \text{$O$ non-minimal quasi-affine non-trivial orbit}\}.
\]
Formula \eqref{ineq.eqn} shows that $r_G\geq d_G \geq m_G$ and that $d_G>m_G$ in case $r_G>m_G$. Comparing the values of $r_G$ and $m_G$ in Table~\ref{tab4} of Lemma~\ref{rG.lem} we get the following result.

\begin{lem}\label{rG=mG.lem}
Let $G$ be simple and simply connected. If $r_G = d_G =m_G$, then we are in one of the following cases.
\be
\item $G$ is of type $\Aone$ and $d_G = 2$;
\item $G$ is of type $\Bn$ and $d_G= 2n$;
\item $G$ is of type $\Dn$, $n\geq 4$, and $d_G = 2n-1$;
\item $G$ is of type $\Ffour$ and $d_G = 16$;
\item $G$ is of type $\Gtwo$ and $d_G = 6$.
\ee
In all other cases we have $r_G\geq d_G>m_G$.
\end{lem}

\begin{prop}\label{dG.prop}
The following table lists the invariants $r_G$, $d_G$ and $m_G$ for the simply connected simple algebraic groups $G$.
\begin{center}
\begin{tabular}{ c || c | c | c | c | c | c | c | c | c | c | c | c } 
$G$ &  $\Aone$ & $\Atwo$ & $\mathsf{A}_{3}$ & $\An, n\geq 4$ &$\Bn$ &$\Cn, n\geq 3$ & $\Dn, n\geq 4$ & 
$\Esix $ & $\Eseven$ & $\Eeight$ & $\Ffour$ & $\Gtwo$ \\ 
\hline\hline
$r_G$ & $2$ & $4$ & $5$ & $2n$ & $2n$ & $4(n-1)$ &  $ 2n-1$ & $26$ & $54$ & $112$ & $16$ & $6$ \\ \hline
$d_G$ & $2$ & $4$ & $5$ & $2n$ & $2n$ & $4(n-1)$ &  $ 2n-1$ & $26$ & $45$ & $86$ & $16$ & $6$ \\ \hline
$m_G$ & $2$ & $3$ & $4$ & $n+1$ & $2n$ & $2n$ &  $ 2n-1$ & $17$ & $28$ & $58$ & $16$ & $6$ 
\end{tabular}
\pmed
\captionof{table}{\label{tab5}The invariants $r_G$, $d_G$ and $m_G$ for the simple groups}
\end{center}
\end{prop}

The first and last row of Table~\ref{tab5} are the rows from Table~\ref{tab3} and Table~\ref{tab4}. 
We have seen above that for $r_G\leq m_G+1$ we have $d_G = r_G$, because $r_G>m_G$ implies that $d_G > m_G$. Thus the only cases to be considered are $\An$ for $n\geq 4$, $\Cn$ for $n\geq 3$ and $\Esix$, $\Eseven$, $\Eeight$. 

\ps
We have seen in Section~\ref{P-lambda.subsec} that 
for a dominant weight $\lambda\in\Lambda_G$ the corresponding parabolic subgroup $P_\lambda \subset G$ and its Lie algebra $\pp_\lambda$ have  well-defined Levi decompositions $P_\lambda = L_\lambda \ltimes U_\lambda$ where $T \subseteq L_\lambda$ and $\pp_\lambda:=\Lie P_\lambda = \ll_\lambda \oplus \uu_\lambda$. In addition, we define the closed subgroup $P_{(\lambda)}:=\ker(\lambda\colon P_\lambda \to \Kst)$ which has the Levi decomposition $P_{(\lambda)} = L_{(\lambda)} \ltimes U_\lambda$, $L_{(\lambda)}:=\ker(\lambda\colon L_\lambda \to \Kst)$, and its Lie algebra 
$$
\pp_{(\lambda)}:=\Lie P_{(\lambda)} = \ll_{(\lambda)} \oplus \uu_\lambda, \ \ 
\ll_{(\lambda)} := \Lie L_{(\lambda)} = \ker (d\lambda\colon \ll_\lambda \to \K).
$$
By construction, the semisimple Lie algebra $[\ll_\lambda,\ll_\lambda]$ is contained in $\ll_{(\lambda)}$, and they are equal in case $\lambda$ is a fundamental weight $\omega_i$.
Note also that $P_{(\lambda)} = G_v$ for $v \in V_\lambda^U$, $v\neq 0$, see Section~\ref{ass-parabolic.subsec}.
For an affine $G$-variety $X$ and $x \in X$ we set $\gg_x:=\Lie G_x$ and denote by $\nn_x \subseteq \gg_x$ the nilradical of $\gg_x$. 
\ps
The method for proving Proposition~\ref{dG.prop} was communicated to us by \name{Oksana Yakimova} who also worked out the result for the symplectic groups and for $\Esix$. It is based on the following lemma which is a translation of a fundamental result of \name{Sukhanov}, see \cite[Theorem~1]{Su1990Description-of-the}.

\begin{lem}\label{sukhanov.lem}
Let $O$ be a quasi-affine $G$-orbit. Then there exist $\lambda\in\Lambda_{G}$ and $x \in O$ such that $\gg_{x} \subseteq \pp_{(\lambda)}$ and $\uu_{x}\subseteq \uu_{\lambda}$. 
In particular, we get an embedding $\ll_x:=\gg_{x}/\nn_{x} \into \ll_{(\lambda)}=\pp_{(\lambda)}/\uu_{\lambda}$. 
If $O$ is not a minimal orbit, then $\dim O \geq \dim \uu_\lambda + 2$.
\end{lem}
\begin{proof} In \name{Sukhanov}'s paper a subgroup $L\subset G$ is called {\it observable} if $G/L$ is quasi-affine. Now \cite[Theorem~1]{Su1990Description-of-the} implies that such an $L$ is {\it subparabolic} which means that there is an embedding $L \into Q$ such that $L_u \into Q_u$ where $Q$ is the isotropy group of a highest weight vector. Translating this into the language of Lie algebras we get the first part of the lemma.
\ps
For the second part, we note that $\gg_x\subsetneq \pp_{(\lambda)}$, so that
\[
\dim O=\codim_{\gg}\gg_x \geq \codim_{\gg}\pp_{(\lambda)}+ 1=\dim \Olam+1=\codim_{\gg}\pp_{\lambda}+2=\dim\uu_{\lambda}+2,
\] 
and the claim follows.
\end{proof}

The strategy of the proof of Proposition~\ref{dG.prop} is the following. Let $O = G x \subset X$ be a non-minimal nontrivial orbit, and
consider an embedding $\gg_x \into \pp_{(\lambda)}$ given by the lemma above.
\ps
(1) Since $O$ is not minimal, we have $\dim O \geq \dim\uu_{\lambda}+2$. Thus one has only to consider those $\pp_{\lambda}$ 
with $\dim\uu_{\lambda} + 2 < d_{G}$. 
 For this one first calculates $\dim \uu_{\omega_i}$, $i=1,\ldots,n$, and then uses that $\dim \uu_\lambda\geq \dim\uu_{\omega_i}$ for all $i$ such that $\omega_i$ appears in $\lambda$, see Lemma~\ref{P-lambda.lem}. 

{\it It turns out that in all cases the remaining $\lambda$ are fundamental weights, and we are left to study some of the embeddings $\gg_x \into \pp_{(\omega_i)}$.}
\ps
(2) Since $O$ is not minimal, the embedding $\gg_x \into \pp_{(\omega_i)}$ is strict, hence one of the two inclusions $\nn_{x}\subseteq \uu_{\omega_i}$ or $\ll_{x} \subseteq \ll_{(\omega_i)}$ has to be strict. 
\ps
(2a) If $\ll_{x} = \ll_{(\omega_i)}$, then $\nn_{x}$ must be a strict $\ll_{(\omega_i)}$-submodule of $\uu_{\omega_i}$. As we have seen in Remark~\ref{dynkin.rem}, $\nn_x$ cannot contain the simple module $V(\alpha_i)$, hence the codimension of $\nn_x$ 
in $\uu_{\omega_i}$ is at least $\dim V(\alpha_i)$.
\ps
(2b) If $\ll_{x}\subsetneqq \ll_{(\omega_i)}$, then $L_{x}^\circ\subsetneqq L_{(\omega_i)}^\circ$ is a proper reductive subgroup of the semisimple group $L_{(\omega_i)}$ , and the  codimension  can be estimated using the values of $r_{H}$ given in our tables.

\begin{rem}\label{quasiaffine_orbits.rem}
In the cases of $\Eseven$ and $\Eeight$ we will have to construct  quasi-affine orbits of a given dimension. For this we will use the following result. 
\ps
{\it Let $H \subset G$ be a closed subgroup. If the character group $X(H)$ is trivial, then there is a $G$-module $V$ and a $v \in V$ such that $G_v = H$.} 
\newline
In fact, there is a $G$-module $V$ and a line $L=\K v \subset V$ such that $H = \Norm_G(L)$ \cite[Chapter~II, Theorem~5.1]{Bo1991Linear-algebraic-g}. Since $H$ has no characters, it acts trivially on $L$ and so $H = G_v$.
\end{rem}

\ps
\subsubsection{The type \texorpdfstring{$\A_{n}$, $n\geq 4$}{An}}

\DynkinAn

Suppose that $G = \SL_{n+1}$ and $\gg = \sL_{n+1}$ with $n\geq 4$ and let $O$ be a non-minimal and non-trivial quasi-affine orbit. We have to show that $\dim O \geq 2n$. We have seen above that it suffices to consider those embeddings $\gg_{x}\subset\pp_{(\lambda)}$ where $\dim\uu_{\lambda} < 2n-2$. We have
$$
\pp_{(\omega_i)} = \sL_i \oplus \sL_{n+1-i}  \oplus \uu_{\omega_i},
$$
and so $\dim\uu_{\omega_{i}}=i(n+1-i)$ which is greater or equal than $2n-2$ for $i\neq 1,n$. Moreover, 
we have $\pp_{\omega_1+\omega_n} = (\sL_{n-1}\oplus \K^2) \oplus \uu_{\omega_1+\omega_n}$ and hence $\pp_{(\omega_1+\omega_2)}=(\sL_{n-1}\oplus \K) \oplus \uu_{\omega_1+\omega_n}=\gl_{n-1}\oplus \uu_{\omega_1+\omega_n}$,
which implies that  $\dim \uu_{\omega_1+\omega_n}=2n-1$. Thus, by (\ref{U.eqn}), the only cases to consider are 
$\lambda=\omega_{1}$ and $\lambda=\omega_{n}$. 
\ps
If $\gg_{x}\subsetneqq \pp_{(\omega_{1})} = \sL_{n}\oplus (\K^{n})^{\perp}$, then
 we either have $\nn_{x} = (0)$ or $\ll_{x}\subsetneqq \sL_{n}$. In the first case we get $\dim O= \codim \gg_{x}= \codim \gg_{\omega_{1}}+n = 2n+1$. In the second case, $\ll$ is a reductive Lie-subalgebra of $\sL_{n}$ and thus has codimension at least $r_{\A_{n-1}} = 2(n-1)$. Hence, $\dim O= \codim \gg_{x}\geq \codim\gg_{\omega_{1}}+2(n-1)=3n-1 >2n$.
\ps
The other case $\lambda = \omega_{n}$ is similar.

\begin{rem} 
We have just shown that for $n\neq 3$ any quasi-affine $\SL_{n+1}$-orbit of dimension $<2n$ is minimal. Furthermore, we have $\dim O_{\lambda}=\dim\uu_{\lambda}+1$ by (\ref{G-P.eqn}) and (\ref{P-lambda.subsec}). 
In particular, since $\dim O_{\omega_i}= i(n+1-i)+1$ (see above), we get
\[
\dim O_{\omega_1}=\dim O_{\omega_n}=n+1,\quad \dim O_{\omega_2}=\dim O_{\omega_{n-1}}=2n-1,
\]
and all other minimal orbits have dimension $\geq2n$. 

Note that $r_{\An} = 2n$  appears as dimension of the affine orbit $ \SL_{n+1}/\GL_{n}$ as well as of the minimal orbit $O_{\omega_1+\omega_2}$ (see above).
\end{rem}

\ps
\subsubsection{The type \texorpdfstring{$\Cn$, $n\geq 3$}{Cn}}
\DynkinCn

Suppose that $G = \Sp_{2n}$ and $\gg = \sP_{2n}$ where $n\geq 3$,
and let  $O$ be a non-minimal and non-trivial quasi-affine orbit. We have to show that $\dim O \geq 4n-4$. We have seen above that it suffices to consider those embeddings $\gg_{x}\subset\gg_{\lambda}$ where $\dim\nn_{\lambda} < 4n-6$. 

For the fundamental weights we get $\pp_{(\omega_j)} = \sL_{j}\oplus\sP_{2n-2j} \oplus \uu_{\omega_j}$.
An easy calculation shows that 
$$
\dim\uu_{\omega_{j}}=2j n + \frac{j(1-3j)}{2}.
$$ 
Thus $\dim\uu_{\omega_{j}} +2 \geq 4n-4$ except for $j=1$, and in this case we have $\dim\uu_{\omega_{1}}=2n-1$ and $\codim\pp_{(\omega_1)} = 2n$.
Thus, by by (\ref{U.eqn}), it suffices to look at the embedding $\gg_{x}\subset \pp_{(\omega_{1})} = \sP_{2n-2}\oplus \uu_{\omega_{1}}$. As a representation of $\sP_{2n-2}$ we get $\uu_{\omega_{1}}=V(\alpha_1) \oplus \K$, $V(\alpha_1) \simeq \K^{2n-2}$. 

Therefore,
if $\ll_{x} = \sP_{2n-2}$, then the codimension of $\gg_{x}$ in $\gg_{\omega_{1}}$ is $\geq \dim V(\alpha_1) = 2n-2$, and so $\dim O = \codim\gg_{x} \geq \codim \pp_{(\omega_{1})}+2(n-1) = 4n-2 > 4n-4$. 

If $\ll_{x}\subsetneqq \sP_{2n-2}$, then the codimension is at least $r_{\C_{n-1}} = 4n-8$, and so $\dim O \geq 2n+4n-8 = 6n-8>4n-4$.

\begin{rem} 
We have just shown above that any quasi-affine orbit of dimension $<4n-4$ is minimal. Furthermore, we have $\dim O_{\lambda}=\dim\uu_{\lambda}+1$ by (\ref{G-P.eqn}) and (\ref{P-lambda.subsec}). 
In particular, since $\dim O_{\omega_i}= 2j n+\frac{j(1-3j)}{2}$ (see above), we get $\dim O_{\omega_1}=2n$, and all other minimal orbits have dimension $\geq 4n-4$. 

Note that $r_{\Cn}=4n-4$ appears as dimension of an affine orbit as well as of a the minimal orbit $O_{\omega_2}$.
\end{rem}

\ps
\subsubsection{The type \texorpdfstring{$\Esix$}{E6}}
\DynkinEsix

Let $G$ be simply connected of type $\Esix$ and $\gg = \Lie G$, and  
let $O$ be a non-minimal and nontrivial quasi-affine orbit. We have to show that $\dim O \geq 26$. We have seen above that it suffices to consider those embeddings $\gg_{x}\into\pp_{(\lambda)}$ where $\dim\uu_{\lambda} < 24$. 
For the fundamental weights $\lambda$ we find
\begin{align*}
\pp_{(\omega_{1})}  &=  \so_{10}\oplus \uu_{\omega_{1}}, \ \dim \uu_{\omega_{1}}=16 = \dim\uu_{\omega_{6}},\\
\pp_{(\omega_{2})} &= \sL_{6}\oplus \uu_{\omega_{2}}, \ \dim \uu_{\omega_{2}}=21,\\
\pp_{(\omega_{3})} &= (\sL_{2}\oplus\sL_{5})\oplus \uu_{\omega_{3}}, \ \dim \uu_{\omega_{3}}=25 = \dim\uu_{\omega_{5}},\\
\pp_{(\omega_{4})} &= (\sL_{3}\oplus\sL_{2}\oplus\sL_{3})\oplus \uu_{\omega_{4}}, \ \dim \uu_{\omega_{4}}=29.
\end{align*}
Since $\dim \uu_{\omega_{1}+\omega_{2}}=\dim \uu_{\omega_{2}+\omega_{6}}=\frac{1}{2}(\dim \Esix-\dim \A_{4}-2)=26$ and $\dim \uu_{\omega_{1}+\omega_{6}}=\frac{1}{2}(\dim \Esix-\dim \Dfour-2)=24$, we have only to consider the cases  $\lambda\in\{\omega_{1},\omega_{2}, \omega_{6}\}$.
\ps
(1) We have
$\pp_{(\omega_{1})} = \so_{10}\oplus\uu_{\omega_{1}}$ and $\uu_{\omega_{1}}=V(\alpha_1)$ is the irreducible representation $V_{\omega_{4}}$ of $\so_{10}$ of dimensions $16$. Since $16>r_{\Dfive}=9$ we see that the codimension of $\gg_x$ in $\pp_{(\omega_1)}$ is at least 9. 
Thus  $\dim O =\codim\gg_{x} \geq \codim \pp_{(\omega_{1})}+9 = 17+9 = 26$.
\ps
(2) We have
$\pp_{(\omega_{2})} = \sL_{6}\oplus\uu_{\omega_{2}}$, $\uu_{\omega_2} = V(\alpha_2)\oplus \K$, and $V(\alpha_2)$ is the irreducible representation $V_{\omega_3}=\bigwedge^3\K^{6}$ of $\sL_{6}$ of dimension $20$. Since $20>r_{\SL_{6}}=10$ we see that the codimension of $\gg_{x}$ in $\pp_{(\omega_{2})}$ is at least $10$. Thus $\dim O =\codim\gg_{x} \geq \codim \pp_{(\omega_{2})}+10 = 22+10 = 32$.
\ps
(3) The case $\pp_{(\omega_{6})}$ is similar to $\pp_{(\omega_{1})}$ from (1).

\begin{rem} 
We have just shown that any quasi-affine orbit of dimension $<26$ is minimal. 
Furthermore, $\dim O_{\lambda}=\dim\uu_{\lambda}+1$ by (\ref{G-P.eqn}) and (\ref{P-lambda.subsec}).
From above we get
\[
\dim O_{\omega_1}=\dim O_{\omega_6}=17,\quad \dim O_{\omega_2}=22,\quad \dim O_{\omega_1+\omega_6}=25,
\]
and all other minimal orbits are of dimension $\geq26$ by (\ref{U.eqn}).
Moreover, $r_{\Esix}=26$ appears as dimension of an affine orbit as well as of the minimal orbits $O_{\omega_3}$ and $O_{\omega_5}$. 
\end{rem}

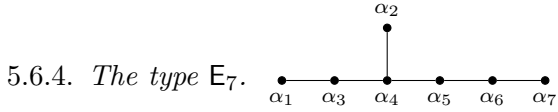
\ps
\subsubsection{The type \texorpdfstring{$\Eseven$}{E7}}
\DynkinEseven

Let $G$ be simply connected of type $\Eseven$ and $\gg = \Lie G$, and let
$O$ be a non-minimal and nontrivial quasi-affine orbit. We have to show that $\dim O \geq 45$. We have seen above that it suffices to consider those embeddings $\gg_{x}\subset\pp_{(\lambda)}$ where $\dim\uu_{\lambda} < 43$. 

If $\lambda$ is a  fundamental weight, then we find
\begin{align*}
\pp_{(\omega_{1})}  &=  \so_{12}\oplus \uu_{\omega_{1}}, \ \dim \uu_{\omega_{1}}=33,\\
\pp_{(\omega_{2})} &= \sL_{7}\oplus \uu_{\omega_{2}}, \ \dim \uu_{\omega_{2}}=42,\\
\pp_{(\omega_{3})} &= (\sL_{2}\oplus\sL_{6})\oplus \uu_{\omega_{3}}, \ \dim \uu_{\omega_{3}}=47,\\
\pp_{(\omega_{4})} &= (\sL_{3}\oplus\sL_{2}\oplus\sL_{4})\oplus \uu_{\omega_{4}}, \ \dim \uu_{\omega_{4}}=53,\\
\pp_{(\omega_{5})} &= (\sL_{5}\oplus\sL_{3})\oplus \uu_{\omega_{5}}, \ \dim \uu_{\omega_{5}}=50,\\
\pp_{(\omega_{6})} &= (\so_{10}\oplus\sL_{2})\oplus \uu_{\omega_{6}}, \ \dim \uu_{\omega_{6}}=42,\\
\pp_{(\omega_{7})} &= \Esix\oplus \uu_{\omega_{7}}, \ \dim \uu_{\omega_{7}}=27.
\end{align*}
Since $\dim \uu_{\omega_{1}+\omega_{2}}=\frac{1}{2}(\dim \Eseven-\dim \A_{5}-2)=48$, $\dim \uu_{\omega_{1}+\omega_{6}}=\frac{1}{2}(\dim \Eseven-\dim \Dfour -\dim \Aone-2)=50$, $\dim \uu_{\omega_{1}+\omega_{7}} = \frac{1}{2}(\dim \Eseven-\dim\Dfive-2)=43$, $\dim \uu_{\omega_{2}+\omega_{6}} = \frac{1}{2}(\dim \Eseven-\dim\A_{4}-\dim\Aone-2)=52$, 
$\dim \uu_{\omega_{2}+\omega_{7}} = \frac{1}{2}(\dim \Eseven-\dim\A_{5}-2)=48$, $\dim \uu_{\omega_{6}+\omega_{7}} = \frac{1}{2}(\dim \Eseven-\dim\Dfive-2)=43$,
we have only to consider the cases  $\lambda\in\{\omega_{1},\omega_{2}, \omega_{6},\omega_{7}\}$.
\ps
(1) We have
$\pp_{(\omega_{1})} = \so_{12}\oplus\uu_{\omega_{1}}$, $\uu_{\omega_1} = V(\alpha_1)\oplus\K$, and $V(\alpha_{1})$ is the irreducible representation $V_{\omega_{5}}$ of $\so_{12}$ of dimensions $32>r_{\Dsix}=11$.
Thus the codimension of $\gg_{x}$ in $\pp_{(\omega_{1})}$ is at least $11$, and so $\dim O =\codim\gg_{x} \geq \codim \pp_{(\omega_{1})}+11 = 34+11 = 45$.
Moreover, the subalgebra $\hh:=\so_{11}\oplus\uu_{\omega_1}\subset \gg$ is the Lie algebra of a subgroup $H$ of codimension 34+11=45 which has no characters. By Remark~\ref{quasiaffine_orbits.rem} we see that $G/H$ is a quasi-affine orbit of dimension 45, and so $d_\Eseven \leq 45$.
\ps
(2) We have
$\pp_{(\omega_{2})} = \sL_{7}\oplus\uu_{\omega_{2}}$, $\uu_{\omega_2} = V(\alpha_2)\oplus \K^7$,  and $V(\alpha_{2})$ is the irreducible representation $V_{\omega_4}=\bigwedge^{4}\K^{7}$ of $\sL_{7}$ of dimensions $35>r_{\SL_{7}}=12$. Thus the codimension of $\gg_{x}$ in $\pp_{(\omega_{2})}$ is at least $12$, and so $\dim O =\codim\gg_{x} \geq \codim \pp_{(\omega_{2})}+12 = 43+12 = 55>45$.
\ps
(3) We have
$\pp_{(\omega_{6})} = (\so_{10}\oplus\sL_{2})\oplus\uu_{\omega_{6}}$, $\uu_{\omega_6} = V(\alpha_6)\oplus \K^{10}$, and $V(\alpha_{6})$ is the irreducible representation $V_{\omega_{5}}\otimes \K^{2}$ of $\so_{10}\oplus\sL_{2}$ of dimensions $2\cdot 16=32>r_{\Dfive\times\Aone}=2$. Thus the codimension of $\gg_{x}$ in $\pp_{(\omega_{6})}$ is at least $2$, and so $\dim O =\codim\gg_{x} \geq \codim \pp_{(\omega_{6})}+2 = 43+2 = 45$.
\ps
(4) We have
$\pp_{(\omega_{7})}=\Esix \oplus \uu_{\omega_{7}}$, and $V(\alpha_{7}) = \uu_{\omega_{7}}$ is the irreducible representation $V_{\omega_{6}}$ of $\Esix$ of dimension $27>r_{\Esix}=26$. Thus the codimension of $\gg_{x}$ in $\pp_{(\omega_{7})}$ is at least $26$, and so $\dim O =\codim\gg_{x} \geq \codim \pp_{(\omega_{7})}+26 = 28+26 = 54>45$.

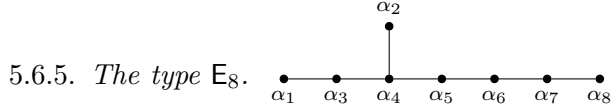
\ps
\subsubsection{The type \texorpdfstring{$\Eeight$}{E8}}
\DynkinEeight

Let $G$ be simply connected of type $\Eeight$ and $\gg = \Lie G$, and let
$O$ be a non-minimal and nontrivial quasi-affine orbit. We have to show that $\dim O \geq 86$. We have seen above that it suffices to consider those embeddings $\gg_{x}\subset\pp_{(\lambda)}$ where $\dim\uu_{\lambda} < 84$. 

If $\lambda$ is a  fundamental weight, then we find
\begin{align*}
\pp_{(\omega_{1})}  &=  \so_{14}\oplus \uu_{\omega_{1}}, \ \dim \uu_{\omega_{1}}=78,\\
\pp_{(\omega_{2})} &= \sL_{8}\oplus \uu_{\omega_{2}}, \ \dim \uu_{\omega_{2}}=92,\\
\pp_{(\omega_{3})} &= (\sL_{2}\oplus\sL_{7})\oplus \uu_{\omega_{3}}, \ \dim \uu_{\omega_{3}}=98,\\
\pp_{(\omega_{4})} &= (\sL_{3}\oplus\sL_{2}\oplus\sL_{5})\oplus \uu_{\omega_{4}}, \ \dim \uu_{\omega_{4}}=106,\\
\pp_{(\omega_{5})} &= (\sL_{5}\oplus\sL_{4})\oplus \uu_{\omega_{5}}, \ \dim \uu_{\omega_{5}}=104,\\
\pp_{(\omega_{6})} &= (\so_{10}\oplus\sL_{3})\oplus \uu_{\omega_{6}}, \ \dim \uu_{\omega_{6}}=97,\\
\pp_{(\omega_{7})} &= (\Esix\oplus \sL_{2}) \oplus \uu_{\omega_{7}}, \ \dim \uu_{\omega_{7}}=83,\\
\pp_{(\omega_{8})} &= \Eseven \oplus \uu_{\omega_{8}}, \ \dim \uu_{\omega_{8}}=57.
\end{align*}
Since $\dim \uu_{\omega_{1}+\omega_{7}}=\frac{1}{2}(\dim \Eeight-\dim \Dfive-\dim \Aone-2)=99$, $\dim \uu_{\omega_{1}+\omega_{8}}=\frac{1}{2}(\dim \Eeight-\dim \Dsix -2)=90$, and $\dim \uu_{\omega_{7}+\omega_{8}} = \frac{1}{2}(\dim \Eeight-\dim\Esix-2)=84$,
we have only to consider the cases  $\lambda\in\{\omega_{1},\omega_{7}, \omega_{8}\}$.
\ps
(1) We have
$\pp_{(\omega_{1})} = \so_{14}\oplus\uu_{\omega_{1}}$, $\uu_{\omega_1} = V(\alpha_1)\oplus \K^{14}$, and $V(\alpha_{1})$ is the irreducible representation $V_{\omega_{7}}$ of $\so_{14}$ of dimensions $64>r_{\Dseven}=13$.
Thus the codimension of $\gg_{x}$ in $\pp_{(\omega_{1})}$ is at least $13$, and so $\dim O =\codim\gg_{x} \geq \codim \pp_{(\omega_{1})}+13 = 79+13 = 92>86$.
\ps
(2) We have
$\pp_{(\omega_{7})} = (\Esix\oplus\sL_{2})\oplus\uu_{\omega_{7}}$, and $V(\alpha_{7})\subset \uu_{\omega_{7}}$ is the irreducible representation $V_{\omega_{6}}\otimes \K^{2}$ of $\Esix\oplus\sL_{2}$ of dimension $2\times 27=54>r_{\Esix\times\Aone}=2$. Thus the codimension of $\gg_{x}$ in $\pp_{(\omega_{7})}$ is at least $2$, and so $\dim O =\codim\gg_{x} \geq \codim \pp_{(\omega_{7})}+2 = 84+2 = 86$.
\ps
(3) We have
$\pp_{(\omega_{8})}=\Eseven \oplus \uu_{\omega_{8}}$, $\uu_{\omega_8} = V(\alpha_8) \oplus \K$, and $V(\alpha_{8})$ is the irreducible representation $V_{\omega_{7}}$ of $\Eseven$ of dimension $56>r_{\Eseven}=54$. Thus the codimension of $\gg_{x}$ in $\pp_{(\omega_{8})}$ is at least $54$, and so $\dim O =\codim\gg_{x} \geq \codim \pp_{(\omega_{8})}+54 = 58+54 = 112 > 86$.
\ps
(4) The subalgebra $\pp_{(\omega_7+\omega_8)} = \Esix\oplus\uu_{\omega_7+\omega_8}$ corresponds to a closed subgroup $H\subset G$ of codimension $84+2=86$ which has no characters. Thus, by Remark~\ref{quasiaffine_orbits.rem}, $G/H$ is a quasi-affine orbit of dimension 86, and so $d_\Eeight \leq 86$.

\ps

\subsection{Computations with LiE}
For our explicit computations we used the computer algebra program LiE \cite{LeCoLi1992LiE:-A-package-for}.
The maximal parabolic Lie subalgebra $\pp_i=\pp_{\omega_i} \subset \gg$ and its Levi decomposition $\pp_i = \ll_i\oplus\uu_i$  have the following description in terms of root spaces.
\be
\item 
The set $\{\alpha_j\mid j\neq i\}$ is a set of simple roots for the Levi part $\ll_i$. Hence the Dynkin diagram of $\ss_i:=[\ll_i,\ll_i]$ is obtained by removing the node $i$ from the Dynkin diagram of $\gg$.
\item 
The center $\zz_i \subset \ll_i$ is one-dimensional and so 
$\dim \gg = (\dim \ss_i + 1) + 2\dim \uu_i$.
\item
The nilradical $\uu_i$ is the direct sum of those root spaces $g_{\beta}$ where $\beta$ is a positive root not containing $\alpha_i$, i.e. $\alpha_i$ does not occur in the presentation of $\beta$ as a linear combination of simple roots.
\item 
Denote by $\nn_{\ll_i}^\pm \subset \ll_i$ the sum of the positive resp. negative root spaces. Then $\gg_{\alpha_i}$ is annihilated by $\nn_{\ll_i}^-$, because $\alpha_i - \beta$ is never a root for a positive root $\beta$ of $\ll_i$, by (3).
\item
The irreducible $\ss_i$-submodule $V(\alpha_i) \subset \gg$ generated by $\gg_{\alpha_i}$ has $\gg_{\alpha_i}$ as lowest weight space, by (4). Denoting by $\lambda_i$ the weight $\alpha_i|_{\ss_i}$ of $\ss_i$, then $V(\alpha_i) \simeq V_{-\lambda_i}^\perp$. In fact, the dual module $V(\alpha_i)^\perp$ has highest weight $-\lambda_i$.
\item
The submodule $V(\alpha_i)$ consists of those root spaces $\gg_\beta$ where $\beta$ is a positive root containing $\alpha_i$ with multiplicity 1.
\item
If the center $\zz_i$ has weight $\kappa$ on $\gg_{\alpha_i}$, then $V(\alpha_i)$ is the $\kappa$-eigenspace of $\zz_i$. This follows from (6), because $\zz_i$ has weight 0 on all of $\ll_i$.
\ee
These remarks offer several possibilities to calculate the dimension of $V(\alpha_i)$ using the computer algebra program LiE \cite{LeCoLi1992LiE:-A-package-for}.

\subsubsection{Computation of  the \texorpdfstring{$\kappa$}{k}-eigenspace}
The rows of the Cartan matrix \verb+Cartan(G)+ are the coordinates of the simple roots $\alpha_j$ in the basis of the fundamental weights $\omega_k$. It follows that the $i$th column of the inverse of the Cartan matrix \verb+i_Cartan(G)+ defines the homomorphism $\K \to \hh=\K^r$ with image $\zz_i$. (In LiE \verb+i_Cartan(G)+ is the inverse of \verb+Cartan(G)+ multiplied by the determinant in order to have integral entries.) Now one uses the function \verb+spectrum(adjoint(G),t,G)+ which calculates the dimensions of the eigenspaces of the toral element \verb+t+.
\begin{exa} We compute $\dim V(\alpha_7)$ for $G = \Eeight$.
\begin{verbatim}

> Cartan(E8)
     [[ 2, 0,-1, 0, 0, 0, 0, 0]
     ,[ 0, 2, 0,-1, 0, 0, 0, 0]
     ,[-1, 0, 2,-1, 0, 0, 0, 0]
     ,[ 0,-1,-1, 2,-1, 0, 0, 0]
     ,[ 0, 0, 0,-1, 2,-1, 0, 0]
     ,[ 0, 0, 0, 0,-1, 2,-1, 0]
     ,[ 0, 0, 0, 0, 0,-1, 2,-1]
     ,[ 0, 0, 0, 0, 0, 0,-1, 2]
     ]
     
> det_Cartan(E8)
     1
     
> Cartan_type(id(8)-7,E8)   # Cartan type of the Levi L(7) of P(7) #
     E6A1T1

> i_Cartan(E8)
     [[ 4, 5, 7,10, 8, 6, 4,2]
     ,[ 5, 8,10,15,12, 9, 6,3]
     ,[ 7,10,14,20,16,12, 8,4]
     ,[10,15,20,30,24,18,12,6]
     ,[ 8,12,16,24,20,15,10,5]
     ,[ 6, 9,12,18,15,12, 8,4]
     ,[ 4, 6, 8,12,10, 8, 6,3]
     ,[ 2, 3, 4, 6, 5, 4, 3,2]
     ]
     
> (*i_Cartan(E8))[7]    # 7th column of i_Cartan(E8) #
     [4,6,8,12,10,8,6,3]

> spectrum(adjoint(E8),[4,6,8,12,10,8,6,3,29],E8)
     82X[ 0] +54X[ 1] +27X[ 2] + 2X[ 3] + 2X[26] +27X[27] +54X[28]
     
\end{verbatim}
\noindent
Thus we have $\ss_{\omega_7}=\Esix\oplus \sL_2$, and the last output
shows that $\dim \ll_{\omega_7} = 82$ and $\dim V(\alpha_7) = 54$. In addition, $\uu_{\omega_7}$ contains two other irreducible representations, of dimensions $27$ and $2$.
\end{exa}

\subsubsection{Computation of the lowest weight \texorpdfstring{$\lambda_i:=\alpha_i|_{\ss_i}$}{lambda}}
The lowest weight $\lambda_i$ is obtained from the $i$th row of the Cartan matrix \verb+Cartan(G)+ by removing the $i$th entry (which is a $2$) and and rearranging the other entries according to the numeration of the nodes of the Dynkin diagram of  $\ss_i$ obtained by removing the $i$th node from the Dynkin diagram of $G$.

\begin{exa} Lowest weight of $V(\alpha_1)$ for $G = \Eeight$. 

The first row of \verb+Cartan(E8)+ is \verb+[2,0,-1,0,0,0,0,0]+, and $\ss_{\omega_1} = \mathsf{D}_{7}$. In the Bourbaki-numeration of the nodes of $\Dn$ this gives the weight \verb+[0,0,0,0,0,0,-1,0]+ of $\mathsf{D}_{7}$ which is $-\omega_6$. Thus $V(\alpha_1)$ is $V_{\omega_6}^\perp$ which is $V_{\omega_7}$ and has dimension $64$:
\begin{verbatim}

> contragr([0,0,0,0,0,1,0],D7)
     [0,0,0,0,0,0,1]
> dim([0,0,0,0,0,0,1],D7)
     64

\end{verbatim}
The same computation for $V(\alpha_7)$ where $\ss_{\omega_7} = \Esix\oplus\sL_2$,  using the $7$th row of \verb+Cartan(E8)+ which is \verb+[0,0,0,0,0,-1,2,-1]+, gives the lowest weight \verb+[0,0,0,0,0,-1,-1]+ which is $-\omega_6-\omega_1'$. Thus $V(\alpha_7) = (V_{\omega_6} \otimes \K^2)^\perp=V_{\omega_1}\otimes\K^2$ and has dimension $27 \times 2 = 54$:
\begin{verbatim}

> contragr([0,0,0,0,0,1,1],E6A1)
     [1,0,0,0,0,0,1]
> dim([1,0,0,0,0,0,1],E6A1)
     54

\end{verbatim}
\end{exa}
\subsubsection{Computation of the decomposition of \texorpdfstring{$\gg$}{g} under \texorpdfstring{$\ll_i$}{l}} The program LiE offers a {\it branch function} which calculated the decomposition of a representation of $G$ under a closed connected subgroup $H \subset G$: \verb+branch(repr,H,res_mat,G)+. It needs the restriction matrix \verb+res_mat+ defining the map $\Lambda_G \to \Lambda_H$ between the weight lattices which is given by the inclusion $T_H \subseteq T_G$ of maximal tori.

For a maximal parabolic subgroup $P_{\omega_i}$ it is easy to compute the restriction matrix from $G$ to $L_{\omega_i}$ by using the LiE-function \verb+res_mat(fundam(id(rank G)-i))+. It follows that the decomposition of the adjoint representation of $G$ under $L_{\omega_i}$ is given by 
\begin{verbatim}

> branch(adjoint(G),Levi(i),res_mat(fundam(id(rank G)-i)),G)

\end{verbatim}
\begin{exa}
We compute the decomposition of $\Lie \Eeight$ under the Levi subgroup $L_{\omega_7}$ of $P_{\omega_7}$.
\begin{verbatim}

> h=Cartan_type(id(8)-7,E8); h     #  Levi subgroup L(7) of P(7)  #
     E6A1T1
     
> m=res_mat(fundam(id(8)-7,E8),E8); m   # restriction matrix from E8 to L(7) #
     [[0,0,0,0,0,1,0, 4]
     ,[0,1,0,0,0,0,0, 6]
     ,[0,0,0,0,1,0,0, 8]
     ,[0,0,0,1,0,0,0,12]
     ,[0,0,1,0,0,0,0,10]
     ,[1,0,0,0,0,0,0, 8]
     ,[0,0,0,0,0,0,0, 6]
     ,[0,0,0,0,0,0,1, 3]
     ]

> branch(adjoint(E8),h,m,E8)
     1X[0,0,0,0,0,0,0, 0] +1X[0,0,0,0,0,0,1,-3] +1X[0,0,0,0,0,0,1, 3] +
     1X[0,0,0,0,0,0,2, 0] +1X[0,0,0,0,0,1,0,-2] +1X[0,0,0,0,0,1,1, 1] +
     1X[0,1,0,0,0,0,0, 0] +1X[1,0,0,0,0,0,0, 2] +1X[1,0,0,0,0,0,1,-1]

\end{verbatim}
The weight spaces are given by the last entry. Hence we see that 
the $0$-weight space is $\Lie L_{\omega_7}=\Lie \Esix \oplus \sL_2\oplus \K$, the $1$-weight space is $V_{\omega_6}\otimes \K^2=V(\alpha_7)$, the $2$-weight space is $V_{\omega_1}$ of $\Esix$, and the $3$-weight space is the standard representation $\K^2$ of $\sL_2$. Thus we get 
$$
\uu_{\omega_7} = V_{\omega_6}\otimes \K^2 \oplus V_{\omega_1} \oplus \K^2
$$
as a representation of $\Esix\times\SL_2$.
\end{exa} 

\renewcommand{\MR}{}
\bibliography{Bib-KrReZi}
\bibliographystyle{amsalpha}

\end{document}